\documentclass[3p]{elsarticle}
\usepackage{amsthm,amsmath,amsfonts,amssymb,mathrsfs,url,bbm,color,upgreek}
\RequirePackage[colorlinks,linkcolor=blue,citecolor=blue,urlcolor=blue]{hyperref}
\usepackage{subcaption,caption}
\usepackage{fancyhdr}
\allowdisplaybreaks
\newtheorem{theorem}{Theorem}[section]
\newtheorem{example}{Example}[section]
\newtheorem{remark}{Remark}[section]
\newtheorem{as}{Assumption --}[section]
\newtheorem{lemma}{Lemma}[section]
\newtheorem{corollary}{Corollary}[section]

\newcommand\imCMsym[4][\mathord]{%
	\DeclareFontFamily{U} {#2}{}
	\DeclareFontShape{U}{#2}{m}{n}{
		<-6> #25
		<6-7> #26
		<7-8> #27
		<8-9> #28
		<9-10> #29
		<10-12> #210
		<12-> #212}{}
	\DeclareSymbolFont{CM#2} {U} {#2}{m}{n}
	\DeclareMathSymbol{#4}{#1}{CM#2}{#3}
}

\imCMsym{cmmi}{124}{\CMjmath}

\makeatletter
\def\ps@pprintTitle{%
	\let\@oddhead\@empty
	\let\@evenhead\@empty
	\let\@oddfoot\@empty
	\let\@evenfoot\@empty}
\makeatother

\usepackage{graphicx,float,placeins}
\begin{document}
	\fancyhead[L]{\it{D. Vashistha and C. Kumar}}
	\fancyhead[R]{\it{First- and Half-order schemes for SDEwMS}}
	\title{First- and Half-order Schemes for Regime Switching Stochastic Differential Equation with Non-differentiable Drift Coefficient}
	\author{Divyanshu Vashistha}
	\ead{dvashistha@ma.iitr.ac.in}
	\author{Chaman Kumar\corref{cor1}}
	\ead{chaman.kumar@ma.iitr.ac.in}
	\cortext[cor1]{Corresponding author}
	\address{Indian Institute of Technology Roorkee, Haridwar, 247667, India}
	
	\begin{abstract}
	An explicit first-order drift-randomized Milstein scheme for a regime switching stochastic differential equation is proposed and its bi-stability and rate of strong convergence are investigated for a non-differentiable drift coefficient.   
	Precisely, drift is Lipschitz continuous while diffusion along with its derivative is Lipschitz continuous.
	Further, we explore the significance of evaluating Brownian trajectories at  every switching time of the underlying Markov chain in achieving the convergence rate $1.0$ of the proposed scheme. 
	In this context, possible variants of the scheme, namely modified randomized and reduced randomized schemes, are considered and their convergence rates are shown to be $1/2$. 
	Numerical experiments are performed to illustrate the convergence rates of these schemes along with their corresponding non-randomized versions.
	Further, it is illustrated that the half-order non-randomized reduced and modified schemes outperforms the classical Euler scheme.
	\end{abstract}
	
	
	\maketitle 
	
\section{{\textbf{Introduction}}}
Let  $(\tilde \Omega,  \tilde{\mathscr{F}}, \tilde{\mathbb{P}})$ be a complete probability space.  
Consider a $\tilde d -$dimensional  standard Brownian motion $B:=\{B(t)\}_{t\geq 0}$ with natural filtration $\tilde{\mathbb{F}}^{B}:=\{\tilde{\mathscr F}_t^{B}\}_{t\geq 0}$.  
Further,   assume that $r:=\{r(t)\}_{t\geq 0}$ is a Markov chain with finite state space $S:=\{1,\ldots, m'\}$ and generator $\mathcal{Q}:=(\mathfrak{q}_{j_0 k_0})_{j_0,k_0 \in S}$ where $\mathfrak{q}_{j_0 k_0}\geq 0$ for $j_0 \neq k_0$ and $\mathfrak{q}_{j_0 j_0}=-\displaystyle \sum_{k_0 \neq j_0} \mathfrak{q}_{j_0 k_0}$ which implies that the transition probability matrix of $r$ is given by,  
\begin{align*}
	\tilde{\mathbb{P}}(r(t+\Delta)=k_0 | r(t)=j_0)=\begin{cases}
		\mathfrak{q}_{j_0 k_0} \Delta +\mbox{o}(\Delta), & k_0\neq j_0, 
		\\
		1+\mathfrak{q}_{j_0 j_0}\Delta +\mbox{o}(\Delta),  & k_0=j_0,
	\end{cases}
\end{align*}
for any $t \geq 0$, $j_0, k_0 \in S$ and $\Delta>0$. 
The natural filtration of $r$ is denoted by $\tilde{\mathbb{F}}^r := \{\tilde{\mathscr F}_t^r\}_{t\geq 0}$. 
Also, consider an $\tilde{\mathscr F}_0^B$-measurable random variable $X_0$ and assume that   $r$, $B$ and $X_0$ are  independent.  
Define $\tilde {\mathbb{F}}:=\{\tilde{\mathscr F}_t\}_{t \geq 0}$ where $\tilde{\mathscr F}_t:=\tilde{\mathscr F}_t^{B} \vee \tilde{\mathscr F}_t^r $ for $t \geq 0$. 
For a fixed  $T>0$, 
let  $b:[0, T]\times \mathbb{R}^d \times S \mapsto \mathbb{R}^d$ and $\sigma: [0, T]\times \mathbb{R}^d \times S \mapsto \mathbb{R}^{d \times \tilde{d}}$ be Borel measurable functions. 

Consider the following $d-$dimensional  regime switching stochastic differential equation, also referred to as stochastic differential equation with Markovian switching (SDEwMS),  
\begin{align}
	X(t) = X_0+ \int_0^t b(s, X(s), r(s)) ds +\sum_{\ell=1}^{\tilde{d}} \int_0^t \sigma_{\ell}(s, X(s), r(s)) dB_{\ell}(s) \label{eq:sdems}
\end{align}
almost surely for any $t \in [0,T]$ with initial values $X(0)=X_0$ and $r(0) = i_0 \in S$. 
Define~$X:=\{X(t)\}_{t\geq 0}$. 

In comparison to the standard SDE, SDEwMS contains a Markov chain $r$ in its coefficients where each state of $r$ signifies a regime and in each regime, one solves an SDE.
For example, the hybrid geometric Brownian motion governed by the following equation, 
$$
d\tilde S(t)= \mu(r(t)) \tilde S(t)dt + \sigma(r(t)) \tilde S(t) dB(t),  
$$
with initial values $\tilde S(0)=\tilde{S}_0$ and $ r(0)= i_0 \in S$, where $\mu$ and $\sigma$ denote the return and the volatility of the stock price which vary with the state of the Markov chain $r$, see \cite[Chapter 10]{mao2006} and \cite[Chapter 11]{yin2010}.
In other words,  the varying return (interest rate) in different regimes are captured by different states of $r$ by introducing the Markov chain in the coefficients of the SDE. 
 Let us consider another example namely, the mean-reverting process with Markovian switching, used for modeling stochastic volatility or asset price or interest rate, is given by
\begin{align*}
d\tilde S(t) = \lambda(r(t))[\mu(r(t)) - \tilde S(t)]dt + \sigma(r(t))\tilde{S}(t)dB(t)
\end{align*}
with initial values $\tilde{S}(0) = \tilde{S}_0 >0$ and $r(0)=i_{0} \in S$ where $\lambda(i),\mu(i)$ and $\sigma(i)$ are positive constants for $i \in S$, see \cite[Chapter 10]{mao2006}. 
The work presented in this manuscript covers both the examples as given above.
Thus, SDEwMS is considered to be a more appropriate method of modelling such complex systems. For a detailed exposition of SDEswMS and their applications,  we refer the interested reader to \cite{mao2006,yin2010} and the references therein. It is worth mentioning that SDEwMS reduce to a standard SDE when the state space $S$ of the chain $r$ is a singleton set. An exposition on SDEs and their numerical approximation can be found in \cite{kloeden1992} and \cite{milstein2021stochastic}. 
Further, tamed schemes for SDE with super-linear coefficients are developed in \cite{martin2012,hutzenthaler2015numerical}, see also \cite{bao2021first,beyn2016stochastic,beyn2017stochastic,dos2022simulation,guo2017partially,kelly2018adaptive,kelly2022adaptive,kelly2023strong,kumar2021explicit,kumar2022well,kumar2017tamed,kumar2019milstein,lord2024convergence,mao2015truncated,mao2016convergence,neelima2020,sabanis2013note,sabanis2016,tretyakov2013fundamental} and the references therein.

Similar to the case of SDEs,  explicit solutions of SDEwMS are often either unavailable or tedious to find. 
Thus,  the development of efficient numerical schemes becomes crucial for obtaining their approximate solutions. 
The Euler scheme with a half rate of strong convergence has been studied in \cite{dareiotis2016tamed, kumar2017explicit,mao2007approximations, nguyen2012, nguyen2018,  nguyen2019,yuan2004convergence} and the references therein. 
For the first and higher-order schemes, the classical technique of applying  It\^o's formula on the coefficients, used in the context of standard SDEs,  does not work here as there is no notion of derivative of a Markov chain dependent function.
Moreover,  the mixing of discontinuity of the Markov chain $r$ with the continuous dynamics of the solution process $X$ pose additional difficulties.  
These challenges have been resolved in \cite[Lemma 2.2]{nguyen2017milstein} by developing an  It\^o's formula for SDEwMS with the help of a martingale associated with the chain $r$. 
This martingale captures the number of switchings of $r$ in a given interval. 
Furthermore, a Milstein-type scheme for SDEwMS is developed with the help of this new It\^o's formula and its rate of strong convergence is shown to be $1.0$ under the Lipschitz continuous coefficients. 
In comparison with the Milstein scheme for standard SDEs, the first-order scheme of SDEwMS consists of an additional term which involves the Brownian increments $\big\{ B_{\ell}(t_{j})-B_{\ell}(\uptau_{1}^{j-1})\big\}_{\ell\in \{1, \ldots, \tilde{d}\}}$ where $0=t_0 < t_1 < \cdots < t_{n_h}=T$ is a partition of the interval $[0,T]$ and $\tau_1^{j-1}$ is the first switching time of the chain $r$ in the time interval $(t_{j-1},t_j]$ for $j \in {1, \ldots, n_h}$, see also Equation \eqref{eq:scheme} below. 
The Milstein-type schemes for SDEwMS are further investigated in \cite{kumar2020explicit, kumar2021note} without assuming second-order derivatives on the drift and diffusion coefficients. 
Further, a stochastic It\^o-Taylor expansion for SDEwMS and its application in developing and analysing general order scheme is investigated in \cite{kumar2022}. 
In \cite{vashistha2024}, authors extend the results of \cite{nguyen2017milstein} to SDEwMS driven by L\'evy noise by proving an It\^o's formula which is then used to derive and study the first order Milstein-type scheme under the assumptions that coefficients and their derivatives satisfy Lipschitz condition.
Moreover, \cite{kelly2024adaptive} explored adaptive numerical methods for SDEwMS. Unlike adaptive strategies that requires careful mesh refinement including the switching times of the Markov chain, the randomized Milstein scheme offers a simpler implementation while capturing the effects of regime-switching. Thus, the first-order schemes for SDEwMS available in the literature, assume coefficients to have at least first-order derivatives whereas randomized approach allows for non-differentiable drift coefficient.

In this manuscript, we develop an explicit first-order numerical scheme for SDEwMS, assuming that the drift coefficient satisfies the Lipschitz continuity but is not differentiable. 
To achieve this, we employ the drift randomization technique developed in \cite{Kruse2019} for standard SDEs and propose a randomized Milstein scheme for SDEwMS (see Equations \eqref{eq:euler} and \eqref{eq:scheme} below). 
Indeed, the drift coefficient is fully randomized with respect to all the variables, including the Markov chain $r$. 
The diffusion coefficient and its derivatives are assumed to be Lipschitz continuous.  
The rate of strong convergence of the proposed scheme is shown to be $1.0$ under the above set-up. 
Additionally, the mean-square error of the randomized Milstein scheme is numerically shown to be less than that of the Milstein scheme for SDEwMS of \cite[Equation (4.20)]{nguyen2017milstein} but at a higher computational cost (see Figure \ref{fig:com:error},\ref{fig:com:cpu} and Table \ref{tab:com}).
Thus, the randomized Milstein scheme can be used for drift coefficients that are non-differentiable yet Lipschitz continuous, while the classical Milstein scheme is preferable for  suffieciently more regular drift coefficients, as developed in \cite{kumar2021note,nguyen2017milstein}.

Further, the above observations coincide with that of \cite{Kruse2019} when the state space of the Markov chain $r$ is a singleton set.
Moreover, if there is no randomization  (\textit{i.e.}, when $u_j  \equiv 0$ for all $j \in \mathbb{N}$ in Equation \eqref{eq:euler}), the proposed scheme coincides with the Milstein scheme for SDEwMS developed in \cite[Equation (4.20)]{nguyen2017milstein}. There are other works on randomization such as \cite{ sani2024,jentzen2009random, morkisz2021randomized, przybylowicz2024randomized, przybylowicz2024approximation} and the references therein, but they do not consider the regime switching SDEs. 
To the best of our knowledge, this is the first result on the first-order scheme for SDEwMS with Lipschitz continuous and non-differentiable drift coefficient.  
\newline 
\textit{Our Achievements:} 
Below,  we list the achievements of this manuscript.  
\begin{itemize}
	\item 
	We propose a first-order randomized Milstein scheme for SDEwMS and establish its bi-stability (see Theorem \ref{thm:bi}) and rate of strong convergence (see Theorem~\ref{thm:rate}) under a much weaker assumption on the drift coefficient as compared to the existing results (see Assumption \ref{as:drift}).
	Precisely,  the drift coefficient is Lipschitz continuous and non-differentiable. The diffusion coefficient along with its derivatives is Lipschitz continuous.   
	\item
	It is worth mentioning that the Brownian trajectory on every switching time of the Markov chain $r$ plays a pivotal role in recovering the rate $1.0$ of  mean-square convergence of the randomized Milstein scheme. 
	To demonstrate this,  we explore the possible variations of the proposed scheme by altering the key term associated with the Brownian increments $\big\{ B_{\ell}(t_{j})-B_{\ell}(\uptau_{1}^{j-1})\big\}_{\ell\in \{1, \ldots, \tilde{d}\}}$ as mentioned below. 
	\newline
	\indent $\ast$ In the first variation, referred as the \textit{modified randomized scheme} in this manuscript, the original increment $\big\{ B_{\ell}(t_{j})-B_{\ell}(\uptau_{1}^{j-1})\big\}_{\ell\in \{1, \ldots, \tilde{d}\}}$ is replaced by simple Brownian increments $\big\{ B_{\ell}(t_{j})-B_{\ell}(t_{j-1})\big\}_{\ell\in \{1, \ldots, \tilde{d}\}}$ for $j \in \{1, \ldots, n_h\}$. 
	\newline
	\indent $\ast$ The second variation, named as the \textit{reduced randomized scheme},  is derived by completely removing the term involving $\big\{ B_{\ell}(t_{j})-B_{\ell}(\uptau_{1}^{j-1})\big\}_{\ell\in \{1, \ldots, \tilde{d}\}}$ from the scheme.
	
	The impact of these variations is on the rate of convergence which drops to $1/2$ from~$1$ (see Theorems \ref{thm:rate:wsi} and \ref{thm:rate:ri}).
	\item
	Having observed the rate $1/2$ for the modified and reduced randomized schemes, it is natural to compare them with the half-order randomized Euler scheme as well as with the non-randomized versions of these schemes.
	The numerical studies indicate that the modified and reduced non-randomized schemes have less mean-squared error than that of the classical Euler schemes (see Figures \ref{fig:inc:swi:error} and Table \ref{tab:comphalf}).
	 The modified and reduced randomized schemes has less mean-square error but at a higher computational cost (see Figures \ref{fig:inc:swi:error},\ref{fig:inc:swi:cpu} and Table \ref{tab:comphalf}). Thus, the classical Euler scheme should be used when drift and diffusion coefficients are Lipschitz continuous. On the other hand, the modified and reduced non-randomized schemes can be used when additionally the diffussion coefficient is assumed to be differentiable (see Examples \ref{ex:mrp} and \ref{ex:bs}).
\end{itemize}

The rest of the manuscript is organized as follows. In Section \ref{sec:mainresult}, assumptions under which we perform our analysis are stated along with the statements of the main results. Section \ref{sec:Rate} contains the proofs of all the essential lemmas and the  main results. 
Finally, the outcomes of the numerical experiments are reported in Section \ref{sec:numerical}.

\subsection{General Notation}
$A_{l}$ denotes the $l$-th element of $A$ if it is a vector and the $l$-th column if it is a matrix. 
$AB$ is  matrix multiplication of two compatible  matrices $A$ and $B$. 
$|\cdot|$ is used for the Euclidean norm. 
$L^q(\Omega)$ denotes the class of random variables having finite $q$-th moment. 
For a function $ f: [0,T] \times \mathbb{R}^d \times S \to \mathbb{R}^d$,  $\mathcal{D}_{x} f$ denotes the gradient of $f$ with respect to the second variable. Moreover, $\mathbbm{1}_{A}$ is indicator function of a set~$A$. 
Throughout, $\mathbb{E}[X]$ signifies the expectation of a random variable $X$. 
Further, $K$ stands for a generic positive constant with values varying from place to place.

\section{Main Results}
\label{sec:mainresult} 
We make the following assumptions. 
\begin{as} \label{as:initial}
	$\mathbbm{E}[|X_0|^{p}]<\infty$ for some $p \geq 2$. 
\end{as}
\begin{as} \label{as:drift}
	There is a constant  $C>0$  such that,  
	for all $t, s \in [0,T]$,  $x,\bar{x}\in \mathbb{R}^d$ and $\imath \in S$,   
	\begin{align*}
		|b(t,x,\imath)-b(s, \bar{x}, \imath)| & \leq C \{ (1+\max(|x|,|\bar{x}|))|t-s|^{1/2} + |x-\bar{x}|\} .
	\end{align*}
\end{as}
\begin{as} \label{as:diffusion}
	There is a constant $C>0$ such that, for all $t,s \in [0,T]$, $x,\bar{x} \in \mathbb{R}^d$ and~$\imath \in S$, 
	\begin{align*}
		\sum_{\ell=1}^{\tilde{d}}|\sigma_{\ell}(t,x,\imath)-\sigma_{\ell}(s, \bar{x}, \imath)| \vee \sum_{\ell=1}^{\tilde{d}} |\mathcal{D}_x \sigma_{\ell}(t, x, \imath) - \mathcal{D}_x \sigma_{\ell}(s, \bar{x}, \imath)| & \leq C \{ (1+\max(|x|,|\bar{x}|))|t-s| + |x-\bar{x}|\}.
	\end{align*}
\end{as} 
The proof of the theorem on existence and uniqueness of SDEwMS mentioned below is similar to~\cite[Theorem 3.23, 3.24]{mao2006}.
\begin{theorem} \label{th:eu}
	Let Assumptions \ref{as:initial} to \ref{as:diffusion} be satisfied. 
	Then, there exists a unique strong solution $X:=\{X(t)\}_{t \in [0,T]}$ of SDEwMS \eqref{eq:sdems} and  for all $\tilde{p} \in [0,p]$,  we have
	\begin{align*}
		\Big( \mathbb{E}\Big[ \sup_{t \in [0,T]}|X(t)|^{\tilde{p}} \big| \mathcal{F}_T^r \Big]\Big)^{1/\tilde{p}} \leq  K
		\text{ and } \Big( \mathbb{E}\big[|X(t) -X(s)|^{\tilde{p}} \big| \mathcal{F}_T^r \big]\Big)^{1/\tilde{p}} \leq  K |t-s|^{1/2}
	\end{align*}
	where $K(d,  T,  \mathbbm{E}[|X_0|^{p}], C)$ is a positive constant.  
\end{theorem}
\subsection{Randomized Milstein Scheme}
In this section, we introduce an explicit Milstein scheme for SDEwMS \eqref{eq:sdems} by adapting the randomization technique for standard SDEs developed  in \cite{Kruse2019}. 
Here, the additional challenges of the intertwining of continuous dynamics of the solution process $X$ of SDEwMS \eqref{eq:sdems} with the discontinuous dynamics of the Markov chain $r$ are tackled by introducing a martingale associated with $r$ which follows from \cite{nguyen2017milstein}.
Define 
\begin{align*}
	[M_{j_0k_0}](t):= \displaystyle \sum_{0 \leq s \leq t} \mathbbm{1}_{\{r(s-)=j_0\}} \mathbbm{1}_{\{r(s)=k_0\}},  & \quad \langle M_{j_0 k_0} \rangle (t) :=\int_0^t \mathfrak{q}_{j_0 k_0}  \mathbbm{1}_{\{r(s-)=j_0\}} ds, 
	\\
	M_{j_0 k_0}(t) := & [M_{j_0 k_0}](t) -\langle M_{j_0 k_0} \rangle (t) 
\end{align*}
almost surely for any $j_0, k_0 \in S$, $j_0 \neq k_0$,  $t \in [0,T]$. 
Clearly,    $[M_{j_0k_0}]:=\{[M_{j_0k_0}](t)\}_{t \geq 0}$ is an optional process,  $\langle M_{j_0k_0} \rangle:=\{\langle M_{j_0k_0} \rangle (t)\}_{t \geq 0}$ is a predictable process and $M_{j_0 k_0}:= \{M_{j_0 k_0}(t)\}_{t \geq 0}$ is a discontinuous squared integrable  $\tilde{\mathbb{F}}^r$-martingale.  Also,  $M_{j_0 k_0}\equiv 0$ for $k_0 = j_0$. 

Moreover, consider a sequence 
\begin{align}
	u:=\{u_j\}_{j\in \mathbb{N}} \label{eq:uniform}
\end{align}
of i.i.d.  standard uniformly distributed random variables on a probability space $(\tilde {\Omega}^u,  \tilde{\mathscr F}^u, \tilde{\mathbb{P}}^u)$ which is equipped with the natural filtration $\tilde{\mathbb{F}}^u:=\{\tilde{\mathscr F}_j^u\}_{j \in \mathbb{N}}$  of $u$. 
Assume  $u$ to be independent of $X_0$, $B$ and $r$. 

Consider a non-equidistant  partition 
\begin{align}
	\varrho_h:=\{(t_0, t_1, \ldots, t_{n_h}): 0=t_0 < t_1<\cdots<t_{n_h}=T\} \label{eq:grid}
\end{align}
of the interval $[0,T]$ with $n_h\in \mathbb{N}$ as the number of sub-intervals having lengths $h_j=t_{j}-t_{j-1}$ for $j \in \{1, \ldots, n_h\}$.  
Define $h:=\displaystyle \max_{j \in \{1,\ldots, n_h\}} h_j\leq \min(1,T)$. 
Let us now introduce a product probability space $(\Omega, \mathscr F, \mathbb{P})$ with $\Omega:=\tilde{\Omega} \times \tilde{\Omega}^u$,  $\mathscr F : =\tilde {\mathscr F} \otimes \tilde{\mathscr F}^u$ and $\mathbb{P}: = \tilde{\mathbb{P}}\otimes \tilde{\mathbb{P}}^u$ and equip this with a filtration $\mathbb{F}^h:=\{{\mathscr F}_j^h\}_{j \in \{0, \ldots, n_h\}}$ where $\mathscr F_j^h := \tilde{\mathscr F}_{t_j}\otimes \tilde{\mathscr F}^u_j$ for $j \in \{0, \ldots, n_h\}$. 
Also,  denote~by, for $j\in \{1, \ldots, n_h\}$,  
\begin{align}
	t_{j-1}^u : =  t_{j-1} + h_j u_j, \quad r_{j-1}:= r(t_{j-1})  \quad  \mbox{and} \quad   r_{j-1}^u:= r(t_{j-1}^u). \label{eq:tju}
\end{align}
The drift-randomized scheme for SDEwMS \eqref{eq:sdems} can be defined in two stages as follows;
\begin{align}
	X_{j}^{h, u}&= X_{j-1}^{h} +b(t_{j-1}, X_{j-1}^h, r_{j-1})   h_j u_j + \sum_{\ell=1}^{\tilde{d}} \sigma_{\ell}(t_{j-1}, X_{j-1}^h, r_{j-1}) \int_{t_{j-1}}^{ t_{j-1}^u} dB_{\ell} (s) \label{eq:euler} 
	\\
	X_{j}^h& = X_{j-1}^h+ b( t_{j-1}^u, X_{j}^{h,u}, r_{j-1}^u)   h_j +  \sum_{\ell=1}^{\tilde{d}} \sigma_{\ell}(t_{j-1}, X_{j-1}^h, r_{j-1})\{B_{\ell}(t_{j})-B_{\ell}(t_{j-1})\} \notag
	\\
	& 
	\quad + \sum_{\ell_{1},\ell_{2}=1}^{\tilde{d}}\int_{t_{j-1}}^{t_{j}} \int_{t_{j-1}}^{s} \mathcal{D}_x \sigma_{\ell_1}(t_{j-1}, X_{j-1}^h, r_{j-1}) \sigma_{\ell_2} (t_{j-1}, X_{j-1}^h, r_{j-1})  dB_{\ell_{2}}(v)dB_{\ell_{1}}(s) \notag
	\\
	&
	\quad + \sum_{\ell=1}^{\tilde{d}} \mathbbm{1}_{\big\{N_{t_{j-1}}^{t_j}=1\big\}} \big\{ \sigma_{\ell}(t_{j-1},X_{j-1}^h, r_{j})- \sigma_{\ell}(t_{j-1},X_{j-1}^h,  r_{j-1}) \big\} \big\{ B_{\ell}(t_{j})-B_{\ell}(\uptau_{1}^{j-1})\big\} \label{eq:scheme}
\end{align}
almost surely with initial values $X_0^h=X_0$ and $r(0) = i_0$, where $N_{t_{j-1}}^{t_j}$ is the number of switchings and $\uptau_{1}^{j-1}$ is  the first switching time of the Markov chain $r$ in the interval $(t_{j-1}, t_{j}]$ for any $j \in \{1, \ldots, n_h\}$.
Further, we denote by $X_j:=X(t_j)$ for $j \in \{0, 1, \ldots, n_h\}$.

The following theorem is the main result of this manuscript, stating the strong rate of convergence of our scheme \eqref{eq:scheme}. The proof is given in Section \ref{sec:Rate}.
\begin{theorem}
	\label{thm:rate}
	Let Assumptions \ref{as:initial} to \ref{as:diffusion} be satisfied. 
	Then, the drift randomized Milstein scheme \eqref{eq:scheme} converges to the true solution of SDEwMS \eqref{eq:sdems} with strong rate of convergence equal to $1.0$, \textit{i.e.}, there exist a positive constant $K$,  independent of $h$, such that
	\begin{align*}
		\Big\| \max_{n\in \{0,\ldots, n_h \}} |X_n - X_n^h|\Big\|_{L^2(\Omega)} \leq Kh.
	\end{align*}
\end{theorem}
\begin{remark}
	When the state space $S$ of the Markov chain $r$ is a singleton set, the  scheme defined in Equations \eqref{eq:euler} and \eqref{eq:scheme} reduces to the randomized Milstein scheme for standard SDE studied in~\cite[Equation 9]{Kruse2019}. 
	If there is no randomization in the drift coefficient, \textit{i.e.}, when the standard uniform random variables $\{u_j\}_{j \in \mathbb{N}}$ in Equation \eqref{eq:euler} are all replaced by zeros, i.e., for $j\in \{1,\ldots,n_h\}$, $t_{j-1}^u = t_{j-1}$ in Equation \eqref{eq:tju} and $X_{j}^{h,u} = X_{j-1}^h$ in Equation \eqref{eq:scheme}, then our scheme becomes Milstein scheme of \cite[Equation (4.20)]{nguyen2017milstein}. 
	Although, we do not claim this to be our special case due to the fact that additional regularity assumptions are needed on the drift coefficient to recover rate $1.0$, \citep[see][]{ kumar2020explicit,  kumar2021note,  kumar2022,nguyen2017milstein}. 	
\end{remark}
\subsection{Importance of the Brownian Increments $\big\{ B_{\ell}(t_{j})-B_{\ell}(\uptau_{1}^{j-1})\big\}$}
Notice that the last term in Equation  \eqref{eq:scheme} consists of  Brownian increments $\big\{ B_{\ell}(t_{j})-B_{\ell}(\uptau_{1}^{j-1})\big\}_{\ell\in\{1, \ldots, \tilde{d}\}}$ for $j\in \{1, \ldots
, n_h\}$ due to  which one requires to evaluate  Brownian trajectories on every switching time of the Markov chain $r$. 
To illustrate the importance of these increments,  the following  possible variants of the randomized Milstein scheme defined in Equations \eqref{eq:euler} and \eqref{eq:scheme} are explored.
\subsubsection{Modified Randomized Scheme}
\label{subsec:bsi}
We examine the rate of convergence of the  randomized Milstein scheme  when  $\big\{ B_{\ell}(t_{j})-B_{\ell}(\uptau_{1}^{j-1})\big\}$   is replaced by $\big\{ B_{\ell}(t_{j})-B_{\ell}(t_{j-1})\big\}$ in Equation \eqref{eq:scheme} and refer this variant as  the modified randomized  scheme. 
Formally, it is defined as follows; 
\begin{align}
	X_{j}^{h, u,M}&= X_{j-1}^{h, M} +b(t_{j-1}, X_{j-1}^{h, M}, r_{j-1})   h_j u_j + \sum_{\ell=1}^{\tilde{d}} \sigma_{\ell}(t_{j-1}, X_{j-1}^{h, M}, r_{j-1}) \int_{t_{j-1}}^{ t_{j-1}^u} dB_{\ell} (s) \label{eq:euler:wsi} 
	\\
	X_{j}^{h, M}& = X_{j-1}^{h, M}+ b( t_{j-1}^u, X_{j}^{h,u, M}, r_{j-1}^u)   h_j +  \sum_{\ell=1}^{\tilde{d}} \sigma_{\ell}(t_{j-1}, X_{j-1}^{h, M}, r_{j-1})\{B_{\ell}(t_{j})-B_{\ell}(t_{j-1})\} \notag
	\\
	& 
	+ \sum_{\ell_{1},\ell_{2}=1}^{\tilde{d}}\int_{t_{j-1}}^{t_{j}} \int_{t_{j-1}}^{s} \mathcal{D}_x \sigma_{\ell_1}(t_{j-1}, X_{j-1}^{h, M}, r_{j-1}) \sigma_{\ell_2} (t_{j-1}, X_{j-1}^{h, M}, r_{j-1})  dB_{\ell_{2}}(v)dB_{\ell_{1}}(s) \notag
	\\
	+ \sum_{\ell=1}^{\tilde{d}} & \mathbbm{1}_{\big\{N_{t_{j-1}}^{t_j}=1\big\}} \big\{ \sigma_{\ell}(t_{j-1},X_{j-1}^{h, M}, r_{j})- \sigma_{\ell}(t_{j-1},X_{j-1}^{h, M},  r_{j-1}) \big\} \big\{ B_{\ell}(t_{j})-B_{\ell}(t_{j-1})\big\} \label{eq:scheme:wsi}
\end{align}
almost surely with initial values $X_0^{h, M }=X_0$ and $r(0) = i_0$. 

The following theorem states that the modified randomized scheme given in Equations \eqref{eq:euler:wsi} and   \eqref{eq:scheme:wsi} has inferior rate of convergence  than the original scheme from Equations \eqref{eq:euler} and  \eqref{eq:scheme}. 
The proof is postponed to Section \ref{sec:Rate} and the numerical illustration is given in Section \ref{sec:numerical}. 
\begin{theorem} 
	\label{thm:rate:wsi}
	Let Assumptions \ref{as:initial} to \ref{as:diffusion} be satisfied. 
	Then, the modified randomized  scheme given in Equations \eqref{eq:euler:wsi} and \eqref{eq:scheme:wsi} converges to the true solution of SDEwMS \eqref{eq:sdems} with a rate $1/2$, \textit{i.e.}, there is a positive constant $K$,  not depending on $h$,  such that
	\begin{align*}
		\Big\| \max_{n\in \{0,\ldots, n_h \}} |X_n - X_n^{h, M}|\Big\|_{L^2(\Omega)} \leq Kh^{1/2}.
	\end{align*}
\end{theorem} 
\subsubsection{Reduced Randomized Scheme}
Apart from the modification mentioned in Subsection \ref{subsec:bsi},  it is worth examining the case when the last term of Equation \eqref{eq:scheme},  consisting of $\mathbbm{1}_{\big\{N_{t_{j-1}}^{t_j}=1\big\}}$,  is completely removed. 
We notice that the rate of convergence of the reduced randomized scheme is $1/2$ which demonstrates that last term of Equation \eqref{eq:scheme} plays an important role in achieving the rate $1.0$.   
More formally,  the reduced randomized scheme is given by, 
\begin{align}
	X_{j}^{h, u,R}&= X_{j-1}^{h, R} +b(t_{j-1}, X_{j-1}^{h, R}, r_{j-1})   h_j u_j + \sum_{\ell=1}^{\tilde{d}} \sigma_{\ell}(t_{j-1}, X_{j-1}^{h, R}, r_{j-1}) \int_{t_{j-1}}^{ t_{j-1}^u} dB_{\ell} (s) \label{eq:euler:ri} 
	\\
	X_{j}^{h, R}& = X_{j-1}^{h, R}+ b( t_{j-1}^u, X_{j}^{h,u, R}, r_{j-1}^u)   h_j +  \sum_{\ell=1}^{\tilde{d}} \sigma_{\ell}(t_{j-1}, X_{j-1}^{h, R}, r_{j-1})\{B_{\ell}(t_{j})-B_{\ell}(t_{j-1})\} \notag
	\\
	& 
	\quad + \sum_{\ell_{1},\ell_{2}=1}^{\tilde{d}}\int_{t_{j-1}}^{t_{j}} \int_{t_{j-1}}^{s} \mathcal{D}_x \sigma_{\ell_1}(t_{j-1}, X_{j-1}^{h, R}, r_{j-1}) \sigma_{\ell_2} (t_{j-1}, X_{j-1}^{h, R}, r_{j-1})  dB_{\ell_{2}}(v)dB_{\ell_{1}}(s)  \label{eq:scheme:ri}
\end{align}
almost surely with initial values $X_0^{h, R}=X_0$ and $r(0) = i_0$. 

The following theorem shows that the reduced randomized  scheme given in Equations~\eqref{eq:euler:ri} and \eqref{eq:scheme:ri} does not achieve the rate $1.0$ of the original scheme from  Equations~\eqref{eq:euler} and \eqref{eq:scheme}. 
The proof  is discussed in Section \ref{sec:Rate} and the numerical demonstration is shown in Section \ref{sec:numerical}. 
\begin{theorem} 
	\label{thm:rate:ri}
	Let Assumptions \ref{as:initial} to \ref{as:diffusion} hold. 
	Then, the reduced randomized scheme given in Equations \eqref{eq:euler:ri} and  \eqref{eq:scheme:ri} converges to the true solution of SDEwMS \eqref{eq:sdems} with a  rate $1/2$, \textit{i.e.}, there is a positive constant $K$,  independent of  $h$,  such that
	\begin{align*}
		\Big\| \max_{n\in \{0,\ldots, n_h \}} |X_n - X_n^{h, R}|\Big\|_{L^2(\Omega)} \leq Kh^{1/2}.
	\end{align*}
\end{theorem} 
\begin{remark}
	Notice that the rate of convergence of the modified and reduced randomized schemes turns out to be $1/2$, which is same as that of the classical Euler scheme for SDEwMS. Thus, it is important to keep track of the values of the driving Brownian noise at the switching times of the Markov chain $r$ to achieve the optimal rate $1.0$ of Milstein scheme. 
	This is indeed a special structure of both the randomized and non-randomized Milstein schemes of SDEwMS as the last term in Equation \eqref{eq:scheme} can not further be simplified.  
\end{remark}

\begin{remark}
	\label{rem:Half}
	In Section \ref{sec:numerical}, we demonstrate that the non-randomized modified scheme (\textit{i.e.} when $u_j \equiv 0$ for all $j \in \mathbb{N}$ in Equation \eqref{eq:euler:wsi} and \eqref{eq:scheme:wsi}) and the non-randomized reduced scheme (\textit{i.e.} when $u_j \equiv 0$ for all $j \in \mathbb{N}$ in Equation \eqref{eq:euler:ri} and \eqref{eq:scheme:ri}) exhibit better results than the classical Euler scheme and their randomized versions. It is worth metioning that in Equations \eqref{eq:scheme:wsi} and \eqref{eq:scheme:ri}, the derivative of the diffusion coefficient appears. However, it can be avoided by writing the derivative-free versions of these schemes by replacing $\sigma'(t,x,i)$ in the scheme by $\{\sigma(t,x+h,i) - \sigma(t,x,i)\}/h$
	for any $t >0$, $h>0$ and $i \in S$.
	Thus, the non-randomized derivative-free modified and reduced schemes can achieve order $1/2$ of convergence under Lipschitz continuity of the drift and diffussion coefficients. Moreover, these schemes are more efficient than the Euler scheme as illustrated in Section \ref{sec:numerical}.  
\end{remark}
The following remarks conclude this section.
\begin{remark} \label{rem:D}
	Due to Assumption \ref{as:diffusion}, there exist a constant $K >0$ such that
	\begin{align*}
		|\mathcal{D}_x \sigma_{\ell_1}(t, x, \imath) \sigma_{\ell_2} (t, x, \imath) - \mathcal{D}_x \sigma_{\ell_1}(s, \bar{x}, \imath) \sigma_{\ell_2} (s, \bar{x}, \imath)| & \leq K \{ (1+\max(|x|,|\bar{x}|))|t-s| + |x-\bar{x}|\},
		\\
		|\sigma_{\ell}(t,x,\imath) - \sigma_{\ell}(s,\bar{x},\imath) - \mathcal{D}_x \sigma_{\ell}(s,\bar{x},\imath)(x-\bar{x})| & \leq  K \{ (1+\max(|x|,|\bar{x}|))|t-s| + |x-\bar{x}|^2\} 
	\end{align*}
	for all $t, s \in [0,T]$,  $x,\bar{x}\in \mathbb{R}^d$, $\ell, \ell_{1},\ell_{2} \in \{1, \ldots, \tilde{d}\}$ and $\imath \in S$. 
\end{remark}
\begin{remark}
	\label{rem:growth}
	By Assumptions \ref{as:drift}, \ref{as:diffusion} and Remark \ref{rem:D}, there is a constant $K>0$ such that 
	\begin{align*}
		|b(t,x,\imath)| \vee |\sigma_{\ell}(t,x,\imath)| \vee |\mathcal{D}_x \sigma_{\ell}(t, x, \imath)| \vee |\mathcal{D}_x \sigma_{\ell_1}(t, x, \imath) \sigma_{\ell_2}(t,x,\imath)| & \leq K  (1+|x|)
	\end{align*}
	for all $t \in [0,T]$, $x \in \mathbb{R}^d$, $\ell,\ell_1,\ell_2 \in \{1,\ldots, \tilde{d}\}$ and $\imath \in S$.
\end{remark}
\section{Rate of Convergence}
\label{sec:Rate}
In order to prove our main result,  stated in Theorem \ref{thm:rate},  we require some  auxiliary lemmas which are shown below. 
The following lemma provides the moment estimate  of the scheme \eqref{eq:scheme}.

\begin{lemma}
	\label{moment}
	Under Assumptions \ref{as:initial}, \ref{as:drift} and \ref{as:diffusion}, there exists a constant $K > 0$, independent of $h$, such that 
	\begin{align*}
		\mathbb{E}\Big[\max_{n\in \{1,\ldots, n_h \}}|  X_{n}^{h}|^p | \mathcal{F}_T^r \Big]  \leq K.
	\end{align*}
\end{lemma}
\begin{proof}
	Using Equation \eqref{eq:euler} and Remark \ref{rem:growth}, one can obtain
	\begin{align}
		\label{eq:meu}
		\mathbb{E} &\big[|X_j^{h,u}|^p | \mathcal{F}_T^r\big]  \leq  K \mathbb{E} \big[| X_{j-1}^{h}|^p \big] + K (h_j )^p \mathbb{E} \big[(u_j)^p | b(t_{j-1}, X_{j-1}^h, r_{j-1})|^p \big] \notag
		\\
		& \quad + K \mathbb{E} \Big[ |  \sum_{\ell=1}^{\tilde{d}} \int_{t_{j-1}}^{ t_{j-1}^u} \sigma_{\ell}(t_{j-1}, X_{j-1}^h, r_{j-1})  dB_{\ell} (s)|^p \Big] \notag
		\\
		&
		\leq K \mathbb{E} \big[| X_{j-1}^{h}|^p \big] + K (h_j)^p \big( 1 + \mathbb{E} [|X_{j-1}^h|^p ] \big) + K \sum_{\ell=1}^{\tilde{d}} \mathbb{E} \Big[ \Big( \int_{t_{j-1}}^{ t_{j-1}^u} |\sigma_{\ell}(t_{j-1}, X_{j-1}^h, r_{j-1})|^2  ds\Big)^{\frac{p}{2}}  \Big] \notag
		\\
		&
		\leq K \mathbb{E} \big[| X_{j-1}^{h}|^p \big] + K (h_j)^{\frac{p}{2}} \big( 1 + \mathbb{E} [|X_{j-1}^h|^p ] \big)
	\end{align}
	for any $j \in \{1,\ldots,n_h\}$. Using Equations \eqref{eq:scheme},
	\begin{align*}
		& \mathbb{E}\Big[\max_{n\in \{1,\ldots, k \}}|  X_{n}^{h}|^p | \mathcal{F}_T^r\Big] \leq K \mathbb{E}\big[ |X_{0}^h|^p \big]+ K\mathbb{E}\Big[\max_{n\in \{1,\ldots, k \}}| \sum_{j=1}^n b( t_{j-1}^u, X_{j}^{h,u}, r_{j-1}^u)   h_j|^p | \mathcal{F}_T^r \Big]  \notag
		\\
		& + K \mathbb{E}\Big[\max_{n\in \{1,\ldots, k \}}| \sum_{j=1}^n \sum_{\ell=1}^{\tilde{d}} \sigma_{\ell}(t_{j-1}, X_{j-1}^h, r_{j-1})\{B_{\ell}(t_{j})-B_{\ell}(t_{j-1})\}|^p | \mathcal{F}_T^r \Big] \notag
		\\
		& 
		+ K \mathbb{E}\Big[\hspace{-0.5mm} \max_{n\in \{1,\ldots, k \}}\hspace{-0.5mm} | \hspace{-0.5mm} \sum_{j=1}^n \sum_{\ell_{1},\ell_{2}=1}^{\tilde{d}}\int_{t_{j-1}}^{t_{j}}\hspace{-0.5mm}  \int_{t_{j-1}}^{s} \mathcal{D}_x \sigma_{\ell_1}(t_{j-1}, X_{j-1}^h, r_{j-1}) \sigma_{\ell_2} (t_{j-1}, X_{j-1}^h, r_{j-1})  dB_{\ell_{2}}(v)dB_{\ell_{1}}(s)|^p | \mathcal{F}_T^r \Big] \notag
		\\
		&
		+ K \mathbb{E}\Big[\hspace{-0.5mm} \max_{n\in \{1,\ldots, k \}}|\hspace{-0.5mm}  \sum_{j=1}^n \sum_{\ell=1}^{\tilde{d}}\hspace{-0.5mm}  \mathbbm{1}_{\big\{N_{t_{j-1}}^{t_j}=1\big\}} \big\{ \sigma_{\ell}(t_{j-1},X_{j-1}^h, r_{j})\hspace{-0.5mm} -\hspace{-0.5mm}  \sigma_{\ell}(t_{j-1},X_{j-1}^h,  r_{j-1}) \big\}\hspace{-0.5mm}  \big\{ B_{\ell}(t_{j})\hspace{-0.5mm} - \hspace{-0.5mm} B_{\ell}(\uptau_{1}^{j-1})\big\}|^p | \mathcal{F}_T^r \Big] 
	\end{align*} 
	for any $k \in \{1,\ldots,n_h\}$. Observe that 
	$$\displaystyle \Big\{ \sum_{j=1}^{n} \sum_{\ell=1}^{\tilde{d}} \mathbbm{1}_{\big\{N_{t_{j-1}}^{t_j}=1\big\}} \big\{ \sigma_{\ell}(t_{j-1},X_{j-1}^h, r_{j})- \sigma_{\ell}(t_{j-1},X_{j-1}^h,  r_{j-1}) \big\} \big\{ B_{\ell}(t_{j})-B_{\ell}(\uptau_{1}^{j-1})\big\} \Big\}_{n \in \{1, \ldots, n_h\}}$$ 
	is an $\{\tilde{\mathcal{F}}_{\uptau_1^{n}  \wedge t_{n+1}} \vee  \tilde{\mathcal{F}}_T^r\}_{n \in \{1, \ldots, n_h\}}$-martingale. Thus,
	\begin{align*} 
		& \mathbb{E}\Big[\max_{n\in \{1,\ldots, k \}}|  X_{n}^{h}|^p | \mathcal{F}_T^r \Big] \leq   K \mathbb{E}\big[ |X_{0}^h|^p \big] +  K (k)^{p-1}\sum_{j=1}^k (h_j)^p\mathbb{E}\Big[|b( t_{j-1}^u, X_{j}^{h,u}, r_{j-1}^u)|^p | \mathcal{F}_T^r \Big] \notag
		\\
		& \quad +   K k^{\frac{p}{2}-1} \sum_{j=1}^k \sum_{\ell=1}^{\tilde{d}} \mathbb{E}\Big[\big(  |\sigma_{\ell}(t_{j-1}, X_{j-1}^h, r_{j-1})| |B_{\ell}(t_{j})-B_{\ell}(t_{j-1})|\big)^{p} | \mathcal{F}_T^r \Big] \notag
		\\
		& 
		+  K k^{\frac{p}{2}-1}\sum_{j=1}^k \sum_{\ell_{1},\ell_{2}=1}^{\tilde{d}}  (h_j)^{\frac{p}{2}-1} \int_{t_{j-1}}^{t_{j}} \mathbb{E}\Big[ \Big( \int_{t_{j-1}}^{s} |\mathcal{D}_x \sigma_{\ell_1}(t_{j-1}, X_{j-1}^h, r_{j-1}) \sigma_{\ell_2} (t_{j-1}, X_{j-1}^h, r_{j-1})|^2  dv \Big)^{\frac{p}{2}}ds  | \mathcal{F}_T^r \Big] \notag
		\\
		&
		\quad +  K k^{\frac{p}{2}-1}\sum_{j=1}^k \sum_{\ell=1}^{\tilde{d}} \mathbb{E}\Big[ \Big(  \mathbbm{1}_{\big\{N_{t_{j-1}}^{t_j}=1\big\}} | \sigma_{\ell}(t_{j-1},X_{j-1}^h, r_{j})- \sigma_{\ell}(t_{j-1},X_{j-1}^h,  r_{j-1}) |^2 
		\\
		&  \quad  \quad \times \mathbb{E} \big[| B_{\ell}(t_{j})-B_{\ell}(\uptau_1^{j-1} \wedge t_{j})|^2  | \tilde{\mathcal{F}}_{\uptau_1^{j-1} \wedge t_{j}} \vee \tilde{\mathcal{F}}_T^r   \big]\Big)^{\frac{p}{2}} | \mathcal{F}_T^r \Big] \notag
	\end{align*}
	for any $k \in \{1,\ldots,n_h\}$. Due to Equation \eqref{eq:meu}, Remarks \ref{rem:D} and \ref{rem:growth},
	\begin{align*} 
		\mathbb{E}\Big[\max_{n\in \{1,\ldots, k \}}&|  X_{n}^{h}|^p \Big] \leq  K \mathbb{E}\big[ |X_{0}^h|^p | \mathcal{F}_T^r\big] +  K(k)^{p-1}\sum_{j=1}^k (h_j)^p\big( 1 + \mathbb{E}\big[|X_{j-1}^h|^p | \mathcal{F}_T^r \big] \big) \notag\\
		& +   K (k)^{\frac{p}{2}-1} \sum_{j=1}^k (h_j)^{\frac{p}{2}} \big( 1 + \mathbb{E}\big[ |X_{j-1}^h|^p | \mathcal{F}_T^r \big] \big) +  K (k)^{\frac{p}{2}-1} \sum_{j=1}^k (h_j)^p \big( 1 +  \mathbb{E}\big[ |X_{j-1}^h|^p | \mathcal{F}_T^r\big] \big) \notag
		\\
		\leq \ &  K \mathbb{E}\big[ |X_{0}^h|^p \big] +  K\Big( 1 + \sum_{j=1}^k h_j\mathbb{E}\big[|X_{j-1}^h|^p | \mathcal{F}_T^r\big] \Big)
	\end{align*}
	for any $k \in \{1,\ldots,n_h\}$. The proof is completed by using the Gr\"onwall's inequality.
\end{proof}
Following the approach of \cite{Kruse2019},  we introduce the notion of bi-stability and consistency of the scheme \eqref{eq:scheme}.
For the time-grid $\varrho_h$ defined in Equation \eqref{eq:grid},  consider Banach spaces  $(G^h_q, \|\cdot\|_{G^h_q})$ and $(G^h_{q,S}, \|\cdot\|_{G^h_{q,S}})$ of stochastic grid processes $Y^h:=\{(Y_{0}^h, \ldots, Y_{n_h}^h): Y_j^h \in L^{q}(\Omega, \mathscr F_j^h, \mathbb{P}; \mathbb{R}^d), \, j\in \{0,\ldots, n_h\}\}$ where $L^{q}(\Omega, \mathscr F_j^h, \mathbb{P}; \mathbb{R}^d)$ is space of $\mathscr F_j^h$-measurable, $\mathbb{R}^d$-valued  random variables having finite $q$-th moment,  and the norms $\|\cdot\|_{G^h_q}$ and $\|\cdot\|_{G^h_{q,S}}$ are given by 
\begin{align}
	\big\|Y^h \big\|_{G^h_q} := \Big\| \max_{j\in \{0,\ldots, n_h \}} |Y_j^h| \Big\|_{L^q(\Omega)} \text{ and } \big\|Y^h \big\|_{G^h_{q,S}}  := \big\| Y_0^h \big\|_{L^q(\Omega)}+ \Big\|\max_{j\in \{1,\ldots, n_h \}} |\sum_{k=1}^jY_k^h| \Big\|_{L^q(\Omega)}, \label{eq:norm:spiker}
\end{align}
respectively. 
Following Equations  \eqref{eq:euler} and \eqref{eq:scheme},  define the increment functions $\Xi^h:=(\Xi^h_1,  \ldots,  \Xi^h_{n_h})$ on the grid $\varrho_h$, 
\begin{align}
	\Upsilon^h_j(y, u, \CMjmath) &:= y+ b(t_{j-1}, y,  \CMjmath)   h_j u + \sum_{\ell=1}^{\tilde{d}} \sigma_{\ell}(t_{j-1}, y,  \CMjmath) \int_{t_{j-1}}^{t_{j-1}+ h_j u} dB_{\ell} (s) \label{eq:Uep} 
	\\
	\Xi^h_j(y, u, \imath,\CMjmath, \bar{\imath}) &:= b(t_{j-1}+h_j u, \Upsilon^h_j(y, u, \CMjmath),  \imath)   h_j +  \sum_{\ell=1}^{\tilde{d}} \sigma_{\ell}(t_{j-1}, y, \CMjmath)\{B_{\ell}(t_{j})-B_{\ell}(t_{j-1})\} \notag
	\\
	& 
	\quad+ \sum_{\ell_{1},\ell_{2}=1}^{\tilde{d}}\int_{t_{j-1}}^{t_{j}} \int_{t_{j-1}}^{s} \mathcal{D}_x \sigma_{\ell_1}(t_{j-1}, y,  \CMjmath) \sigma_{\ell_2} (t_{j-1}, y, \CMjmath)  dB_{\ell_{2}}(v)dB_{\ell_{1}}(s) \notag
	\\
	&
	\quad + \sum_{\ell=1}^{\tilde{d}} \mathbbm{1}_{\big\{N_{t_{j-1}}^{t_j}=1\big\}} \big\{ \sigma_{\ell}(t_{j-1},y, \bar{\imath})- \sigma_{\ell}(t_{j-1},y,  \CMjmath) \big\} \big\{ B_{\ell}(t_{j})-B_{\ell}(\uptau_{1}^{j-1})\big\} \label{eq:incrment}
\end{align}  
almost surely for any $y\in \mathbb{R}^d$,  $j \in\{1,\ldots, n_h\}$, $u \in (0,1)$,  $j_0, k_0 \in S$. 
For a grid process $Y^h \in G^h$,   let us define the residual grid process $\mathcal{R}[Y^h]:= \big(\mathcal{R}_0[Y^h], \mathcal{R}_1[Y^h], \ldots, \mathcal{R}_{n_h}[Y^h]\big)\in G^h_{q, S}$ as, 
\begin{align}
	\mathcal{R}_0[Y^h]:= Y_0^h - X_0^h, \quad \mathcal{R}_j[Y^h] := Y_j^h- Y_{j-1}^h - \Xi_j^h(Y_{j-1}^h , u_j,   r_{j-1}^u,   r_{j-1}, r_j) \label{eq:residual}
\end{align}
for $j \in \{1, \ldots, n_h\}$. 

To prove bi-stability of the randomized Milstein scheme \eqref{eq:scheme}, we also require the following lemma. 
\begin{lemma} \label{lem:E}
	Let Assumptions \ref{as:drift} and \ref{as:diffusion} hold. 
	Then, $Y^h, Z^h \in G_q^h$, $q \geq 2$ satisfies the following 
	\begin{align}
		\Big\|  \max_{n\in\{1, \ldots, k\}} &  \big|\sum_{j=1}^n \big[\Xi^h_j(Y_{j-1}^h,  u_j,r_{j-1}^u, r_{j-1}, r_j)-\Xi^h_j(Z_{j-1}^h, u_j,r_{j-1}^u, r_{j-1}, r_j)\big]\big| \Big\|_{L^q(\Omega)} \notag 
		\\
		&  \leq   K   \Big( \sum_{j=1}^k   h_j \Big\|\max_{i\in \{0, \ldots, j-1\}} |Y_{i}^h-Z_{i}^h|  \Big\|_{L^{q}(\Omega)}^2   \Big)^{1/2} \notag
	\end{align}
	for any $k \in \{1, \ldots, n_h\}$. 
	This further implies that 
	\begin{align*}
		\Big\|  \max_{n\in\{1, \ldots, n_h\}} &  \big|\sum_{j=1}^n \big[\Xi^h_j(Y_{j-1}^h,  u_j, r_{j-1}^u, r_{j-1}, r_j)-\Xi^h_j(Z_{j-1}^h, u_j, r_{j-1}^u, r_{j-1}, r_j)\big]\big| \Big\|_{L^q(\Omega)} \notag  \leq K \| Y^h - Z^h \|_{G^h_q}. 
	\end{align*}
\end{lemma}
\begin{proof}
	From Equation \eqref{eq:incrment}, 
	\begin{align}
		\Xi^h_j(Y_{j-1}^h,  u_j,& r_{j-1}^u, r_{j-1}, r_j)-\Xi^h_j(Z_{j-1}^h, u_j, r_{j-1}^u, r_{j-1}, r_j) \notag
		\\
		=& \big[b( t_{j-1}^u, \Upsilon^h_j(Y_{j-1}^h, u_j, r_{j-1}),  r_{j-1}^u) -b( t_{j-1}^u, \Upsilon^h_j(Z_{j-1}^h, u_j, r_{j-1}),  r_{j-1}^u)  \big] h_j  \notag
		\\
		& +  \sum_{\ell=1}^{\tilde{d}} \big[\sigma_{\ell}(t_{j-1}, Y_{j-1}^h, r_{j-1})-\sigma_{\ell}(t_{j-1}, Z_{j-1}^h, r_{j-1})\big](B_{\ell}(t_{j})-B_{\ell}(t_{j-1})) \notag
		\\
		& 
		+ \sum_{\ell_{1},\ell_{2}=1}^{\tilde{d}}\int_{t_{j-1}}^{t_{j}} \int_{t_{j-1}}^{s} \big[ \mathcal{D}_x \sigma_{\ell_1}(t_{j-1}, Y_{j-1}^h, r_{j-1}) \sigma_{\ell_2}(t_{j-1}, Y_{j-1}^h, r_{j-1}) \notag
		\\
		&\qquad - \mathcal{D}_x \sigma_{\ell_1}(t_{j-1}, Z_{j-1}^h, r_{j-1}) \sigma_{\ell_2}(t_{j-1}, Z_{j-1}^h, r_{j-1})\big]dB_{\ell_{2}}(v)dB_{\ell_{1}}(s) \notag
		\\
		&
		+ \sum_{\ell=1}^{\tilde{d}} \mathbbm{1}_{\big\{N_{t_{j-1}}^{t_j}=1\big\}} \big\{ \sigma_{\ell}(t_{j-1},Y_{j-1}^h, r_j)-\sigma_{\ell}(t_{j-1},Z_{j-1}^h, r_j) \big\} \big\{ B_{\ell}(t_{j})-B_{\ell}(\uptau_{1}^{j-1}) \big\} \notag
		\\
		&
		+ \sum_{\ell=1}^{\tilde{d}} \mathbbm{1}_{\big\{N_{t_{j-1}}^{t_j}=1\big\}} \big\{  \sigma_{\ell}(t_{j-1},Y_{j-1}^h,  r_{j-1})-\sigma_{\ell}(t_{j-1},Z_{j-1}^h,  r_{j-1}) \big\} \big\{ B_{\ell}(t_{j})-B_{\ell}(\uptau_{1}^{j-1})\big\} \notag
		\\
		= : \, & \mathcal{T}_j^{(1)} + \mathcal{T}_j^{(2)} + \mathcal{T}_j^{(3)} + \mathcal{T}_j^{(4)} + \mathcal{T}_j^{(5)} \label{eq:T1+T5}
	\end{align}  
	for any $j \in \{1, \ldots, n_h\}$ which further implies on using Minkowski inequality, 
	\begin{align}
		\Big\|  &\max_{n\in\{1, \ldots, k\}}   \big|\sum_{j=1}^n \big[\Xi^h_j(Y_{j-1}^h,  u_j, r_{j-1}^u, r_{j-1}, r_j)-\Xi^h_j(Z_{j-1}^h, u_j, r_{j-1}^u, r_{j-1}, r_j)\big]\big| \Big\|_{L^q(\Omega)}  \notag
		\\
		&  \leq \Big\| \max_{n\in\{1, \ldots, k\}} \big| \sum_{j=1}^n \mathcal{T}^{(1)}_j \big| \Big\|_{L^q(\Omega)} +  \Big\| \max_{n\in\{1, \ldots, k\}} \big| \sum_{j=1}^n \mathcal{T}^{(2)}_j \big| \Big\|_{L^q(\Omega)} + \Big\| \max_{n\in\{1, \ldots, k\}} \big| \sum_{j=1}^n \mathcal{T}^{(3)}_j \big| \Big\|_{L^q(\Omega)}\notag
		\\
		&+ \Big\| \max_{n\in\{1, \ldots, k\}} \big| \sum_{j=1}^n \mathcal{T}^{(4)}_j \big| \Big\|_{L^q(\Omega)}+ \Big\| \max_{n\in\{1, \ldots, k\}} \big| \sum_{j=1}^n \mathcal{T}^{(5)}_j \big| \Big\|_{L^q(\Omega)} \label{eq:sum:T1+T5}
	\end{align}
	for $k \in \{1, \ldots, n_h\}$. 
	We now estimate each term on the right side of the above Equation as follows. 
	By Minkowski inequality and Assumption \ref{as:drift},  
	\begin{align*}
		\Big\|& \max_{n\in\{1, \ldots, k\}} \big| \sum_{j=1}^n \mathcal{T}^{(1)}_j \big| \Big\|_{L^q(\Omega)} 
		\\
		& : = \Big\| \max_{n\in\{1, \ldots, k\}} \big| \sum_{j=1}^n  \big[b( t_{j-1}^u, \Upsilon^h_j(Y_{j-1}^h, u_j, r_{j-1}),  r_{j-1}^u) -b( t_{j-1}^u, \Upsilon^h_j(Z_{j-1}^h, u_j, r_{j-1}),  r_{j-1}^u)  \big] h_j  \big| \Big\|_{L^q(\Omega)}  \notag
		\\
		& \leq    \sum_{j=1}^k   h_j \big\|b( t_{j-1}^u, \Upsilon^h_j(Y_{j-1}^h, u_j, r_{j-1}),  r_{j-1}^u) -b( t_{j-1}^u, \Upsilon^h_j(Z_{j-1}^h, u_j, r_{j-1}),  r_{j-1}^u)   \big \|_{L^q(\Omega)}  \notag
		\\
		& \leq  K  \sum_{j=1}^k   h_j \big\| \Upsilon^h_j(Y_{j-1}^h, u_j, r_{j-1})-\Upsilon^h_j(Z_{j-1}^h, u_j, r_{j-1}) \big \|_{L^q(\Omega)}  \notag
	\end{align*}
	which on using Equation \eqref{eq:Uep} and Assumption \ref{as:drift} yields the following, 
	\begin{align*}
		\Big\| & \max_{n\in\{1, \ldots, k\}}  \big| \sum_{j=1}^n \mathcal{T}^{(1)}_j \big| \Big\|_{L^q(\Omega)} \leq  K  \sum_{j=1}^k   h_j  \| Y_{j-1}^h- Z_{j-1}^h\|_{L^q(\Omega)} 
		\\
		& \quad + K  \sum_{j=1}^k  h_j \big\| b(t_{j-1}, Y_{j-1}^h,  r_{j-1}) - b(t_{j-1}, Z_{j-1}^h,  r_{j-1})  \big\|_{L^q(\Omega)}   h_j 
		\\
		& \quad + K  \sum_{j=1}^k  h_j \Big\| \sum_{\ell=1}^{\tilde{d}} \big[\sigma_{\ell}(t_{j-1}, Y_{j-1}^h,  r_{j-1})-  \sigma_{\ell}(t_{j-1}, Z_{j-1}^h,  r_{j-1}) \big] \int_{t_{j-1}}^{ t_{j-1}^u}  dB_{\ell} (s) \Big\|_{L^q(\Omega)}
		\\
		& \leq K  \sum_{j=1}^k  h_j \big(1+ h_j+ h_j^{1/2} \big) \big\| Y_{j-1}^h- Z_{j-1}^h \big\|_{L^q(\Omega)} \leq K  \sum_{j=1}^k  h_j \big(1+ h+ h^{1/2} \big) \Big\| \sup_{i \in \{0, \ldots, j-1\}}|Y_{i}^h- Z_{i}^h| \Big\|_{L^q(\Omega)}
	\end{align*}
	and then squaring both the sides,  
	\begin{align}
		\Big\|  \max_{n\in\{1, \ldots, k\}}   \big| \sum_{j=1}^n \mathcal{T}^{(1)}_j \big| \Big\|_{L^q(\Omega)}^2 \leq  K  \sum_{j=1}^k  h_j  \Big\| \sup_{i \in \{0, \ldots, j-1\}}|Y_{i}^h- Z_{i}^h| \Big\|_{L^q(\Omega)}^2 \label{eq:T1}
	\end{align}
	for any $k \in \{1, \ldots, n_h\}$ where $h \leq 1$ is used. 
	From Equation \eqref{eq:T1+T5},  notice that  $\displaystyle \Big\{\sum_{j=1}^n \mathcal{T}^{(2)}_j \Big\}_{n \in \{1, \ldots, n_h\}}$ is a discrete-time martingale with respect to the filtration $\{\mathscr F_n^h\}_{n \in \{1,\ldots, n_h\}}$ and thus using martingale's inequality,    one has  
	\begin{align*}
		\Big\| & \max_{n\in\{1, \ldots, k\}} \big| \sum_{j=1}^n \mathcal{T}^{(2)}_j \big| \Big\|_{L^q(\Omega)} 
		\\ 
		& :=  \Big\| \max_{n\in\{1, \ldots, k\}} \big| \sum_{j=1}^n  \sum_{\ell=1}^{\tilde{d}} \big[\sigma_{\ell}(t_{j-1}, Y_{j-1}^h, r_{j-1})-\sigma_{\ell}(t_{j-1}, Z_{j-1}^h, r_{j-1})\big](B_{\ell}(t_{j})-B_{\ell}(t_{j-1}))  \big| \Big\|_{L^q(\Omega)}
		\\
		& \leq K  \Big\|   \sum_{j=1}^k \big| \sum_{\ell=1}^{\tilde{d}} \big[\sigma_{\ell}(t_{j-1}, Y_{j-1}^h, r_{j-1})-\sigma_{\ell}(t_{j-1}, Z_{j-1}^h, r_{j-1})\big](B_{\ell}(t_{j})-B_{\ell}(t_{j-1})) \big|^2    \Big\|_{L^{q/2}(\Omega)}^{1/2} 
	\end{align*} 
	which on squaring both sides and using Minkowski's inequality and Assumption \ref{as:diffusion} yield, 
	\begin{align}
		& \Big\| \max_{n\in\{1, \ldots, k\}} \big| \sum_{j=1}^n \mathcal{T}^{(2)}_j \big| \Big\|_{L^q(\Omega)}^2  \notag
		\\
		& \ \leq   K     \sum_{j=1}^k \Big\| \big| \sum_{\ell=1}^{\tilde{d}} \big[\sigma_{\ell}(t_{j-1}, Y_{j-1}^h, r_{j-1})-\sigma_{\ell}(t_{j-1}, Z_{j-1}^h, r_{j-1})\big](B_{\ell}(t_{j})-B_{\ell}(t_{j-1})) \big|^2    \Big\|_{L^{q/2}(\Omega)} \notag
		\\
		& \ \leq   K    \sum_{j=1}^k   \sum_{\ell=1}^{\tilde{d}} \Big\| \big|Y_{j-1}^h-Z_{j-1}^h   \big| \big|B_{\ell}(t_{j})-B_{\ell}(t_{j-1})\big|    \Big\|_{L^{q}(\Omega)}^2  \notag
		\\
		&  \leq   K    \sum_{j=1}^k   \sum_{\ell=1}^{\tilde{d}} \big\| Y_{j-1}^h-Z_{j-1}^h  \big\|_{L^{q}(\Omega)}^2   \big\| B_{\ell}(t_{j})-B_{\ell}(t_{j-1})   \big\|_{L^{q}(\Omega)}^2  \leq   K    \sum_{j=1}^k    h_j \Big\|\max_{i\in \{0, \ldots, j-1\}} |Y_{i}^h-Z_{i}^h|  \Big\|_{L^{q}(\Omega)}^2   \label{eq:T2}
	\end{align}
	for any $k \in \{1, \ldots, n_h\}$. 
	As before,  $\displaystyle \Big\{\sum_{j=1}^n \mathcal{T}^{(3)}_j \Big\}_{n \in \{1, \ldots, n_h\}}$ is an $\{\mathscr F_n^h \}_{n \in \{1, \ldots, n_h\}}$-martingale and thus,   
	\begin{align*}
		\Big\| \max_{n\in\{1, \ldots, k\}} \hspace{-0.5mm}  \big| \hspace{-0.5mm} \sum_{j=1}^n \mathcal{T}^{(3)}_j \big| & \Big\|_{L^q(\Omega)} 
		\hspace{-0.8mm}:= \hspace{-0.5mm} \Big\|  \max_{n\in\{1, \ldots, k\}} \hspace{-0.5mm} \big| \sum_{j=1}^n \sum_{\ell_{1},\ell_{2}=1}^{\tilde{d}}\int_{t_{j-1}}^{t_{j}} \hspace{-0.5mm} \int_{t_{j-1}}^{s} \hspace{-0.8mm} \big[ \mathcal{D}_x \sigma_{\ell_1}(t_{j-1}, Y_{j-1}^h, r_{j-1}) \sigma_{\ell_2}(t_{j-1}, Y_{j-1}^h, r_{j-1}) \notag
		\\
		& \qquad - \mathcal{D}_x \sigma_{\ell_1}(t_{j-1}, Z_{j-1}^h, r_{j-1}) \sigma_{\ell_2}(t_{j-1}, Z_{j-1}^h, r_{j-1})\big]dB_{\ell_{2}}(v)dB_{\ell_{1}}(s) \big| \Big\|_{L^q(\Omega)} \notag
		\\
		& \leq K\Big\|  \big(  \sum_{j=1}^k  \big| \sum_{\ell_{1},\ell_{2}=1}^{\tilde{d}}\int_{t_{j-1}}^{t_{j}} \int_{t_{j-1}}^{s} \big[ \mathcal{D}_x \sigma_{\ell_1}(t_{j-1}, Y_{j-1}^h, r_{j-1}) \sigma_{\ell_2}(t_{j-1}, Y_{j-1}^h, r_{j-1}) \notag
		\\
		& \qquad - \mathcal{D}_x \sigma_{\ell_1}(t_{j-1}, Z_{j-1}^h, r_{j-1}) \sigma_{\ell_2}(t_{j-1}, Z_{j-1}^h, r_{j-1})\big]dB_{\ell_{2}}(v)dB_{\ell_{1}}(s)\big|^2 \big)  \Big\|^{1/2}_{L^{q/2}(\Omega)} \notag
	\end{align*}
	which taking squaring on both sides and using Minkowski's inequality gives,
	\begin{align}
		\Big\|  \max_{n\in\{1, \ldots, k\}} \big| \sum_{j=1}^n \mathcal{T}^{(3)}_j \big| \Big\|_{L^q(\Omega)}^2  & \leq K\sum_{j=1}^k \Big\|  \big| \sum_{\ell_{1},\ell_{2}=1}^{\tilde{d}}\int_{t_{j-1}}^{t_{j}} \int_{t_{j-1}}^{s} \big[ \mathcal{D}_x \sigma_{\ell_1}(t_{j-1}, Y_{j-1}^h, r_{j-1}) \sigma_{\ell_2}(t_{j-1}, Y_{j-1}^h, r_{j-1}) \notag
		\\
		& \quad - \mathcal{D}_x \sigma_{\ell_1}(t_{j-1}, Z_{j-1}^h, r_{j-1}) \sigma_{\ell_2}(t_{j-1}, Z_{j-1}^h, r_{j-1})\big]dB_{\ell_{2}}(v)dB_{\ell_{1}}(s)\big|^2  \Big\|_{L^{q/2}(\Omega)} \notag
		\\
		& \leq  K \sum_{j=1}^k \Big\|  \sum_{\ell_{1},\ell_{2}=1}^{\tilde{d}}\int_{t_{j-1}}^{t_{j}} \int_{t_{j-1}}^{s} \big| \mathcal{D}_x \sigma_{\ell_1}(t_{j-1}, Y_{j-1}^h, r_{j-1}) \sigma_{\ell_2}(t_{j-1}, Y_{j-1}^h, r_{j-1}) \notag
		\\
		& \quad - \mathcal{D}_x \sigma_{\ell_1}(t_{j-1}, Z_{j-1}^h, r_{j-1}) \sigma_{\ell_2}(t_{j-1}, Z_{j-1}^h, r_{j-1})\big|^2dv ds \Big\|_{L^{q/2}(\Omega)} \notag
	\end{align}
	and further application of Remark \ref{rem:D} yields,
	\begin{align}
		\Big\|  \max_{n\in\{1, \ldots, k\}}  \big| \sum_{j=1}^n \mathcal{T}^{(3)}_j \big| \Big\|_{L^q(\Omega)}^2  \hspace{-0.8mm} \leq K  \sum_{j=1}^k \Big\|  h_j^2 \big|Y_{j-1}^h - Z_{j-1}^h\big|^2\Big\|_{L^{q/2}(\Omega)} \hspace{-0.8mm} \leq   K    \sum_{j=1}^k    h_j^2 \Big\|\max_{i\in \{0, \ldots, j-1\}} |Y_{i}^h-Z_{i}^h|  \Big\|_{L^{q}(\Omega)}^2   \label{eq:T3}	
	\end{align}
	for any $k \in \{1, \ldots, n_h\}$. 
	Moreover,  $\Big\{\displaystyle \sum_{j=1}^n \mathcal{T}^{(4)}_j\Big\}_{n\in \{1, \ldots, n_h\}}$ is a martingale with respect to the filtration $\{\tilde{\mathscr{F}}_{\uptau_{1}^{n} \wedge t_{n+1}}^B \vee \tilde{\mathscr{F}}_{T}^r \otimes \tilde{\mathscr{F}}_{n}^u\}_{n \in \{1, \ldots, n_h\}}$.  
	Thus, 
	\begin{align*}
		\Big\|  & \max_{n\in\{1, \ldots, k\}}  \big| \sum_{j=1}^n \mathcal{T}^{(4)}_j \big| \Big\|_{L^q(\Omega)}
		\\
		& =: \Big\|  \max_{n\in\{1, \ldots, k\}}   \big| \sum_{j=1}^n  \sum_{\ell=1}^{\tilde{d}} \mathbbm{1}_{\big\{N_{t_{j-1}}^{t_j}=1\big\}} \big\{ \sigma_{\ell}(t_{j-1},Y_{j-1}^h, r_j)-\sigma_{\ell}(t_{j-1},Z_{j-1}^h, r_j) \big\} \big\{ B_{\ell}(t_{j})-B_{\ell}(\uptau_{1}^{j-1})\big\} \big| \Big\|_{L^q(\Omega)} 
		\\
		& \leq K  \Big\|   \sum_{j=1}^k  \big|  \sum_{\ell=1}^{\tilde{d}} \mathbbm{1}_{\big\{N_{t_{j-1}}^{t_j}=1\big\}} \big\{ \sigma_{\ell}(t_{j-1},Y_{j-1}^h, r_j)-\sigma_{\ell}(t_{j-1},Z_{j-1}^h, r_j) \big\} \big\{ B_{\ell}(t_{j})-B_{\ell}(\uptau_{1}^{j-1}) \big\} \big|^2  \Big\|_{L^{q/2}(\Omega)}^{1/2} 
	\end{align*}
	which due to Minkowski's inequality and Assumption \ref{as:diffusion} yields, 
	\begin{align}
		\Big\|  & \max_{n\in\{1, \ldots, k\}}  \big| \sum_{j=1}^n \mathcal{T}^{(4)}_j \big| \Big\|_{L^q(\Omega)}^2   \notag
		\\
		& \leq  K  \sum_{j=1}^k   \Big\| \big|  \sum_{\ell=1}^{\tilde{d}} \mathbbm{1}_{\big\{N_{t_{j-1}}^{t_j}=1\big\}} \big\{ \sigma_{\ell}(t_{j-1},Y_{j-1}^h, r_j)-\sigma_{\ell}(t_{j-1},Z_{j-1}^h, r_j) \big\} \big\{ B_{\ell}(t_{j})-B_{\ell}(\uptau_{1}^{j-1})\big\} \big|^2  \Big\|_{L^{q/2}(\Omega)}  \notag
		\\
		& \leq   K \sum_{j=1}^k      \sum_{\ell=1}^{\tilde{d}} \Big\| \big| Y_{j-1}^h-Z_{j-1}^h\big| \big| B_{\ell}(t_{j})-B_{\ell}(\uptau_{1}^{j-1})\big|  \Big\|_{L^{q}(\Omega)}^2   \leq   K    \sum_{j=1}^k   h_j \Big\|\max_{i\in \{0, \ldots, j-1\}} |Y_{i}^h-Z_{i}^h|  \Big\|_{L^{q}(\Omega)}^2   \label{eq:T4}
	\end{align}
	for any $k \in \{1, \ldots, n_h\}$. 
	By applying similar arguments,  one arrive at the following estimate, 
	\begin{align}
		& \Big\| \max_{n\in\{1, \ldots, k\}}  \big| \sum_{j=1}^n \mathcal{T}^{(5)}_j \big| \Big\|_{L^q(\Omega)}^2    \notag
		\\
		& : = \Big\|   \hspace{-0.5mm} \max_{n\in\{1, \ldots, k\}}  \hspace{-0.5mm}   \big|  \hspace{-0.5mm} \sum_{j=1}^n  \hspace{-0.5mm}  \sum_{\ell=1}^{\tilde{d}}  \hspace{-0.5mm} \mathbbm{1}_{\big\{N_{t_{j-1}}^{t_j}=1\big\}} \hspace{-0.5mm}  \big\{  \sigma_{\ell}(t_{j-1},Y_{j-1}^h,  r_{j-1})-\sigma_{\ell}(t_{j-1},Z_{j-1}^h,  r_{j-1}) \big\}  \hspace{-0.5mm} \big\{ B_{\ell}(t_{j})-B_{\ell}(\uptau_{1}^{j-1})\big\}\big| \Big\|_{L^q(\Omega)}^2   \notag
		\\
		&  \leq   K    \sum_{j=1}^k   h_j \Big\|\max_{i\in \{0, \ldots, j-1\}} |Y_{i}^h-Z_{i}^h|  \Big\|_{L^{q}(\Omega)}^2   \label{eq:T5}
	\end{align}
	for any $k \in \{1, \ldots, n_h\}$. 
	
	On combining the estimates from Equations \eqref{eq:T1} to \eqref{eq:T5} in Equation \eqref{eq:sum:T1+T5},  we obtain
	\begin{align}
		\Big\|  \max_{n\in\{1, \ldots, k\}}   \big|\sum_{j=1}^n \big[\Xi^h_j(Y_{j-1}^h,  u_j, r_{j-1}^u,& r_{j-1}, r_j)- \Xi^h_j(Z_{j-1}^h, u_j, r_{j-1}^u, r_{j-1}, r_j)\big]\big| \Big\|_{L^q(\Omega)}^2  \notag 
		\\
		&  \leq   K    \sum_{j=1}^k   h_j \Big\|\max_{i\in \{0, \ldots, j-1\}} |Y_{i}^h-Z_{i}^h|  \Big\|_{L^{q}(\Omega)}^2   \notag
	\end{align}
	for any $k \in \{1, \ldots, n_h\}$.  
	This proves the first part which implies the second part of the lemma. 
\end{proof}
Notice that from Equations \eqref{eq:scheme},  \eqref{eq:Uep} and  \eqref{eq:incrment},  one can write, 
\begin{align}
	X_n^h- X_{0}^h - \sum_{j=1}^n \Xi_j^h(X_{j-1}^h , u_j, r_{j-1}^u, r_{j-1}, r_j)=0 \label{eq:resX}
\end{align}
almost surely for $n\in \{1, \ldots, n_h\}$. 
Clearly, $X^h \in G_q^h$.
For any grid process $Y^h \in G_q^h$,  
\begin{align}
	\sum_{j=1}^n  \mathcal{R}_j[Y^h]= Y_n^h- Y_{0}^h - \sum_{j=1}^n \Xi_j^h(Y_{j-1}^h , u_j, r_{j-1}^u,  r_{j-1}, r_j)  \label{eq:resY}
\end{align}
almost surely for $n\in \{1, \ldots, n_h\}$. 
Also, on using Equation \eqref{eq:resX} in Equation \eqref{eq:resY},  
\begin{align}
	\sum_{j=1}^n  \mathcal{R}_j[Y^h] = (Y_n^h- X_n^h) - (Y_{0}^h-X_{0}^h) - \sum_{j=1}^n \big[\Xi_j^h(Y_{j-1}^h , u_j, r_{j-1}^u, r_{j-1}, r_j)-\Xi_j^h(X_{j-1}^h , u_j, r_{j-1}^u, r_{j-1}, r_j)\big] \label{eq:res:sum}
\end{align}
almost surely for any $n \in \{1, \ldots, n_h\}$. 

The following theorem gives the bi-stability of the scheme \eqref{eq:scheme}. Here, we recall the norms $\| \cdot \|_{G^h_{q}}$ and $\| \cdot \|_{G^h_{q,S}}$ from Equation \eqref{eq:norm:spiker} for clarity. 
\begin{theorem}\label{thm:bi}
	Let Assumptions \ref{as:initial}, \ref{as:drift} and \ref{as:diffusion} hold. 
	Then, for any $ Y^h \in G_q^h$ with residual  $\mathcal{R}[Y^h] \in G^h_{q,S}$, $q \geq 2$ there exist positive constatnts $C_1$ and $C_2$ such that
	\begin{align*}
		C_1 \big\|\mathcal{R}[Y^h] \big\|_{G^h_{q,S}}\leq  \big\| Y^h-X^h \big\|_{G_q^h}   \leq C_2  \big\|\mathcal{R}[Y^h] \big\|_{G^h_{q,S}}.
	\end{align*}
\end{theorem}
\begin{proof}
	On using  Equations \eqref{eq:norm:spiker},  \eqref{eq:residual} and \eqref{eq:res:sum},     
	\begin{align*}
		\big\|\mathcal{R} & [Y^h] \big\|_{G^h_{q,S}}  = \big\| \mathcal{R}_0 [Y^h] \big\|_{L^q(\Omega)}+ \Big\|\max_{n\in \{1,\ldots, n_h \}} |\sum_{j=1}^n \mathcal{R}_j [Y^h]| \Big\|_{L^q(\Omega)}
		\\
		& \leq  2 \big\| Y^h_0-X_0^h \big\|_{L^q(\Omega)} + \Big\|  \max_{n\in \{1,\ldots, n_h \}} |Y_n^h- X_n^h|  \Big\|_{L^q(\Omega)}
		\\
		&\quad + \Big\| \max_{n\in \{1,\ldots, n_h \}} \Big| \sum_{j=1}^n \big[\Xi_j^h(Y_{j-1}^h , u_j,r_{j-1}^u, r_{j-1}, r_j)-\Xi_j^h(X_{j-1}^h , u_j, r_{j-1}^u, r_{j-1}, r_j)\big] \Big| \Big\|_{L^q(\Omega)}
		\\
		& \leq K \big\| Y^h-X^h \big \|_{G_q^h}  + \Big\| \max_{n\in \{1,\ldots, n_h \}} \Big| \sum_{j=1}^n \big[\Xi_j^h(Y_{j-1}^h , u_j, r_{j-1}^u, r_{j-1}, r_j)-\Xi_j^h(X_{j-1}^h , u_j, r_{j-1}^u, r_{j-1}, r_j)\big] \Big| \Big\|_{L^q(\Omega)}
	\end{align*}
	which on using Lemma \ref{lem:E} yields,
	\begin{align}
		\label{eq:bileft}
		C_1 \big\|\mathcal{R} [Y^h] \big\|_{G^h_{q,S}}  \leq  \big\| Y^h-X^h \big\|_{G_q^h}
	\end{align} 
	Again, re-arranging Equation \eqref{eq:res:sum} as, 
	\begin{align*}
		Y_n^h- X_n^h =  Y_{0}^h-X_{0}^h  + \sum_{j=1}^n  \mathcal{R}_j[Y^h]+  \sum_{j=1}^n \big[\Xi_j^h(Y_{j-1}^h , u_j, r_{j-1}^u, r_{j-1}, r_j)-\Xi_j^h(X_{j-1}^h , u_j, r_{j-1}^u, r_{j-1}, r_j)\big]
	\end{align*}
	for any $n \in \{1, \ldots, n_h\}$ and then Minkowski's inequality and Lemma \ref{lem:E} yields, 
	\begin{align*}
		\Big\| & \max_{n \in \{0, \ldots, k\}}|Y_n^h- X_n^h | \Big\|_{L^q (\Omega)}^2  \leq K  \big\| \mathcal{R}_0[Y^h] \big\|_{L^q(\Omega)}^2  +  K  \Big \| \max_{n \in \{1,   \ldots, k\}}\big|\sum_{j=1}^n  \mathcal{R}_j[Y^h]\big| \Big\|_{L^q(\Omega)}^2 
		\\
		& \qquad + K \Big \| \max_{n \in \{1, \ldots, k\}} \big| \sum_{j=1}^n \big[\Xi_j^h(Y_{j-1}^h , u_j, r_{j-1}^u, r_{j-1}, r_j)-\Xi_j^h(X_{j-1}^h , u_j, r_{j-1}^u, r_{j-1}, r_j)\big] \big| \Big\|_{L^q(\Omega)}^2
		\\
		& \leq K   \big\|\mathcal{R}[Y^h] \big\|_{G^h_{q,S}}^2+  K    \sum_{j=1}^k   h_j \Big\|\max_{i\in \{0, \ldots, j-1\}} |Y_{i}^h-Z_{i}^h|  \Big\|_{L^{q}(\Omega)}^2   \notag
	\end{align*}
	for any $k \in \{1, \ldots, n_h\}$.  
	Further, due to Gr\"onwall's inequality,  one has, 
	\begin{align}
		\label{eq:biright}
		\Big\|  \max_{n \in \{0, \ldots, n_h\}}|Y_n^h- X_n^h | \Big\|_{L^q (\Omega)}^2 \leq K  \big\|\mathcal{R}[Y^h] \big\|_{G^h_{q,S}}^2 \exp\big(K   \sum_{j=1}^{n_h} h_j\big) \leq C_2  \big\|\mathcal{R}[Y^h] \big\|_{G^h_{q,S}}.
	\end{align}
	The proof is complete when Equation \eqref{eq:bileft} is combined with Equation \eqref{eq:biright}.
\end{proof}
In order to establish the consistency of the scheme \eqref{eq:scheme} and hence to recover its rate of strong convergence,  we require some lemmas as given below. 
The following lemma is shown in \cite[Lemma 4.1]{nguyen2017milstein},  see also  \cite[Lemma 5]{kumar2021note}. 
\begin{lemma}
	\label{lem:jump}
	Define  $\mathfrak{q} := \max\{-\mathfrak{q}_{k_{0}k_{0}}: k_{0} \in S\}$ and let $N_{s}^{t}$ be the number of switchings of the Markov chain $r$ in the interval $(s,t]$ for $0 \leq s < t \leq T$. Then, 
	$$	
	\mathbb{P} (N_{s}^{t} \geq N) \leq \mathfrak{q}^{N}(t-s)^{N},  \, N \in \mathbb{N}. 
	$$
	If $0<(t-s) < 1/(2\mathfrak{q})$ holds, then there is a constant $K>0$ independent of $t-s$ such that  
	$$
	\mathbb{E}[N_{s}^{t}] \leq K (t-s) \mbox{ and } \mathbb{E}[(N_{s}^{t})^2] \leq K.  
	$$
\end{lemma}

The proof of the following lemma is similar to \cite[Lemma 2.2]{nguyen2017milstein}.
However,  we additionally include an $L^2(\tilde{\Omega})$-estimate.
\begin{lemma}\label{lem:Ito}
	Let  $\uptau_1 < \uptau_2 < \ldots < \uptau_{N_s^t} $ be the switching times  of the Markov chain $r$ in the interval $(s,t]$ for any $0 \leq s < t \leq T$ where $t$ may or may not be the switching time of the chain.
	Define $\uptau_0 := s$ and $ \uptau_{ N_s^t +1} := t$. 
	For every $i_0 \in S$, if  $g:[0,T] \times \mathbb{R}^d \times \{i_0\} \to \mathbb{R}^d$ is a Borel measurable function,  then
	\begin{align*}
		g(t,X(t),r(t)) - & g(s,X(s),r(s)) =  \sum_{i_0 \neq j_0} \int_s^t [g(v,X(v),j_0) - g(v,X(v),i_0)]dM_{i_0 j_0}(v)\\ \nonumber
		& + \sum_{j_0 \in S} \int_s^t q_{r(v-)j_0}[g(v,X(v),j_0) - g(v,X(v),r(v-))]dv\\ \nonumber
		& + \sum_{k=0}^{N_s^t} [g(\uptau_{k+1},X(\uptau_{k+1}),r(\uptau_k)) - g(\uptau_k,X(\uptau_k),r(\uptau_k))]
	\end{align*}
	almost surely for any $0 \leq s < t \leq T$.
	In addition,  if $g$ is ${1/2}-$H\"older continuous in time and  Lipschitz continuous in state,   \textit{i.e.},  there exists a constant $C>0$ such that 
	\begin{align*}
		\big|g(t, x, j_0)- g(s, \bar x, j_0) \big| \leq C \{(1+|x|)|t-s|^{1/2}+ |x-\bar{x}|\}
	\end{align*} 
	for all $t, s \in [0,T]$, $x, \bar{x}\in \mathbb{R}^d$ and $j_0 \in S$,  then there is a constant $K>0$, independent of $s$ and $t$ such that
	\begin{align*}
		\big\|g(t, X(t), r(t))- g(s, X(s), r(s)) \big\|_{L^2(\tilde{\Omega})} \leq K |t-s|^{1/2}.
	\end{align*}
\end{lemma}
\begin{proof}
	One can clearly write the following
	\begin{align*} 
		g(t,X(t),r(t)) - & g(s,X(s),r(s)) =  \sum_{k=0}^{N_{s}^{t}} [g(\uptau_{k+1},X(\uptau_{k+1}),r(\uptau_{k+1})) - g(\uptau_k,X(\uptau_k),r(\uptau_k))] \notag
		\\
		= & \sum_{k=0}^{N_{s}^{t}} [g(\uptau_{k+1},X(\uptau_{k+1}),r(\uptau_{k+1})) - g(\uptau_{k+1},X(\uptau_{k+1}),r(\uptau_{k}))]  \notag
		\\
		& + \sum_{k=0}^{N_{s}^{t}} [g(\uptau_{k+1},X(\uptau_{k+1}),r(\uptau_{k})) - g(\uptau_k,X(\uptau_k),r(\uptau_k))]   \notag
	\end{align*}
	almost surely for any  $0 \leq s < t \leq T$. 
	Now, using $M_{i_0 j_0}(t) = [M_{i_0 j_0}](t) - \langle M_{i_0 j_0}\rangle (t) $ and   $r(\uptau_k) = r(u)$ for $u \in (\uptau_k,\uptau_{k+1})$, $k \in \{1,\ldots, N_s^t\}$,  one can write, 
	\begin{align*}
		&\sum_{k=0}^{N_{s}^{t}} [g(\uptau_{k+1}, X(\uptau_{k+1}),r(\uptau_{k+1})) - g(\uptau_{k+1},X(\uptau_{k+1}),r(\uptau_{k+1}-))] \\
		& = \sum_{i_0 \neq j_0} \int_s^t [g(v,X(v),j_0)-g(v,X(v),i_0)] d[M_{i_0 j_0}](v)
		\\
		& =  \sum_{i_0 \neq j_0} \int_s^t [g(v,X(v),j_0) - g(v,X(v),i_0)]dM_{i_0 j_0}(v) + \sum_{j_0 \in S} \int_s^t q_{r(v-)j_0}[g(v,X(v),j_0) - g(v,X(v),r(v-))]dv
	\end{align*}
	almost surely, which on substituting completes the first part of the  proof.
	For the second part, using Minkowski inequality,  H\"older's inequality and Theorem \ref{th:eu} and Lemma \ref{lem:jump}, 
	\begin{align*}
		\big\|g(t, &X(t),   r(t))   \hspace{-0.5mm} - g(s, X(s), r(s)) \big\|_{L^2(\tilde{\Omega})} \hspace{-0.5mm} \leq \hspace{-0.5mm}   K \sum_{i_0 \neq j_0}  \hspace{-0.5mm} \Big( \tilde{\mathbb{E}} \Big[ \int_s^t \big| g(v,X(v),j_0) - g(v,X(v),i_0)\big|^2  d[M_{i_0 j_0}](v)\Big] \Big)^{1/2}
		\\
		& \quad + K \Big( |t-s|\tilde{\mathbb{E}}  \Big[ \sum_{j_0 \in S} \int_s^t (q_{r(v-)j_0})^2\big|g(v,X(v),j_0) - g(v,X(v),r(v-))\big|^2 dv \Big]\Big)^{1/2}
		\\
		& \quad + K \Big( \tilde{\mathbb{E}} \Big[ (1+N_{s}^{t}) \sum_{k=0}^{N_{s}^{t}} \big|g(\uptau_{k+1},X(\uptau_{k+1}),r(\uptau_k)) - g(\uptau_k,X(\uptau_k),r(\uptau_k)) \big|^2 \Big]\Big)^{1/2}
		\\
		& \leq K \sum_{i_0 \neq j_0}  \Big( \tilde{\mathbb{E}} \Big[ \int_s^t (1+\tilde{\mathbb{E}}[|X(v)|^2|\mathcal{F}_T^r])  d[M_{i_0 j_0}](v) \Big]\Big)^{1/2}
		\\
		& \ + K \Big( |t-s| \tilde{\mathbb{E}}  \Big[ \sum_{j_0 \in S} \int_s^t (q_{r(v-)j_0})^2 (1+\tilde{\mathbb{E}}[|X(v)|^2|\mathcal{F}_T^r]) dv \Big]\Big)^{1/2}
		\\
		& \ +  K \Big( \tilde{\mathbb{E}} \Big[ (1+N_{s}^{t}) \sum_{k=0}^{N_{s}^{t}} \big\{ |\uptau_{k+1}-\uptau_{k}| (1+\tilde{\mathbb{E}}[|X(\uptau_k)|^2|\mathcal{F}_T^r])+ \tilde{\mathbb{E}}[|X(\uptau_{k+1}) - X(\uptau_k)|^2|\mathcal{F}_T^r]\big\} \Big]\Big)^{1/2}
		\\
		& \leq K \sum_{i_0 \neq j_0}  \Big( \tilde{\mathbb{E}} \big[ [M_{i_0 j_0}](t) - [M_{i_0j_0}](s) \big] \Big)^{1/2} + K|t-s|  + K \Big( \tilde{\mathbb{E}} \Big[ (1+N_{s}^{t}) \sum_{k=0}^{N_{s}^{t}} |\uptau_{k+1}-\uptau_{k}| \Big]\Big)^{1/2}
		\\
		& \leq K |t-s|^{1/2}
	\end{align*}
	for all $s, t \in [0, T]$. 
\end{proof}
The proof of the following lemma  on the randomized quadrature rule for stochastic processes can be found in \cite{Kruse2019}. 
For this,  let us recall the sequence of uniform random variables $u:=\{u_j\}_{j\in \mathbb{N}}$ from Equation \eqref{eq:uniform},  the temporal grid  $\varrho_h$ from Equation \eqref{eq:grid} and the notations from Equation~\eqref{eq:tju}. 
\begin{lemma} \label{lem:quadrature}
	Let $\tilde{p} \geq 2$ and $Y: [0,T]\times  \tilde{\Omega} \mapsto \mathbb{R}^{d}$ be a stochastic process such that 
	$
	\tilde{\mathbb{E}} \Big[ \hspace{-0.5mm}\displaystyle  \hspace{-0.5mm} \int_0^T |Y(t)|^{\tilde{p}} dt \Big] \hspace{-0.5mm} <~\infty.  
	$
	For every $n \in \{1, \ldots, n_h\}$,   the randomized Riemann sum approximation  of  $ \displaystyle \int_0^{t_n} Y(t) dt$,  defined~as 
	$$
	\mathcal{A}^n_{u}[Y]:= \displaystyle  \sum_{j=1}^n h_j Y(t_{j-1}^u) 
	$$
	almost surely satisfies the following; 
	\newline 
	\textnormal{(i)} $\mathcal{A}^n_{u}[Y] \in L^{\tilde{p}}(\Omega)$, 
	\newline 
	\textnormal{(ii)} $\tilde{\mathbb E}^u \big[ \mathcal{A}^n_{u}[Y] \big]= \displaystyle \int_0^{t_n} Y(s) ds \in L^{\tilde{p}} (\Omega)$,  and 
	\newline 
	\textnormal{(iii)}  $\displaystyle \Big\| \max_{n \in \{1, \ldots, n_h\}} \big| \mathcal{A}^n_{u}[Y]-\int_0^{t_n} Y(s) ds\big|\Big\|_{L^{\tilde{p}}(\Omega)} \leq 2C_{\tilde{p}} T^{\frac{{\tilde{p}}-2}{2\tilde{p}}} \Big(\tilde{\mathbb{E}} \Big[ \int_0^T |Y(s)|^{\tilde{p}} ds  \Big] \Big)^{1/\tilde{p}} h^{1/2}$ where $C_{\tilde{p}}$ is a constant only depending on $\tilde{p}$. 
	In addition,  if 
	\begin{align}
		\|Y\|_{\gamma}=\displaystyle \sup_{t \in [0,T]}\|Y(t)\|_{L^{\tilde{p}}(\Omega)} + \sup_{ s \neq t\in [0, T]} \frac{\|Y(t)- Y(s)\|_{L^{\tilde{p}}(\Omega)}}{|t-s|^{\gamma}}< \infty \notag
	\end{align}
	for some $\gamma \in (0,1)$,   then 
	\begin{align*}
		\Big\| \max_{n \in \{1, \ldots, n_h\}} \big| \mathcal{A}^n_{u}[Y]-\int_0^{t_n} Y(s) ds\big|\Big\|_{L^{\tilde{p}}(\Omega)} \leq C_{\tilde{p}} \sqrt{T}  \|Y\|_\gamma \,h^{\gamma+1/2}.  
	\end{align*}
\end{lemma} 

In what follows,  denote  the solution process $X$ of SDEwMS \eqref{eq:sdems}  restricted on the time-grid  $\varrho_h$ defined in Equation \eqref{eq:grid}  by 
\begin{align}
	\label{eq:Xrho}
	X^{\varrho_h}:= (X_{0}, \ldots, X_{n_h})
\end{align}
where  $X_j:=X(t_j) \in L^{p}(\Omega, \mathscr F_j^h, \mathbb{P}; \mathbb{R}^d)$ for all $j\in \{0,\ldots, n_h\}$ due to Theorem  \ref{th:eu}. Clearly, $X^{\varrho_h} \in G_q^h$.  

Due to Lemma \ref{lem:Ito}, $b$ satisfies the condition of Lemma \ref{lem:quadrature} with $\gamma=1/2$. Thus, the following corollary is a consequence of Lemmas \ref{lem:Ito} and \ref{lem:quadrature}.

\begin{corollary} 
	\label{cor:quadrature_r(j-1)}
	Let Assumptions \ref{as:initial} to \ref{as:diffusion} hold. Then, there is a constant $K>0$, independent of $h$, such that 
	\begin{align*}
		\Big\| \max_{n\in \{1,\ldots, n_h \}} \big|\sum_{j=1}^{n} \int_{t_{j-1}}^{t_j} [b(s,X(s),r(s))-b( t_{j-1}^u,X( t_{j-1}^u),r_{j-1}^u)]ds \big| \Big\|_{L^2(\Omega)} \leq K h. \notag
	\end{align*}
\end{corollary}
\begin{lemma}
	\label{lem:J1}
	Let Assumptions \ref{as:initial} to \ref{as:diffusion} be satisfied. Then, there exists a constant $K>0$ such that
	\begin{align*}
		\Big\|  \max_{n\in \{1,\ldots, n_h \}} \big|  \sum_{j=1}^{n} \int_{t_{j-1}}^{t_j} \big[ b(s, X(s), r(s))-b( t_{j-1}^u, \Upsilon^h_j(X_{j-1}, u_j, r_{j-1}),  r_{j-1}^u) \big] ds\big|  \Big\|_{L^2(\Omega)} \leq K h.
	\end{align*}
\end{lemma}
\begin{proof}
	Clearly,  one can compute
	\begin{align*}
		&\Big\|  \max_{n\in \{1,\ldots, n_h \}} \big|  \sum_{j=1}^{n} \int_{t_{j-1}}^{t_j} \big[ b(s, X(s), r(s))-b( t_{j-1}^u, \Upsilon^h_j(X_{j-1}, u_j, r_{j-1}),  r_{j-1}^u) \big] ds\big|  \Big\|_{L^2(\Omega)}
		\\
		& \leq  \Big\|  \max_{n\in \{1,\ldots, n_h \}} \big|  \sum_{j=1}^{n} \int_{t_{j-1}}^{t_j} \big[ b(s, X(s), r(s))-b( t_{j-1}^u, X( t_{j-1}^u),   r_{j-1}^u) \big] ds\big|  \Big\|_{L^2(\Omega)}
		\\
		& + \Big\|  \max_{n\in \{1,\ldots, n_h \}} \big|  \sum_{j=1}^{n} \int_{t_{j-1}}^{t_j} \big[ b( t_{j-1}^u, X( t_{j-1}^u),   r_{j-1}^u)  -b( t_{j-1}^u, \Upsilon^h_j(X_{j-1}, u_j, r_{j-1}),  r_{j-1}^u) \big] ds\big|  \Big\|_{L^2(\Omega)}. 
	\end{align*}
	The first term in the above equation can be estimated by Corollary \ref{cor:quadrature_r(j-1)}. For the second term,  one can apply Assumption~\ref{as:drift} and Minkowski's inequality to obtain, 
	\begin{align}
		\Big\|  \max_{n\in \{1,\ldots, n_h \}}  & \big|  \sum_{j=1}^{n} \int_{t_{j-1}}^{t_j} \big[ b(s, X(s), r(s))-b( t_{j-1}^u, \Upsilon^h_j(X_{j-1}, u_j, r_{j-1}),  r_{j-1}^u) \big] ds\big|  \Big\|_{L^2(\Omega)} \notag
		\\
		& \leq K h + K   \sum_{j=1}^{n_h} h_j \big\|  X( t_{j-1}^u)-\Upsilon^h_j(X_{j-1}, u_j, r_{j-1}) \big\|_{L^2(\Omega)}. \label{eq:j1}
	\end{align}
	Further,  for the second term on the right side of Equation \eqref{eq:j1},  one uses Equations  \eqref{eq:sdems} and  \eqref{eq:Uep} to obtain the following, 
	\begin{align*}
		X( t_{j-1}^u) -\Upsilon^h_j(X_{j-1}, u_j, r_{j-1}) &=  X( t_{j-1}^u)- X_{j-1} - b(t_{j-1}, X_{j-1},  r_{j-1})   h_j u_j  
		\\
		& \quad - \sum_{\ell=1}^{\tilde{d}} \sigma_{\ell}(t_{j-1}, X_{j-1},  r_{j-1}) \int_{t_{j-1}}^{ t_{j-1}^u} dB_{\ell} (s)
		\\
		& = \int_{t_{j-1}}^{ t_{j-1}^u} [ b(s, X(s), r(s)) -  b(t_{j-1}, X_{j-1},  r_{j-1}) ] ds 
		\\
		& \quad +\sum_{\ell=1}^{\tilde{d}} \int_{t_{j-1}}^{ t_{j-1}^u} [\sigma_{\ell}(s, X(s), r(s)) - \sigma_{\ell}(t_{j-1}, X_{j-1},  r_{j-1}) ]dB_{\ell}(s)
	\end{align*}
	which on using Minkowski and H\"older's inequalities along with $u_j\leq 1$ yields, 
	\begin{align}
		\big\|  X(t_{j-1}&+h_j u_j)    -\Upsilon^h_j(X_{j-1}, u_j, r_{j-1})\big\|_{L^2(\Omega)} \notag
		\\
		& \leq  h_{j}^{1/2} \Big(  \int_{t_{j-1}}^{ t_{j}} \mathbb{E} \big[ | b(s, X(s), r(s)) -  b(t_{j-1}, X_{j-1},  r_{j-1}) |^2 \big] ds   \Big)^{1/2}\notag
		\\
		& \quad +  \sum_{\ell=1}^{\tilde{d}} \Big( \int_{t_{j-1}}^{ t_{j}} \mathbb{E} \big[ | \sigma_{\ell}(s, X(s), r(s)) - \sigma_{\ell}(t_{j-1}, X_{j-1},  r_{j-1}) |^2 \big] ds  \Big)^{1/2} \notag
	\end{align}
	for $j \in \{1, \ldots, n_h\}$. 
	Also,  the implementation of Lemma \ref{lem:Ito} leads to the following estimates,  
	\begin{align}
		\big\|  X(t_{j-1}&+h_j u_j)    -\Upsilon^h_j(X_{j-1}, u_j, r_{j-1})\big\|_{L^2(\Omega)} \leq K h \notag
	\end{align}
	for $j \in \{1, \ldots, n_h\}$ which on substituting in Equation \eqref{eq:j1} completes the proof. 
\end{proof}
Now,  using Lemma \ref{lem:Ito} and $M_{j_0 k_0}= [M_{j_0 k_0}]- \langle M_{j_0 k_0} \rangle $, one can obtain
\begin{align}
	\int_{t_{j-1}}^{t_j}  [\sigma_{\ell}(s&, X(s), r(s)) -   \sigma_{\ell}(t_{j-1} , X_{j-1}, r_{j-1}) ] dB_{\ell}(s)   \notag
	\\
	& = \int_{t_{j-1}}^{t_j} \sum_{j_0 \neq k_0} \int_{t_{j-1}}^s (\sigma_{\ell}(v,  X(v), k_0)-\sigma_{\ell}(v,  X(v), j_0)) d[M_{j_0 k_0}](v)dB_{\ell}(s)  \notag
	\\
	&\quad  - \int_{t_{j-1}}^{t_j} \sum_{j_0 \neq k_0} \int_{t_{j-1}}^s (\sigma_{\ell}(v,  X(v), k_0)-\sigma_{\ell}(v,  X(v), j_0)) d\langle M_{j_0 k_0}\rangle(v)dB_{\ell}(s)  \notag
	\\
	& \quad + \int_{t_{j-1}}^{t_j}\sum_{k_0 \in S} \int_{t_{j-1}}^s q_{r(v-)k_0} \big[\sigma_{\ell}(v,  X(v), k_0)-\sigma_{\ell}(v,  X(v), r(v-)) \big]dvdB_{\ell}(s) \notag
	\\
	& \quad + \int_{t_{j-1}}^{t_j} \sum_{k=0}^{N_{t_{j-1}}^{s}} \big(\sigma_{\ell}(\uptau_{k+1}, X(\uptau_{k+1}), r(\uptau_k))- \sigma_{\ell}(\uptau_{k}, X(\uptau_{k}),r(\uptau_k))\big) dB_{\ell}(s) \label{eq:Jsigma}
\end{align}
almost surely for any $j \in \{1, \ldots, n_h\}$ and $\ell \in \{1, \ldots, \tilde{d}\}$
where $t_{j-1}=\uptau_0<\uptau_1<\uptau_2 < \cdots< \uptau_{N_{t_{j-1}}^s}< \uptau_{N_{t_{j-1}}^s+1}=s$ are switching times of the chain $r$ in the interval $(t_{j-1}, s]$. 

By using \cite[Lemma 4.3]{nguyen2017milstein} and observing that $r(\uptau_{1}^{j-1}) = r(t_{j})$ on the set $\{N_{t_{j-1}}^{t_j} = 1 \}$,  the first term on the right side of Equation \eqref{eq:Jsigma} can be written as, 
\begin{align}
	\int_{t_{j-1}}^{t_j} &\sum_{j_0 \neq k_0} \int_{t_{j-1}}^s (\sigma_{\ell}(v,  X(v), k_0)-\sigma_{\ell}(v,  X(v), j_0)) d[M_{j_0 k_0}](v)dB_{\ell}(s) \notag
	\\
	& =  \mathbbm{1}_{\big\{N_{t_{j-1}}^{t_j}=1\big\}} \big(  \sigma_{\ell}(\uptau_{1}^{j-1},X(\uptau_{1}^{j-1}), r_j) - \sigma_{\ell}(\uptau_{1}^{j-1},X(\uptau_{1}^{j-1}), r_{j-1}) \big) \big( B_{\ell}(t_{j})-B_{\ell}(\uptau_{1}^{j-1})\big)  \notag
	\\
	& \quad + \int_{t_{j-1}}^{t_j} \mathbbm{1}_{\big\{N_{t_{j-1}}^{t_j}\geq 2\big\}}\sum_{j_0 \neq k_0} \int_{t_{j-1}}^s (\sigma_{\ell}(v,  X(v), k_0)-\sigma_{\ell}(v,  X(v), j_0)) d[M_{j_0 k_0}](v)dB_{\ell}(s) \label{eq:Jsigma1}
\end{align}
almost surely for any $j \in \{1, \ldots, n_h\}$ and $\ell \in \{1, \ldots, \tilde{d}\}$. 

For the last term on the right side of Equation \eqref{eq:Jsigma},  one uses Equation \eqref{eq:sdems} to obtain, 
\begin{align}
	\int_{t_{j-1}}^{t_j} &\sum_{k=0}^{N_{t_{j-1}}^{s}} \big(\sigma_{\ell}(\uptau_{k+1}, X(\uptau_{k+1}), r(\uptau_k))- \sigma_{\ell}(\uptau_{k}, X(\uptau_{k}),r(\uptau_k))\big) dB_{\ell}(s)   \notag
	\\
	& =\int_{t_{j-1}}^{t_j} \sum_{k=0}^{N_{t_{j-1}}^{s}} \big(\sigma_{\ell}(\uptau_{k+1}, X(\uptau_{k+1}), r(\uptau_k))- \sigma_{\ell}(\uptau_{k}, X(\uptau_{k}),r(\uptau_k))  \notag
	\\
	& \qquad - \mathcal{D}_x \sigma_{\ell}(\uptau_k,X(\uptau_k),r(\uptau_k))(X(\uptau_{k+1}) - X(\uptau_k))\big)dB_{\ell}(s)  \notag
	\\
	& \quad + \int_{t_{j-1}}^{t_j} \sum_{k=0}^{N_{t_{j-1}}^{s}} \int_{\uptau_k}^{\uptau_{k+1}} \mathcal{D}_x \sigma_{\ell}(\uptau_k,X(\uptau_k),r(\uptau_k)) b(v,X(v),r(v)) dvdB_{\ell}(s) \notag
	\\
	& +  \int_{t_{j-1}}^{t_j}\sum_{k=0}^{N_{t_{j-1}}^{s}} \int_{\uptau_k}^{\uptau_{k+1}} \mathcal{D}_x \sigma_{\ell}(\uptau_k,X(\uptau_k),r(\uptau_k)) \sum_{\ell_2=1}^{\tilde{d}}\sigma_{\ell_2}(v,X(v),r(v)) dB_{\ell_{2}}(v) dB_{\ell}(s) \label{eq:Jsigma4}
\end{align}
almost surely for any $j \in \{1, \ldots, n_h\}$ and $\ell \in \{1, \ldots, \tilde{d}\}$.  

Moreover,  one can also write
\begin{align}
	\int_{t_{j-1}}^{t_{j}} & \int_{t_{j-1}}^{s}  \mathcal{D}_x \sigma_{\ell_1}(t_{j-1}, X_{j-1}, r_{j-1}) \sigma_{\ell_2}(t_{j-1}, X_{j-1}, r_{j-1}) dB_{\ell_{2}}(v)dB_{\ell_{1}}(s)  \notag
	\\
	& = \int_{t_{j-1}}^{t_{j}} \sum_{k=0}^{N_{t_{j-1}}^{s}} \int_{\uptau_k}^{\uptau_{k+1}}  \mathcal{D}_x \sigma_{\ell_1}(t_{j-1}, X_{j-1}, r_{j-1}) \sigma_{\ell_2}(t_{j-1}, X_{j-1}, r_{j-1}) dB_{\ell_{2}}(v)dB_{\ell_{1}}(s)  \label{eq:j22}
\end{align}
almost surely for any $j \in \{1, \ldots, n_h\}$ and $\ell_1, \ell_2 \in \{1, \ldots, \tilde{d}\}$. 

Now,  let us  substitute  Equations \eqref{eq:Jsigma1} and \eqref{eq:Jsigma4} in Equation \eqref{eq:Jsigma} and then use Equation \eqref{eq:j22} to obtain, 
\begin{align}
	& \sum_{\ell=1}^{\tilde{d}} \int_{t_{j-1}}^{t_j} \big[\sigma_{\ell}(s, X(s), r(s)) -  \sigma_{\ell}(t_{j-1}, X_{j-1}, r_{j-1})\big] dB_{\ell}(s)   \notag  
	\\
	& 
	\quad - \sum_{\ell_{1},\ell_{2}=1}^{\tilde{d}}\int_{t_{j-1}}^{t_{j}} \int_{t_{j-1}}^{s}  \mathcal{D}_x \sigma_{\ell_1}(t_{j-1}, X_{j-1}, r_{j-1}) \sigma_{\ell_2}(t_{j-1}, X_{j-1}, r_{j-1}) dB_{\ell_{2}}(v)dB_{\ell_{1}}(s) \notag
	\\
	&
	\quad - \sum_{\ell=1}^{\tilde{d}} \mathbbm{1}_{\big\{N_{t_{j-1}}^{t_j}=1\big\}} \big( \sigma_{\ell}(t_{j-1},X_{j-1}, r_j)- \sigma_{\ell}(t_{j-1},X_{j-1},  r_{j-1}) \big) \big( B_{\ell}(t_{j})-B_{\ell}(\uptau_{1}^{j-1})\big) \notag
	\\
	& =   \sum_{\ell=1}^{\tilde{d}} \mathbbm{1}_{\big\{N_{t_{j-1}}^{t_j}=1\big\}} \big(  \sigma_{\ell}(\uptau_{1}^{j-1},X(\uptau_{1}^{j-1}), r_j) - \sigma_{\ell}(\uptau_{1}^{j-1},X(\uptau_{1}^{j-1}), r_{j-1}) \notag
	\\
	& \qquad - \sigma_{\ell}(t_{j-1},X_{j-1}, r_j) + \sigma_{\ell}(t_{j-1},X_{j-1},  r_{j-1}) \big) \big( B_{\ell}(t_{j})-B_{\ell}(\uptau_{1}^{j-1})\big) \notag 
	\\
	& \quad + \sum_{\ell=1}^{\tilde{d}} \int_{t_{j-1}}^{t_{j}} \mathbbm{1}_{\big\{N_{t_{j-1}}^{t_j} \geq 2\big\}} \sum_{j_0 \neq k_0} \int_{t_{j-1}}^s (\sigma_{\ell}(v,  X(v), k_0)-\sigma_{\ell}(v,  X(v), j_0)) d[M_{j_0 k_0}](v) dB_{\ell}(s) \notag
	\\
	& \quad -  \sum_{\ell=1}^{\tilde{d}} \int_{t_{j-1}}^{t_{j}}  \sum_{j_0 \neq k_0} \int_{t_{j-1}}^s (\sigma_{\ell}(v,  X(v), k_0)-\sigma_{\ell}(v,  X(v), j_0)) d\langle M_{j_0 k_0}\rangle (v) dB_{\ell}(s) \notag
	\\
	&\quad + \sum_{\ell=1}^{\tilde{d}} \int_{t_{j-1}}^{t_{j}} \sum_{k_0 \in S} \int_{t_{j-1}}^s q_{r(v-)k_0} \big[\sigma_{\ell}(v,  X(v), k_0)-\sigma_{\ell}(v,  X(v), r(v-)) \big]dv dB_{\ell}(s) \notag 
	\\
	& \quad +  \sum_{\ell=1}^{\tilde{d}} \int_{t_{j-1}}^{t_{j}} \sum_{k=0}^{N_{t_{j-1}}^{s}} \big[\sigma_{\ell}(\uptau_{k+1}, X(\uptau_{k+1}), r(\uptau_k))- \sigma_{\ell}(\uptau_{k}, X(\uptau_{k}),r(\uptau_k)) \notag
	\\
	& \qquad - \mathcal{D}_x \sigma_{\ell}(\uptau_k,X(\uptau_k),r(\uptau_k))(X(\uptau_{k+1}) - X(\uptau_k)) \big] dB_{\ell}(s) \notag 
	\\
	& \quad +  \sum_{\ell=1}^{\tilde{d}}\int_{t_{j-1}}^{t_{j}} \sum_{k=0}^{N_{t_{j-1}}^{s}} \int_{\uptau_k}^{\uptau_{k+1}} \mathcal{D}_x \sigma_{\ell}(\uptau_k,X(\uptau_k),r(\uptau_k))b(v,X(v),r(v)) dv dB_{\ell}(s) \notag
	\\
	& 
	\quad + \sum_{\ell_{1},\ell_{2}=1}^{\tilde{d}}\int_{t_{j-1}}^{t_{j}} \sum_{k=0}^{N_{t_{j-1}}^{s}} \int_{\uptau_k}^{\uptau_{k+1}} \mathcal{D}_x \sigma_{\ell_1}(\uptau_k,X(\uptau_k),r(\uptau_k))\sigma_{\ell_2}(v,X(v),r(v)) \notag
	\\
	& \qquad - \mathcal{D}_x \sigma_{\ell_1}(t_{j-1}, X_{j-1}, r_{j-1}) \sigma_{\ell_2}(t_{j-1}, X_{j-1}, r_{j-1}) dB_{\ell_{2}}(v)dB_{\ell_{1}}(s) \notag
	\\
	& 
	=: \mathcal{I}_j^{(1)} + \mathcal{I}_j^{(2)} + \mathcal{I}_j^{(3)} + \mathcal{I}_j^{(4)} + \mathcal{I}_j^{(5)} + \mathcal{I}_j^{(6)} + \mathcal{I}_j^{(7)}
	\label{eq:I1-7}
\end{align}
almost surely for any $j \in \{1, \ldots, n_h\}$. 
\begin{lemma}
	\label{lem:I1}
	Let Assumptions \ref{as:initial} to \ref{as:diffusion} be satisfied. Then,
	\begin{align*}
		\sum_{i=1}^7 \Big\| \max_{n\in \{1,\ldots, n_h \}} \big| \sum_{j=1}^{n} \mathcal{I}_j^{(i)} \big| \Big\|_{L^2(\Omega)} \leq Kh.
	\end{align*}
\end{lemma}
\begin{proof}
	
	The estimate of the terms involving $\mathcal{I}^1$, $\mathcal{I}^2$ and $\mathcal{I}^3$ follow from the estimates of the remainder terms $R_{p,7}$, $R_{p,8}$ and $R_{p,5}$  in \cite[Equation 4.21]{nguyen2017milstein}, in Lemmas 4.5, 4.6 and 4.7, respectively.
	Moreover, one can estimate the terms involving $\mathcal{I}^4$, $\mathcal{I}^5$, $\mathcal{I}^6$ and $\mathcal{I}^7$ by following the same path as of remainder terms $R_{n}(9)$, $R_{n}(10)$, $R_{n}(11)$ and $R_{n}(12)$ in \cite[Equation 10]{kumar2021note} and their corresponding Lemmas 8, 6, 7 and 9, respectively.
\end{proof}
Let us recall $X^{\varrho_h}$ from Equation \eqref{eq:Xrho}.
The following lemma states that the scheme \eqref{eq:scheme} is consistent.
\begin{lemma}
	\label{lem:bound}
	Assume that Assumptions \ref{as:initial} to \ref{as:diffusion} hold. Then, 
	\begin{align*}
		\big\|\mathcal{R}[X^{\varrho_h}] \big\|_{G^h_{2,S}} \leq K h
	\end{align*}
\end{lemma}
\begin{proof}
	Using the definition of residuals from Equations \eqref{eq:incrment} and \eqref{eq:residual},  
	\begin{align}
		\mathcal{R}_0[X^{\varrho_h}]=  & X_0 - X_0^h \notag
		\\
		\mathcal{R}_j[X^{\varrho_h}] = & X_j- X_{j-1} - \Xi_j^h(X_{j-1} , u_j, r_{j-1}^u,  r_{j-1}, r_j) \notag
		\\
		=&\int_{t_{j-1}}^{t_j} \big[ b(s, X(s), r(s))-b( t_{j-1}^u, \Upsilon^h_j(X_{j-1}, u_j, r_{j-1}),  r_{j-1}^u) \big] ds \notag
		\\
		& +\sum_{\ell=1}^{\tilde{d}}\int_{t_{j-1}}^{t_j} \big[\sigma_{\ell}(s, X(s), r(s)) -  \sigma_{\ell}(t_{j-1}, X_{j-1}, r_{j-1})\big] dB_{\ell}(s) \notag
		\\
		&-  \sum_{\ell_{1},\ell_{2}=1}^{\tilde{d}}\int_{t_{j-1}}^{t_{j}} \int_{t_{j-1}}^{s}  \mathcal{D}_x \sigma_{\ell_1}(t_{j-1}, X_{j-1}, r_{j-1}) \sigma_{\ell_2}(t_{j-1}, X_{j-1}, r_{j-1}) dB_{\ell_{2}}(v)dB_{\ell_{1}}(s) \notag
		\\
		& -\sum_{\ell=1}^{\tilde{d}} \mathbbm{1}_{\big\{N_{t_{j-1}}^{t_j}=1\big\}} \big( \sigma_{\ell}(t_{j-1},X_{j-1}, r_j)- \sigma_{\ell}(t_{j-1},X_{j-1},  r_{j-1}) \big) \big( B_{\ell}(t_{j})-B_{\ell}(\uptau_{1}^{j-1})\big) \notag
	\end{align} 
	for $j\in \{1,\ldots, n_h\}$ which due to Equation \eqref{eq:norm:spiker} and Minkowski inequality gives, 
	\begin{align}
		\big\|\mathcal{R}[&X^{\varrho_h}] \big\|_{G^h_{2,S}} = \big\| \mathcal{R}_0[X^{\varrho_h}]\big\|_{L^2(\Omega)} + \Big\|  \max_{n\in \{1,\ldots, n_h \}}  \big| \sum_{j=1}^{n}\mathcal{R}_j  [X^{\varrho_h}] \big| \Big\|_{L^2(\Omega)} \notag 
		\\
		\leq & \big\| X_0 - X_0^h \big\|_{L^2(\Omega)}   \notag
		\\
		& + \Big\|  \max_{n\in \{1,\ldots, n_h \}} \big|  \sum_{j=1}^{n} \int_{t_{j-1}}^{t_j} \big[ b(s, X(s), r(s))-b( t_{j-1}^u, \Upsilon^h_j(X_{j-1}, u_j, r_{j-1}),  r_{j-1}^u) \big] ds\big|  \Big\|_{L^2(\Omega)} \notag  
		\\
		& +\Big\|  \max_{n\in \{1,\ldots, n_h \}} \big| \sum_{j=1}^{n} \sum_{\ell=1}^{\tilde{d}} \int_{t_{j-1}}^{t_j} \big[\sigma_{\ell}(s, X(s), r(s)) -  \sigma_{\ell}(t_{j-1}, X_{j-1}, r_{j-1})\big] dB_{\ell}(s)   \notag  
		\\
		& 
		- \sum_{j=1}^{n} \sum_{\ell_{1},\ell_{2}=1}^{\tilde{d}}\int_{t_{j-1}}^{t_{j}} \int_{t_{j-1}}^{s}  \mathcal{D}_x \sigma_{\ell_1}(t_{j-1}, X_{j-1}, r_{j-1}) \sigma_{\ell_2}(t_{j-1}, X_{j-1}, r_{j-1}) dB_{\ell_{2}}(v)dB_{\ell_{1}}(s) \notag
		\\
		&
		- \sum_{j=1}^{n} \sum_{\ell=1}^{\tilde{d}} \mathbbm{1}_{\big\{N_{t_{j-1}}^{t_j}=1\big\}} \big( \sigma_{\ell}(t_{j-1},X_{j-1}, r_j)- \sigma_{\ell}(t_{j-1},X_{j-1},  r_{j-1}) \big) \big( B_{\ell}(t_{j})-B_{\ell}(\uptau_{1}^{j-1})\big)   \big|  \Big\|_{L^2(\Omega)} \notag.
	\end{align}
	Further, using Lemma \ref{lem:J1} and Equation \eqref{eq:I1-7},
	\begin{align*}
		\big\|\mathcal{R}[X^{\varrho_h}]\big\|_{G^h_{2,S}} & \leq Kh + \sum_{i=1}^7 \Big\| \max_{n\in \{1,\ldots, n_h \}} \Big| \sum_{j=1}^{n} \mathcal{I}_j^{(i)}\Big| \Big\|_{L^2(\Omega)}
	\end{align*}
	Finally, the use of Lemma \ref{lem:I1} completes the proof.
\end{proof}
\begin{proof}[\textbf{Proof of Theorem \ref{thm:rate}}]
	The proof is a consequence of Theorem \ref{thm:bi} and Lemma \ref{lem:bound}.
	Using Bistability Theorem \ref{thm:bi} for $q =2$ with $Y^h = X^{\varrho_h} \in G_2^h$ and Lemma \ref{lem:bound} gives,
	\begin{align*}
		\big\| X^{\varrho_h}-X^h \big\|_{G_2^h}  &=  \Big\| \max_{n\in \{0,\ldots, n_h \}} |X_n - X_n^h|\Big\|_{L^2(\Omega)} \leq C_2  \big\|\mathcal{R}[X^{\varrho_h}] \big\|_{G^h_{2,S}} \leq Kh. \notag
	\end{align*}
	The proof of the main result is hence completed. 
\end{proof}
\begin{lemma}
	\label{lem:moment:m}
	Let Assumptions \ref{as:initial}, \ref{as:drift} and \ref{as:diffusion} hold. Then, there exists a constant $K > 0$, independent of $h$, such that 
	$$\mathbb{E}\big[\max_{n \in \{0,\ldots,n_h\}} | X_{n}^{h,M}|^2 | \mathcal{F}_T^r \big]\leq K. $$
\end{lemma}
\begin{proof}
	The proof of the following Lemma is similar to the proof of Lemma \ref{moment}.
\end{proof}
\begin{proof}[\textbf{Proof of Theorem \ref{thm:rate:wsi}}]
	Using Minkowski inequality and Theorem \ref{thm:rate},
	\begin{align*}
		\big\| \max_{n \in \{0, \ldots, n_h\}} |X_n - X_n^{h,M}| \big\|_{L^2(\Omega)} & \leq \big\| \max_{n \in \{0, \ldots, n_h\}} |X_n - X_n^{h}| \big\|_{L^2(\Omega)} + \big\| \max_{n \in \{0, \ldots, n_h\}} |X_n^h - X_n^{h,M}| \big\|_{L^2(\Omega)} 
		\\
		& \leq Kh + \big\| \max_{n \in \{0, \ldots, n_h\}} |X_n^h - X_n^{h,M}| \big\|_{L^2(\Omega)}.  
	\end{align*}
	Thus, it is sufficient to prove that,  
	$ \displaystyle \big\| \max_{n \in \{0, \ldots, n_h\}} |X_n^h - X_n^{h,M}| \big\|_{L^2(\Omega)} \leq Kh^{1/2}
	$.
	Due to  Equations~\eqref{eq:euler} and \eqref{eq:euler:wsi} along with Assumptions \ref{as:drift} and \ref{as:diffusion},  one can estimate
	\begin{align}
		\label{eq:e:rs}
		\mathbb{E} \big[|X_j^{h,u} - X_j^{h,u,M}|^2\big]  &\leq  K \mathbb{E} \big[\big| X_{j-1}^{h} - X_{j-1}^{h,M}\big|^2 \big] \notag
		\\
		& \quad + K  \mathbb{E} \big[\big(u_j h_j  \big|b(t_{j-1}, X_{j-1}^{h}, r_{j-1}) -  b(t_{j-1}, X_{j-1}^{h,M}, r_{j-1})\big|\big)^2 \big] \notag
		\\
		& \quad + K \mathbb{E} \Big[\big|  \sum_{\ell=1}^{\tilde{d}} \int_{t_{j-1}}^{ t_{j-1}^u} \big[\sigma_{\ell}(t_{j-1}, X_{j-1}^{h}, r_{j-1}) - \sigma_{\ell}(t_{j-1}, X_{j-1}^{h,M}, r_{j-1}) \big]  dB_{\ell} (s)\big|^2 \Big] \notag
		\\
		&\leq K \mathbb{E} \big[\big| X_{j-1}^{h} - X_{j-1}^{h,M}\big|^2 \big]
	\end{align}
	for any $j \in \{1,\ldots,n_h\}$.
	Next, the use of Equations \eqref{eq:scheme} and \eqref{eq:scheme:wsi} yields
	\begin{align*}
		& \mathbb{E}\big[\max_{n\in \{1,\ldots, k \}}\big|  X_{n}^{h} - X_{n}^{h,M} \big|^2 \big] \leq  K\mathbb{E}\Big[\max_{n\in \{1,\ldots, k \}}| \sum_{j=1}^n h_j \big[ b( t_{j-1}^u, X_{j}^{h,u}, r_{j-1}^u) - b( t_{j-1}^u, X_{j}^{h,u,M}, r_{j-1}^u)\big] \big|^2 \Big]  \notag
		\\
		& + K \mathbb{E}\Big[\max_{n\in \{1,\ldots, k \}}| \sum_{j=1}^n \sum_{\ell=1}^{\tilde{d}} \big[\sigma_{\ell}(t_{j-1}, X_{j-1}^{h}, r_{j-1}) -  \sigma_{\ell}(t_{j-1}, X_{j-1}^{h,M}, r_{j-1})\big] \big[B_{\ell}(t_{j})-B_{\ell}(t_{j-1})\big]|^2 \Big] \notag
		\\
		& 
		+ K \mathbb{E}\Big[\hspace{-.5mm} \max_{n\in \{1,\ldots, k \}}| \sum_{j=1}^n \sum_{\ell_{1},\ell_{2}=1}^{\tilde{d}}\int_{t_{j-1}}^{t_{j}} \int_{t_{j-1}}^{s} \big[ \mathcal{D}_x \sigma_{\ell_1}(t_{j-1}, X_{j-1}^{h}, r_{j-1}) \sigma_{\ell_2} (t_{j-1}, X_{j-1}^{h}, r_{j-1})  \notag
		\\
		& \quad - \mathcal{D}_x \sigma_{\ell_1}(t_{j-1}, X_{j-1}^{h,M}, r_{j-1}) \sigma_{\ell_2} (t_{j-1}, X_{j-1}^{h,M}, r_{j-1}) \big] dB_{\ell_{2}}(v)dB_{\ell_{1}}(s)|^2 \Big] \notag
		\\
		&
		+ K \mathbb{E}\Big[\max_{n\in \{1,\ldots, k \}} \Big| \sum_{j=1}^n \sum_{\ell=1}^{\tilde{d}} \mathbbm{1}_{\big\{N_{t_{j-1}}^{t_j}=1\big\}} \Big( \big\{ \sigma_{\ell}(t_{j-1},X_{j-1}^{h}, r_{j})- \sigma_{\ell}(t_{j-1},X_{j-1}^{h},  r_{j-1}) \big\}
		\\
		& \quad - \big\{ \sigma_{\ell}(t_{j-1},X_{j-1}^{h,M}, r_{j})- \sigma_{\ell}(t_{j-1},X_{j-1}^{h,M},  r_{j-1}) \big\}\Big)  \big\{ B_{\ell}(t_{j})-B_{\ell}(\uptau_{1}^{j-1})\big\} \Big|^2 \Big]
		\\
		& + K \mathbb{E}\Big[\max_{n\in \{1,\ldots, k \}} \Big| \sum_{j=1}^n \sum_{\ell=1}^{\tilde{d}} \mathbbm{1}_{\big\{N_{t_{j-1}}^{t_j}=1\big\}} \big\{ \sigma_{\ell}(t_{j-1},X_{j-1}^{h,M}, r_{j})- \sigma_{\ell}(t_{j-1},X_{j-1}^{h,M},  r_{j-1}) \big\} \big\{B_{\ell}(\uptau_{1}^{j-1}) -  B_{\ell}(t_{j-1})\big\} \Big|^2 \Big] 
	\end{align*} 
	for any $k \in \{1,\ldots,n_h\}$. 
	Observe that,  for the second last term, 
	\begin{align*}
		\Big\{\sum_{j=1}^n \sum_{\ell=1}^{\tilde{d}} \mathbbm{1}_{\big\{N_{t_{j-1}}^{t_j}=1\big\}} \big( \sigma_{\ell}(t_{j-1},X_{j-1}^{h}, r_{j}) &- \sigma_{\ell}(t_{j-1},X_{j-1}^{h},  r_{j-1})
		\\
		- \sigma_{\ell}(t_{j-1},X_{j-1}^{h,M}, r_{j})& + \sigma_{\ell}(t_{j-1},X_{j-1}^{h,M},  r_{j-1}) \big) \big\{ B_{\ell}(t_{j})-B_{\ell}(\uptau_{1}^{j-1})\big\} \Big\}_{n \in \{1, \ldots, n_h\}}
	\end{align*}
	is an $\{\tilde{\mathcal{F}}_{\uptau_1^{n}  \wedge t_{n+1}} \vee  \tilde{\mathcal{F}}_T^r\}_{n \in \{1, \ldots, n_h\}}$-martingale and for the last term,  
	\begin{align*}
		\Big\{\sum_{j=1}^n \sum_{\ell=1}^{\tilde{d}} \mathbbm{1}_{\big\{N_{t_{j-1}}^{t_j}=1\big\}} \big\{ 	\sigma_{\ell}(t_{j-1},X_{j-1}^{h,M}, r_{j})- \sigma_{\ell}(t_{j-1},X_{j-1}^{h,M},  r_{j-1}) \big\} \big\{B_{\ell}(\uptau_{1}^{j-1}) -  B_{\ell}(t_{j-1})\big\}\Big\}_{n \in \{1, \ldots, n_h\}}
	\end{align*} 
	is a martingale with respect to the filtration $\{\tilde{\mathcal{F}}_{t_{n}} \vee  \tilde{\mathcal{F}}_T^r\}_{n \in \{1, \ldots, n_h\}}$.
	Hence, the use of Assumption~\ref{as:drift} yields,
	\begin{align*} 
		& \mathbb{E}\big[\max_{n\in \{1,\ldots, k \}}\big|  X_{n}^{h} - X_{n}^{h,M} |^2 \big]  \leq  K k\sum_{j=1}^k (h_j)^2\mathbb{E}\big[|X_{j}^{h,u} - X_{j}^{h,u,M} |^2 \big] \notag
		\\
		& +   K \sum_{j=1}^k h_j  \mathbb{E}\big[ | \sigma_{\ell}(t_{j-1}, X_{j-1}^{h}, r_{j-1}) -  \sigma_{\ell}(t_{j-1}, X_{j-1}^{h,M}, r_{j-1}) |^2 \big] \notag
		\\
		& 
		+  K  \sum_{j=1}^k \sum_{\ell_{1},\ell_{2}=1}^{\tilde{d}} \int_{t_{j-1}}^{t_{j}}  \int_{t_{j-1}}^{s} \mathbb{E}\big[ |\mathcal{D}_x \sigma_{\ell_1}(t_{j-1}, X_{j-1}^{h}, r_{j-1}) \sigma_{\ell_2} (t_{j-1}, X_{j-1}^{h}, r_{j-1})  \notag
		\\
		& \quad - \mathcal{D}_x \sigma_{\ell_1}(t_{j-1}, X_{j-1}^{h,M}, r_{j-1}) \sigma_{\ell_2} (t_{j-1}, X_{j-1}^{h,M}, r_{j-1})|^2  \big]  dv ds  \notag
		\\
		& + K \sum_{j=1}^k \sum_{\ell=1}^{\tilde{d}} \mathbb{E}\Big[ \big| \big\{ \sigma_{\ell}(t_{j-1},X_{j-1}^{h}, r_{j})- \sigma_{\ell}(t_{j-1},X_{j-1}^{h},  r_{j-1}) \big\}
		\\
		& \quad - \big\{ \sigma_{\ell}(t_{j-1},X_{j-1}^{h,M}, r_{j})- \sigma_{\ell}(t_{j-1},X_{j-1}^{h,M},  r_{j-1}) \big\}\big|^2   \mathbb{E} \big[| B_{\ell}(t_{j})-B_{\ell}(\uptau_{1}^{j-1})|^2  | \tilde{\mathcal{F}}_{\uptau_1^{j-1} \wedge t_{j}} \vee \tilde{\mathcal{F}}_T^r   \big] \Big]
		\\
		& + K  \sum_{j=1}^k \sum_{\ell=1}^{\tilde{d}} \mathbb{E}\Big[  \mathbbm{1}_{\big\{N_{t_{j-1}}^{t_j}=1\big\}} \big| \sigma_{\ell}(t_{j-1},X_{j-1}^{h,M}, r_{j})- \sigma_{\ell}(t_{j-1},X_{j-1}^{h,M},  r_{j-1}) \big|^2 
		\\
		&
		\quad \times \mathbb{E} \big[| B_{\ell}(\uptau_{1}^{j-1}) -  B_{\ell}(t_{j-1})  |^2 | \tilde{\mathcal{F}}_{t_{j-1}} \vee \tilde{\mathcal{F}}_T^r   \big]  \Big] 
	\end{align*}
	for any $k \in \{1,\ldots,n_h\}$. Due to Equation \eqref{eq:e:rs}, Remarks \ref{rem:D}, \ref{rem:growth} and Assumption \ref{as:diffusion}
	\begin{align*}
		\mathbb{E}\big[\max_{n\in \{1,\ldots, k \}}\big|  X_{n}^{h}  - X_{n}^{h,M} |^2 \big]  \leq  K h \sum_{j=1}^k \mathbb{E}\big[ | X_{n}^{h} - X_{n}^{h,M}|^2 \big]+ K  h \sum_{j=1}^k  \mathbb{E}\Big[  \mathbbm{1}_{\big\{N_{t_{j-1}}^{t_j}=1\big\}} \Big( 1 + \mathbb{E} \big[ | X_{n}^{h,M}|^2| \mathcal{F}_T^r \big] \Big) \Big] 
	\end{align*}
	which on using  Lemma \ref{lem:jump} and \ref{lem:moment:m} gives
	\begin{align*} 
		\mathbb{E}\big[ \max_{n\in \{1,\ldots, k \}}\big|  X_{n}^{h} - X_{n}^{h,M} |^2 \big]  \leq Kh + K h \sum_{j=1}^k \mathbb{E}\big[ | X_{n}^{h} - X_{n}^{h,M}|^2 \big]  
	\end{align*}
	for any $k \in \{1,\ldots,n_h\}$. 
	Finally,   discrete Gr\"onwall's inquality leads to, 
	\begin{align*}
		\mathbb{E}\big[ \max_{n\in \{1,\ldots,  k \}}\big|  X_{n}^{h} - X_{n}^{h,M} |^2 \big]  \leq Kh e^{h k} \leq Kh 
	\end{align*}
	and thus the proof is completed by taking square-root on both sides. 
\end{proof}

\begin{proof}[\textbf{Proof of Theorem \ref{thm:rate:ri}}]
	By the application of  Minkowski inequality and Theorem \ref{thm:rate},
	\begin{align*}
		\big\| \max_{n \in \{0, \ldots, n_h\}} |X_n - X_n^{h,R}| \big\|_{L^2(\Omega)} & \leq \big\| \max_{n \in \{0, \ldots, n_h\}} |X_n - X_n^{h}| \big\|_{L^2(\Omega)} + \big\| \max_{n \in \{0, \ldots, n_h\}} |X_n^h - X_n^{h,R}| \big\|_{L^2(\Omega)} 
		\\
		& \leq Kh + \big\| \max_{n \in \{0, \ldots, n_h\}} |X_n^h - X_n^{h,R}| \big\|_{L^2(\Omega)}. 
	\end{align*}
	From  Equations \eqref{eq:euler} and \eqref{eq:euler:ri},  one uses  Assumptions \ref{as:drift} and \ref{as:diffusion} to obtain, 
	\begin{align}
		\label{eq:e:ri}
		& \mathbb{E} \big[|X_j^{h,u} - X_j^{h,u,R}|^2\big]  \leq  K \mathbb{E} \big[\big| X_{j-1}^{h} - X_{j-1}^{h,R}\big|^2 \big]  + K \mathbb{E} \big[(u_j  h_j \big|b(t_{j-1}, X_{j-1}^{h}, r_{j-1}) -  b(t_{j-1}, X_{j-1}^{h,R}, r_{j-1})\big|\big)^2 \big] \notag
		\\
		& + K \mathbb{E} \hspace{-0.5mm} \Big[\big| \hspace{-0.5mm}  \sum_{\ell=1}^{\tilde{d}} \int_{t_{j-1}}^{ t_{j-1}^u}\hspace{-0.5mm}  \big[\sigma_{\ell}(t_{j-1}, X_{j-1}^{h}, r_{j-1}) \hspace{-0.5mm} - \hspace{-0.5mm}  \sigma_{\ell}(t_{j-1}, X_{j-1}^{h,R}, r_{j-1}) \big]  dB_{\ell} (s)\big|^2 \Big] \hspace{-0.5mm}  \leq \hspace{-0.5mm} K \mathbb{E} \big[\big| X_{j-1}^{h} - X_{j-1}^{h,R}\big|^2 \big]
	\end{align}
	for any $j \in \{1,\ldots,n_h\}$.
	Moreover,    from  Equations \eqref{eq:scheme} and \eqref{eq:scheme:ri}, 
	\begin{align*}
		& \mathbb{E}\big[\max_{n\in \{1,\ldots, k \}}\big|  X_{n}^{h} - X_{n}^{h,R} \big|^2 \big] \leq  K\mathbb{E}\Big[\max_{n\in \{1,\ldots, k \}}| \sum_{j=1}^n h_j \big[ b( t_{j-1}^u, X_{j}^{h,u}, r_{j-1}^u) - b( t_{j-1}^u, X_{j}^{h,u,R}, r_{j-1}^u)\big] \big|^2 \Big]  \notag
		\\
		& + K \mathbb{E}\Big[\max_{n\in \{1,\ldots, k \}}| \sum_{j=1}^n \sum_{\ell=1}^{\tilde{d}} \big[\sigma_{\ell}(t_{j-1}, X_{j-1}^{h}, r_{j-1}) -  \sigma_{\ell}(t_{j-1}, X_{j-1}^{h,R}, r_{j-1})\big] \big[B_{\ell}(t_{j})-B_{\ell}(t_{j-1})\big]|^2 \Big] \notag
		\\
		& 
		+ K \mathbb{E}\Big[\hspace{-.5mm} \max_{n\in \{1,\ldots, k \}}| \sum_{j=1}^n \sum_{\ell_{1},\ell_{2}=1}^{\tilde{d}}\int_{t_{j-1}}^{t_{j}} \int_{t_{j-1}}^{s} \big[ \mathcal{D}_x \sigma_{\ell_1}(t_{j-1}, X_{j-1}^{h}, r_{j-1}) \sigma_{\ell_2} (t_{j-1}, X_{j-1}^{h}, r_{j-1})  \notag
		\\
		& \quad - \mathcal{D}_x \sigma_{\ell_1}(t_{j-1}, X_{j-1}^{h,R}, r_{j-1}) \sigma_{\ell_2} (t_{j-1}, X_{j-1}^{h,R}, r_{j-1}) \big] dB_{\ell_{2}}(v)dB_{\ell_{1}}(s)|^2 \Big] \notag
		\\
		&
		+ K \mathbb{E}\Big[\max_{n\in \{1,\ldots, k \}} \Big| \sum_{j=1}^n \sum_{\ell=1}^{\tilde{d}} \mathbbm{1}_{\big\{N_{t_{j-1}}^{t_j}=1\big\}}  \big\{ \sigma_{\ell}(t_{j-1},X_{j-1}^{h}, r_{j})- \sigma_{\ell}(t_{j-1},X_{j-1}^{h},  r_{j-1}) \big\}\big\{ B_{\ell}(t_{j})-B_{\ell}(\uptau_{1}^{j-1})\big\} \Big|^2 \Big]
	\end{align*} 
	and notice that 
	$$
	\Big\{\sum_{j=1}^n \sum_{\ell=1}^{\tilde{d}} \mathbbm{1}_{\big\{N_{t_{j-1}}^{t_j}=1\big\}} \big\{ \sigma_{\ell}(t_{j-1},X_{j-1}^{h}, r_{j}) - \sigma_{\ell}(t_{j-1},X_{j-1}^{h},  r_{j-1})\big\}\big\{ B_{\ell}(t_{j})-B_{\ell}(\uptau_{1}^{j-1})\big\} \Big\}_{n \in \{1, \ldots, n_h\}}
	$$ is an $\{\tilde{\mathcal{F}}_{\uptau_1^{n}  \wedge t_{n+1}} \vee  \tilde{\mathcal{F}}_T^r\}_{n \in \{1, \ldots, n_h\}}
	$-martingale.
	Hence,  Assumption \ref{as:drift} yields,
	\begin{align*} 
		\mathbb{E}\big[\max_{n\in \{1,\ldots, k \}}\big|  X_{n}^{h}& - X_{n}^{h,R} |^2 \big]  \leq  K k\sum_{j=1}^k (h_j)^2\mathbb{E}\big[|X_{j}^{h,u} - X_{j}^{h,u,R} |^2 \big] \notag
		\\
		& +   K \sum_{j=1}^k h_j  \mathbb{E}\big[ | \sigma_{\ell}(t_{j-1}, X_{j-1}^{h}, r_{j-1}) -  \sigma_{\ell}(t_{j-1}, X_{j-1}^{h,R}, r_{j-1}) |^2 \big] \notag
		\\
		& 
		+  K  \sum_{j=1}^k \sum_{\ell_{1},\ell_{2}=1}^{\tilde{d}} \int_{t_{j-1}}^{t_{j}}  \int_{t_{j-1}}^{s} \mathbb{E}\big[ |\mathcal{D}_x \sigma_{\ell_1}(t_{j-1}, X_{j-1}^{h}, r_{j-1}) \sigma_{\ell_2} (t_{j-1}, X_{j-1}^{h}, r_{j-1})  \notag
		\\
		& \qquad \qquad- \mathcal{D}_x \sigma_{\ell_1}(t_{j-1}, X_{j-1}^{h,R}, r_{j-1}) \sigma_{\ell_2} (t_{j-1}, X_{j-1}^{h,R}, r_{j-1})|^2  \big]  dv ds  \notag
		\\
		& + K \sum_{j=1}^k \sum_{\ell=1}^{\tilde{d}} \mathbb{E}\Big[\mathbbm{1}_{\big\{N_{t_{j-1}}^{t_j}=1\big\}} \big| \big\{ \sigma_{\ell}(t_{j-1},X_{j-1}^{h}, r_{j})- \sigma_{\ell}(t_{j-1},X_{j-1}^{h},  r_{j-1}) \big\} \big|^2
		\\
		& \qquad \qquad \times \big[| B_{\ell}(t_{j})-B_{\ell}(\uptau_{1}^{j-1})|^2  | \tilde{\mathcal{F}}_{\uptau_1^{j-1} \wedge t_{j}} \vee \tilde{\mathcal{F}}_T^r   \big] \Big]
	\end{align*}
	for any $k \in \{1,\ldots,n_h\}$. 
	Further,   on using  Equation \eqref{eq:e:ri}, Remarks \ref{rem:D}, \ref{rem:growth} and Assumption \ref{as:diffusion}, one obtains, 
	\begin{align*}
		\mathbb{E}\big[\max_{n\in \{1,\ldots, k \}}\big|  X_{n}^{h}  - X_{n}^{h,R} |^2 \big]  & \leq  K h \sum_{j=1}^k \mathbb{E}\big[ | X_{n}^{h} - X_{n}^{h,R}|^2 \big]
		\ + K  h \sum_{j=1}^k  \mathbb{E}\Big[  \mathbbm{1}_{\big\{N_{t_{j-1}}^{t_j}=1\big\}} \Big( 1 + \mathbb{E} \big[ | X_{n}^{h}|^2 | \mathcal{F}_T^r \big] \Big) \Big] 
	\end{align*}
	and thus Lemmas  \ref{moment} and  \ref{lem:jump} gives, 
	\begin{align*} 
		\mathbb{E}\big[ \max_{n\in \{1,\ldots, k \}}\big|  X_{n}^{h} - X_{n}^{h,R} |^2 \big] & \leq Kh + K h \sum_{j=1}^k \mathbb{E}\big[ | X_{n}^{h} - X_{n}^{h,R}|^2 \big]  \label{eq:E:RI}
	\end{align*}
	for any $k \in \{1,\ldots,n_h\}$.
	Finally, applying discrete Gr\"onwall's inequality  and taking square root on both sides completes the proof.
\end{proof}

\section{Numerical Implementation} 	
\label{sec:numerical}
The randomized Milstein scheme \eqref{eq:scheme} involves the simulation of L\'evy area given by 
$
\displaystyle \int_{t_{j-1}}^{t_{j}} \int_{t_{j-1}}^{s}dB_{\ell_{2}}(v)dB_{\ell_{1}}(s)
$
for $\ell_1\neq \ell_2 \in \{1, \ldots, \tilde{d}\}$ which is difficult to simulate. 
However,  under the following  diffusion commutative conditions 
\begin{equation}
	\mathcal{D}_x \sigma_{\ell_1}(t,x,i_0) \  \sigma_{\ell_2} (t,x,i_0) = \mathcal{D}_x \sigma_{\ell_2}(t,x,i_0) \ \sigma_{\ell_1} (t,x,i_0) \label{eq:commutative}
\end{equation}
for all  $x \in \mathbb{R}^d$, $t \in [0,T]$, $i_0 \in S$  and  $l_1,l_2 \in \{1,\ldots,\tilde{d}\}$,  the term involving iterated Brownian integrals can be simplified \citep[see][Chapter 10, Equation (3.6)]{kloeden1992}, which leads to the following easily implementable Milstein scheme, 
\begin{align}
	X_{j}^{h, u} & = X_{j-1}^{h} + b(t_{j-1}, X_{j-1}^h, r_{j-1}) h_j u_j + \sum_{\ell=1}^{\tilde{d}} \sigma_{\ell}(t_{j-1}, X_{j-1}^h, r_{j-1}) \{B_{l}(t_{j-1}^u)-B_l(t_{j-1})\}\label{eq:euler:com}
	\\
	X_{j}^{h} & = X_{j-1}^{h}+ b( t_{j-1}^u, X_{j}^{h,u}, r_{j-1}^u) h_j  +  \sum_{\ell=1}^{\tilde{d}} \sigma_{\ell}(t_{j-1}, X_{j-1}^h, r_{j-1})\{B_{\ell}(t_{j})-B_{\ell}(t_{j-1})\} \notag
	\\
	& 
	\quad + \frac{1}{2} \sum_{\ell_{1},\ell_{2}=1}^{\tilde{d}} \mathcal{D}_x \sigma_{\ell_1}(t_{j-1}, X_{j-1}^h, r_{j-1}) \sigma_{\ell_2} (t_{j-1}, X_{j-1}^h, r_{j-1})   \notag
	\\
	& \qquad \qquad \times \Big\{ ( B_{\ell_2}(t_{j})-B_{\ell_2}(t_{j-1})) (B_{\ell_1}(t_{j})-B_{\ell_1}(t_{j-1})) -  \mathbbm{1}_{\{\ell_1 = \ell_2\}}(t_{j}-t_{j-1})\Big\} \notag
	\\
	&
	\quad + \sum_{\ell=1}^{\tilde{d}} \mathbbm{1}_{\big\{N_{t_{j-1}}^{t_j}=1\big\}} \big\{ \sigma_{\ell}(t_{j-1},X_{j-1}^h, r_{j})- \sigma_{\ell}(t_{j-1},X_{j-1}^h,  r_{j-1}) \big\} \big\{ B_{\ell}(t_{j})-B_{\ell}(\tau_{1}^{j-1})\big\}  \label{eq:scheme:com}
\end{align}
almost surely for any $j \in \{1, \ldots, n_h\}$.  
For  one-dimensional  Brownian motion, the commutative condition is automatically satisfied and the L\'evy area issue does not arise. 
We also refer the readers to \cite{davie2015} and the references therein for the simulation of L\'evy area arising in the case of standard SDEs. 

For the simulation of the Markov chain  $r$,   we refer the readers to \cite[Section 2.4]{Yin2012}. 
Indeed,  if $\imath$ is the current state of the chain $r$,  then the switching  time to the next state follows exponential distribution with parameter $-\mathfrak{q}_{\imath \imath}$. 
Further,  the new state is determined by the random variable $Z_{\imath}$ given~by, 
\begin{align*}
	Z_{\imath} = 
	\begin{cases}
		1, &\text{ if } U \leq \mathfrak{q}_{\imath1}/(-\mathfrak{q}_{\imath \imath}),
		\\
		2, &\text{ if } \mathfrak{q}_{\imath1}/(-\mathfrak{q}_{\imath \imath}) < U \leq (\mathfrak{q}_{\imath 1} + \mathfrak{q}_{\imath 2})/(-\mathfrak{q}_{\imath\imath}),\\
		\vdots & \vdots\\
		m', &\text{ if } \displaystyle{\sum_{\CMjmath \neq \imath,  \CMjmath < m'}} \mathfrak{q}_{\imath \CMjmath }/(-\mathfrak{q}_{\imath \imath}) \leq U. 
	\end{cases}
\end{align*}
where $U$ is a standard uniform random variable.

Moreover, for the simulation of randomization component from Equation \eqref{eq:euler},  we generate independent standard uniform random variables $\{u_j\}_{j \in \mathbb{N}}$ to be used on the finest discretization grid.  
On subsequent coarser grids, we recursively select the uniform random variables based on Bernoulli ($0.5$) distribution.

Further, we demonstrate numerically the theoretical findings of Theorems \ref{thm:rate}, \ref{thm:rate:wsi} and \ref{thm:rate:ri}.
We also mention that these results hold even if the randomization in the drift coefficient is not carried out, \textit{i.e.}, when $u_j\equiv 0$ for all $j\in \mathbb{N}$ in Equation \eqref{eq:euler}.
In such a case, the randomized Milstein scheme $X^{h}$ in Equation \eqref{eq:scheme} reduces to the Milstein scheme $X^{h,C}$ studied in \cite{nguyen2017milstein}, see also \cite{kumar2020explicit,kumar2021note, kumar2022}.  
Unlike the Milstein scheme, the randomized Milstein scheme does not require the drift to be differentiable. 
Indeed, it is only assumed to be Lipschitz continuous in the state variable. 
Also, the $L^2$-error of the randomized Milstein scheme is numerically shown to be less than that of the Milstein scheme of \cite{nguyen2017milstein} but at a higher computational cost, see Figures \ref{fig:com:error},\ref{fig:com:cpu} and Table \ref{tab:com}. Thus, the randomized Milstein scheme is recommended when the drift coefficient is non-differentiable and Lipschitz continuous. However, for more regular drift coefficients, one should use the Milstein scheme of \cite{nguyen2017milstein,kumar2021note}.
Moreover, when the state space of the Markov chain $r$ is a singleton set, the above observations coincide with that of  \cite{Kruse2019}.  
It is also illustrated that both the modified randomized scheme given in Equations \eqref{eq:euler:wsi} and  \eqref{eq:scheme:wsi} and the reduced randomized scheme from Equations \eqref{eq:euler:ri} and \eqref{eq:scheme:ri} performs better than the classical Euler scheme in terms of their errors. Indeed, this is exhibited with non-randomization (when $u_j\equiv 0$ for all $j \in \mathbb{N}$) as well as with randomization (when $u_j \sim U(0,1)$ for all $j \in \mathbb{N}$), see Figures \ref{fig:inc:swi:error},\ref{fig:inc:swi:cpu} and Table \ref{tab:comphalf}. Thus, the Euler scheme can be used for Lipschitz continuous drift and diffusion coefficients. However, the non-randomized modified and reduced schemes can be preferred over their randomized versions as well as the classical Euler schemes when diffusion coefficient is aditionally assumed to be differentiable.

\begin{example}
	\label{ex:1} 
Consider the following one-dimensional SDEwMS,  
\begin{align*}
	dX(t) = b (t, X(t),r(t)) dt + \sigma(t,X(t),r(t))dB(t)
\end{align*}
almost surely for any $t \in [0,1]$ with initial values $X(0)=1$ and $ r(0)=1$ where  
\begin{align*}
	b(t, x, \imath) =  
	\begin{cases}
		|x|,  & \mbox{ if } \imath=0, \\
		\sin |x|, & \mbox{ if }  \imath=1,
	\end{cases} \quad \mbox{ and } \quad \sigma(t, x, \imath) =  
	\begin{cases}
		x,   & \mbox{ if } \imath=0,  \\
		\sin x, & \mbox{ if }  \imath=1, 
	\end{cases} 
\end{align*}
for any $x \in \mathbb{R}$, $t \in [0,1]$ and $\imath \in S$.
Also,  we take state space $S=\{0,1\}$ and generator 
$
\mathcal{Q} = \begin{pmatrix}
	-0.5 & 0.5 \\
	0.5 & -0.5
\end{pmatrix}
$. 
Clearly,  $b$ is a non-differentiable function and satisfies Assumption \ref{as:drift}, and $\sigma$ satisfies Assumption \ref{as:diffusion} and the commutative condition \eqref{eq:commutative}. 
A comparison of $L^2$-error and CPU time of the randomized Milstein scheme $X^{h}$  from  Equations \eqref{eq:euler:com} and \eqref{eq:scheme:com}  with the Milstein scheme  $X^{h, C}$ from \cite[Equation (3.3)]{nguyen2017milstein} is given in Figures \ref{fig:com:error},\ref{fig:com:cpu} and Table~\ref{tab:com}. 
In both cases, the scheme with step-size $h_j=2^{-15}$ is taken to be the true solution of the above SDEwMS and the simulation is performed with $160000$ number of paths. 
Clearly,  the scheme  $X^{h}$ exhibits the rate of convergence as $1.0$ and hence,  the result of Theorem \ref{thm:rate} is justified.

\begin{figure}[h]
	\centering
	\begin{subfigure}[]{\textwidth}  
		\centering
		\includegraphics[width=0.85\textwidth]{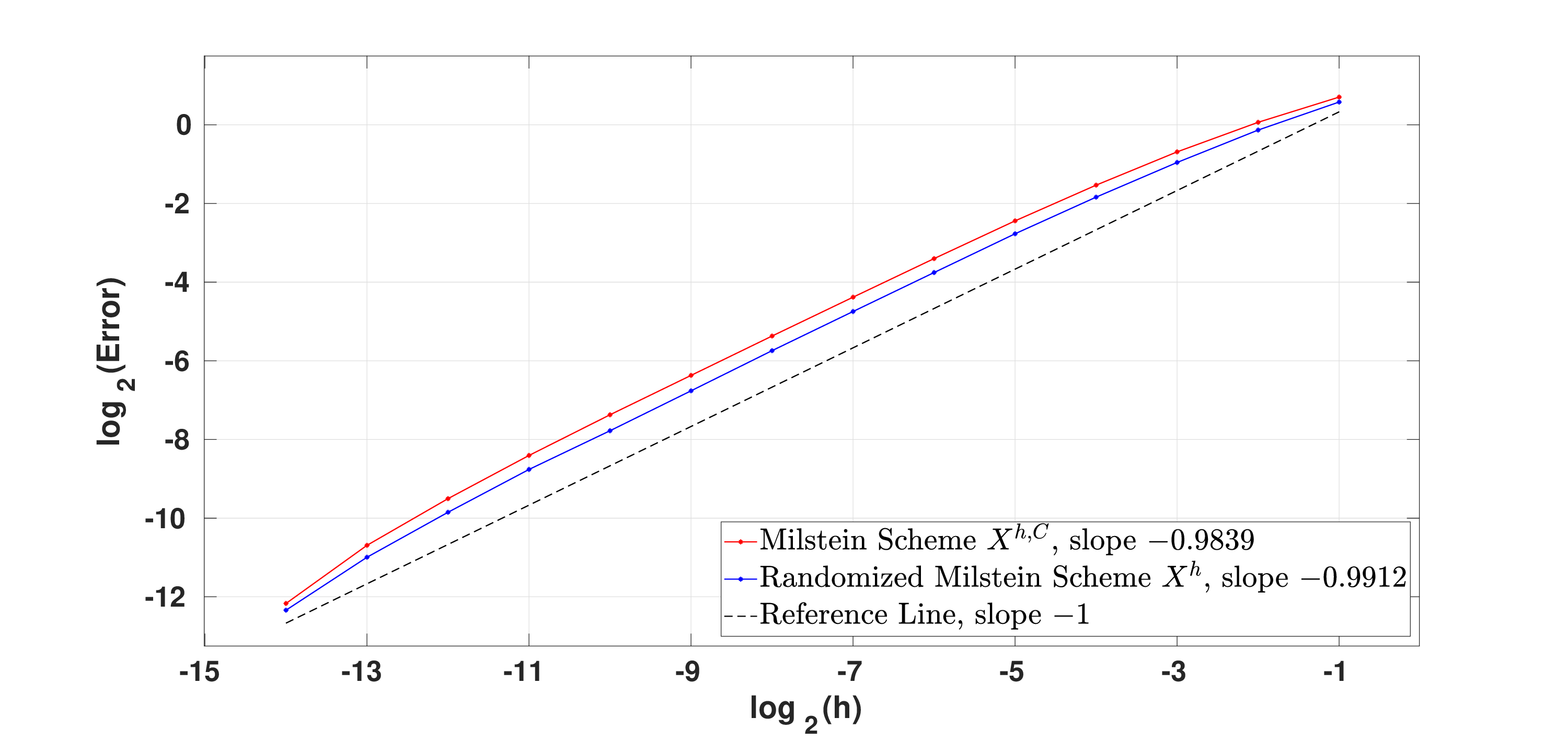}
		\caption{Error comparison}
		\label{fig:com:error}
	\end{subfigure}
	\begin{subfigure}[b]{\textwidth}  
		\centering
		\includegraphics[width=0.85\textwidth]{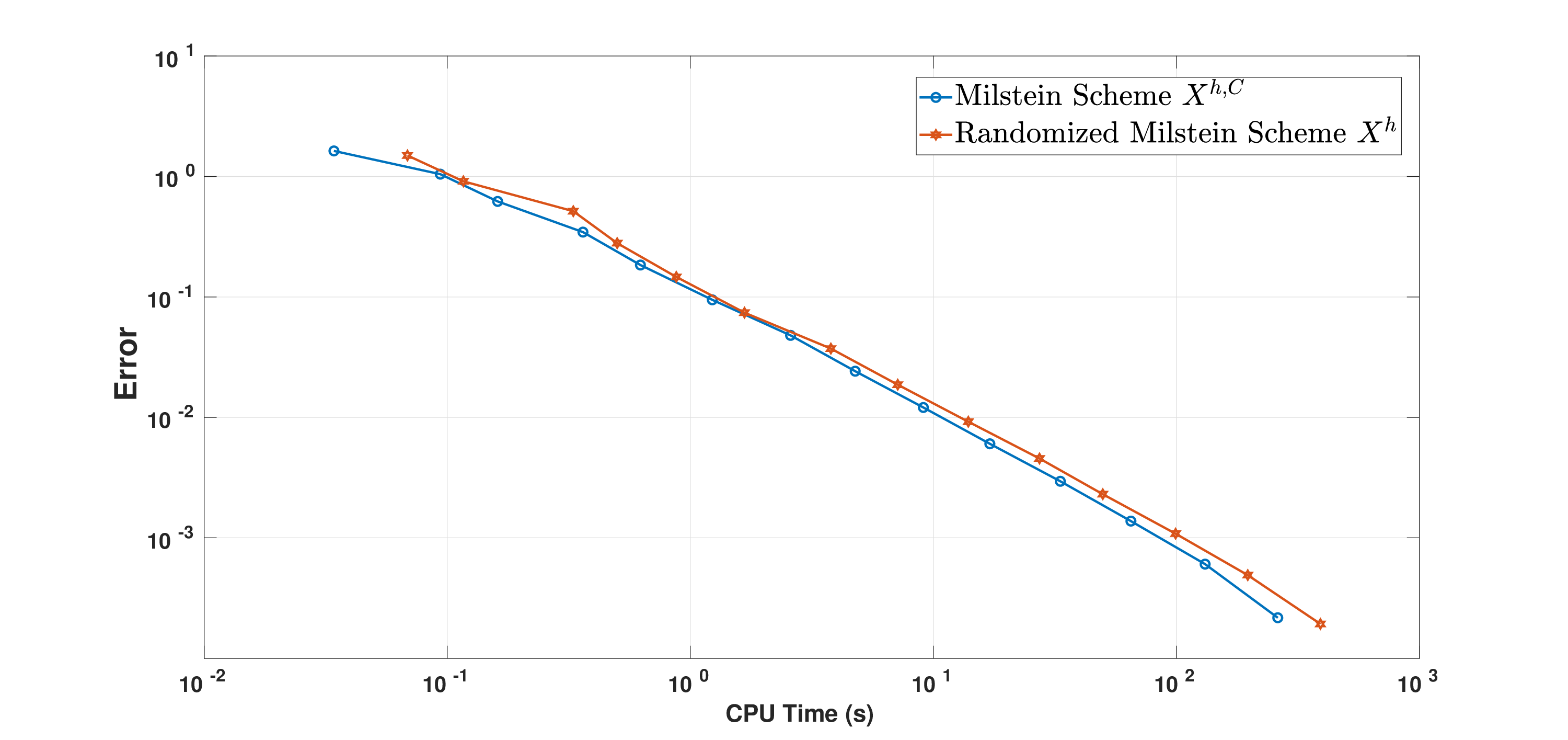}
		\caption{CPU Time comparison}
		\label{fig:com:cpu}
	\end{subfigure}
	
	\caption{Randomized Milstein Scheme vs. Milstein Scheme}
	\label{fig:comparison}
\end{figure}
\FloatBarrier
\begin{table}[h]
	\centering
	\begingroup
	\footnotesize
	\caption{Errors and CPU Time of the Randomized Milstein Scheme and the Milstein Scheme}
	\begin{tabular}{|c|ccccccc|}
		\hline
		\noalign{\vspace{1.5pt}} 
		$h$         & $2^{-14}$ & $2^{-13}$ & $2^{-12}$ & $2^{-11}$ & $2^{-10}$ & $2^{-9}$ & $2^{-8}$
		\\
		\hline 
		$\|X_{n_h} - X_{n_h}^{h,C} \|_{L^2(\Omega)}$ & 
		$0.00021708$ & $0.00060482$ & $0.00137690$ & $0.00294788$ & $0.00603756$ & $0.01208470$ & $0.02417852$  
		\\
		CPU Time & $261.470582$ & $131.341813$ & $65.035875$ & $33.325486$ & $17.089388$ & $9.102217$ & $4.772586$
		\\
		\hline
		$\|X_{n_h} - X_{n_h}^h \|_{L^2(\Omega)}$ & 
		$0.00019280$ & $0.00049020$ & $0.00108269$ & $0.00230352$ & $0.00455661$ & $0.00920727$ & $0.01868714$ 
		\\
		CPU Time & $391.632137$ & $196.778117$ & $99.305244$ & $49.821193$ & $27.357823$ & $13.934358$ & $7.140480$ 
		\\
		\hline
		\noalign{\vspace{2pt}}
		\hline
		\noalign{\vspace{1.5pt}} 
		$h$         & $2^{-7}$ & $2^{-6}$ & $2^{-5}$ & $2^{-4}$ & $2^{-3}$ & $2^{-2}$ & $2^{-1}$ 
		\\
		\hline
		$\|X_{n_h} - X_{n_h}^{h,C} \|_{L^2(\Omega)}$ & $0.04796301$ & $0.09464937$ & $0.18408339$ & $0.34538574$ & $0.62091218$ & $1.04555606$ & $1.63042997$ 
		\\
		CPU Time & $2.587062$ & $1.232459$ & $0.623805$ & $0.361891$ & $0.161321$ & $0.093668$ & $0.034194$ 
		\\
		\hline
		$\|X_{n_h} - X_{n_h}^h \|_{L^2(\Omega)}$ & $0.03726680$ & $0.07402992$ & $0.14663204$ & $0.27959858$ & $0.51513882$ & $0.91186019$ & $1.49448935$
		\\
		CPU Time & $3.790610$ & $1.671852$ & $0.876138$ & $0.500450$ & $0.330213$ & $0.116581$ & $0.068625$ 
		\\
		\hline
	\end{tabular}
	\label{tab:com}
	\endgroup
\end{table} 
\FloatBarrier
\begin{remark}
	The Milstein scheme need the drift coefficients to be differentiable to achieve rate of convergence $1.0$. It is observed in Figure \ref{fig:com:error} that the classical Milstein scheme shows similar convergence results as the randomized Milstein scheme at a less computational cost, while having non-differentiable drift coefficient.   
	 We are unable to find an example of SDEwMS with non-differentiable drift coefficient for which one could observe the difference of the experimental orders of convergence. However, there is no result in the literature which guarantees the convergence rate $1.0$ of the Milstein scheme for SDEwMS with non-differentiable drift coefficients.
\end{remark}

We now proceed for the illustration of results obtained in Theorems \ref{thm:rate:wsi} and \ref{thm:rate:ri}. 
Indeed, the rate of convergence of the modified randomized scheme $X^{h,M}$ from Equations \eqref{eq:euler:wsi} and \eqref{eq:scheme:wsi} and that of the reduced randomized scheme $X^{h,R}$ from Equations \eqref{eq:euler:ri} and \eqref{eq:scheme:ri} are shown to be $1/2$. These two schemes are considered to demonstrate the importance of  the last term on the right side of Equation~\eqref{eq:scheme}. 
Further, it is also natural to compare the schemes $X^{h,M}$ and $X^{h,R}$ with the classical Euler scheme $X^{h,E}$ and with the non-randomized modified scheme $X^{h,MC}$ and the non-randomized reduced scheme $X^{h,RC}$ which are defined through $X^{h,M}$ and $X^{h,R}$, respectively by taking $u_j\equiv 0$ for all $j\in\mathbb{N}$. 
The scheme with a step-size $h_j=2^{-18}$ is considered the true solution for all the above schemes and the simulation is carried out with $80000$ number of paths.
Figure \ref{fig:inc:swi:error} and Table \ref{tab:comphalf} illustrate that the rate of convergence of the schemes $X^{h,E}$,  $X^{h,M}$,  $X^{h,R}$,   $X^{h,MC}$ and  $X^{h,RC}$ is  $1/2$.
At last, Figure \ref{fig:inc:swi:cpu} and Table \ref{tab:comphalf}, presents a comparison of computational efficiency of all the above half-order schemes. 
Thus, non-randomized modified and reduced schemes perform better than Euler scheme and their randomized versions.

\begin{figure}[h]
	\centering
	\begin{subfigure}[b]{\textwidth}  
	\centering
	\includegraphics[width=0.85\textwidth]{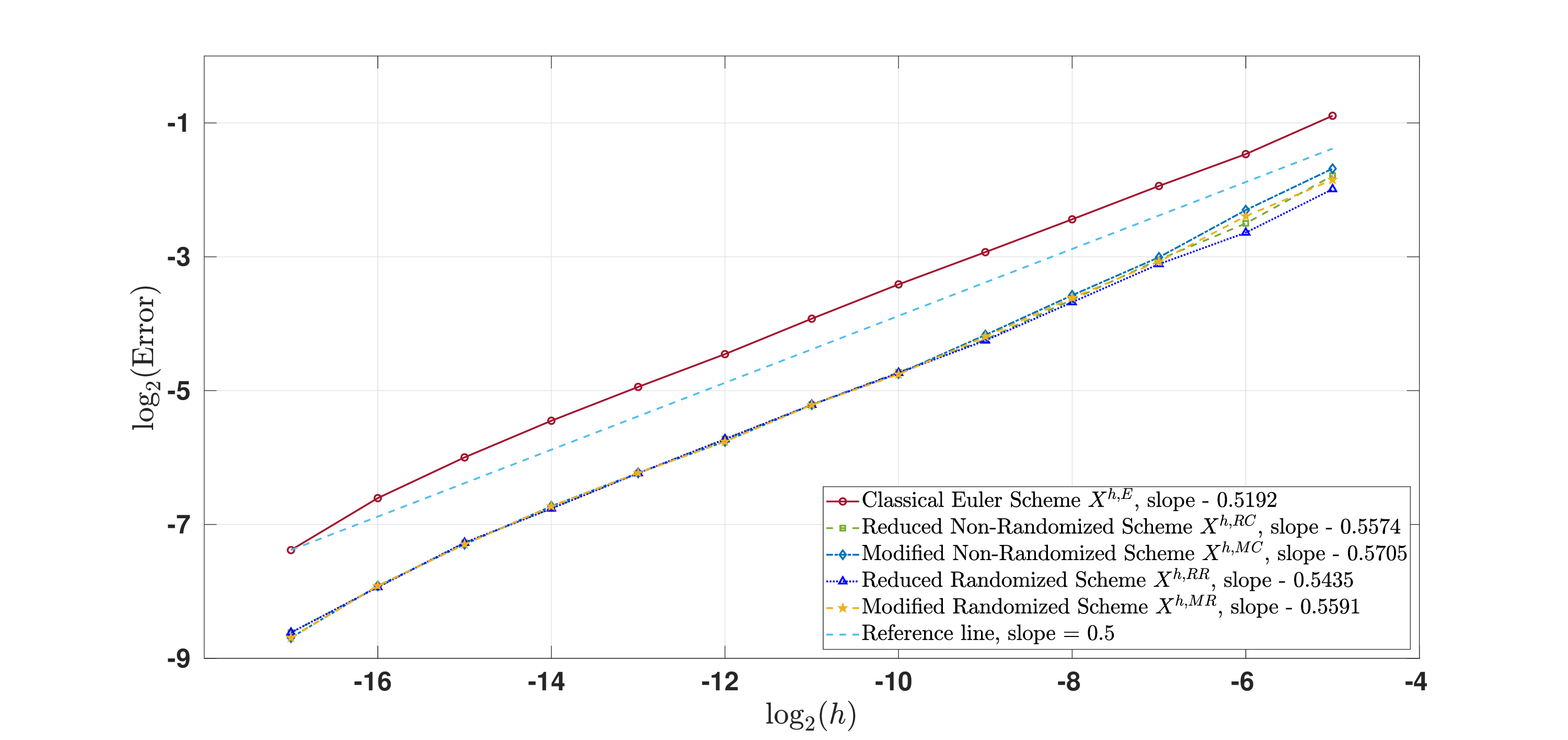}
	\caption{Error comparison} 
	\label{fig:inc:swi:error}
	\end{subfigure}
\end{figure}
\begin{figure}[h]
	\ContinuedFloat
	\centering
	\begin{subfigure}[]{\textwidth}  
	\centering	
	\includegraphics[width=0.85\textwidth]{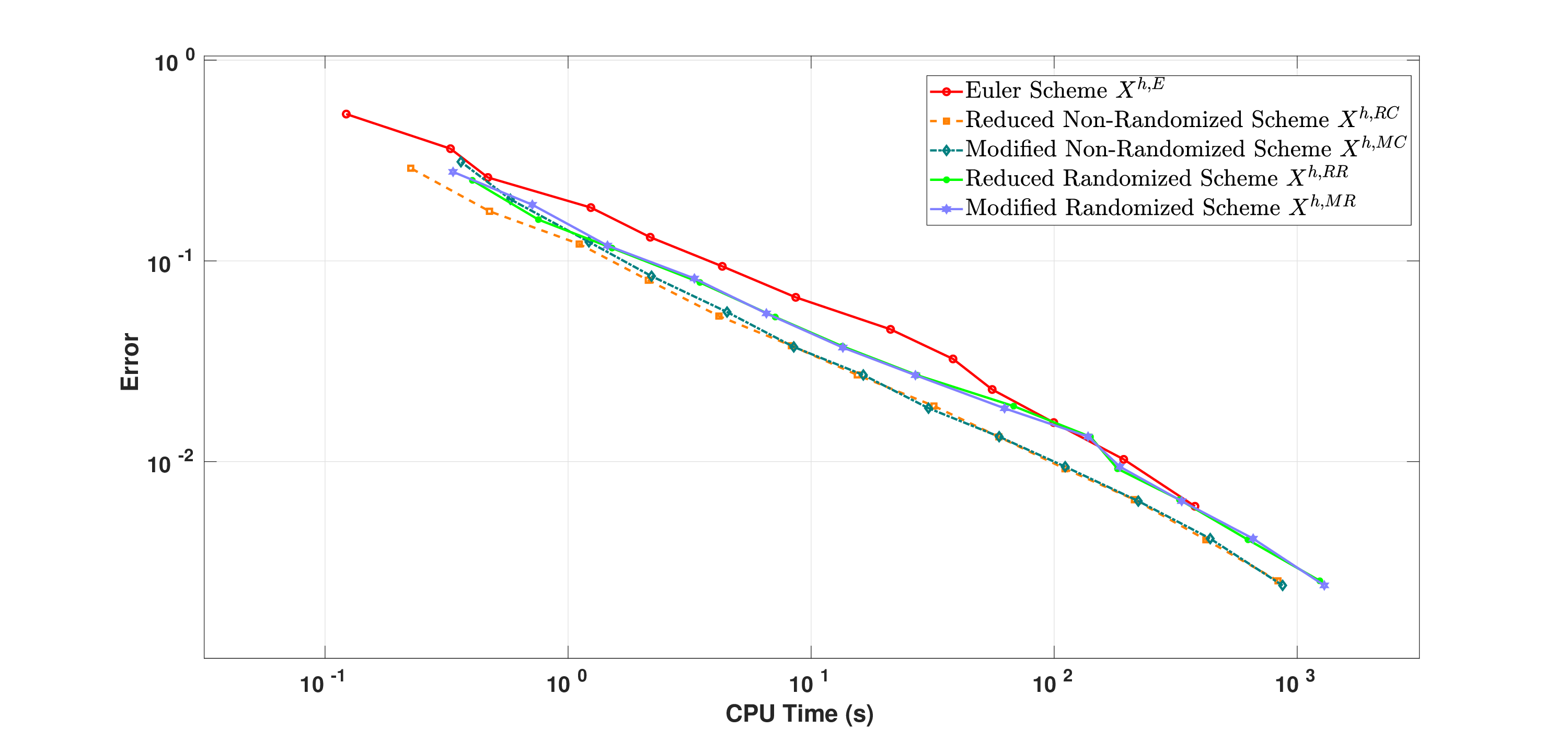} 
	\caption{CPU Time comparison}
	\label{fig:inc:swi:cpu}
	\end{subfigure}
\caption{Randomized and Non-randomized Modified and Reduced Schemes vs. Classical Euler Schemes.}
\label{fig:halfcom}
\end{figure}  
\FloatBarrier
\begin{table}[H]
	\caption{Errors and CPU Time of Half-order Schemes.}
	\centering
	\begingroup
	\footnotesize
	\begin{tabular}{@{}|c|cccccc|@{}}
		\hline
		\noalign{\vspace{1.5pt}} 
		$h$ & $2^{-17}$ & $2^{-16}$ & $2^{-15}$ & $2^{-14}$ & $2^{-13}$ & $2^{-12}$  
		\\ 
		\hline
		$\|X_{n_h} - X_{n_h}^{h,E}\|_{L^2(\Omega)}$ &  $0.00599139$ & $0.01026297$ & $0.01565156$ & $0.02287485$ & $0.03247471$ & $0.04560142$
		\\ 
		CPU Time & $379.549801$ & $193.294414$ & $99.640363$ & $55.591084$ & $38.355161$ & $21.217252$
		\\
		\hline 
		$\|X_{n_h} - X_{n_h}^{h,RC}\|_{L^2(\Omega)}$ & $0.00254758$ & $0.00408114$ & $0.00646238$ & $0.00921385$ & $0.01328638$ & $0.01892801$   
		\\
		CPU Time & $833.698372$ & $421.096085$ & $213.971714$ & $110.782109$ & $59.234410$ & $32.016569$  
		\\
		\hline 
		$\|X_{n_h} - X_{n_h}^{h,MC}\|_{L^2(\Omega)}$ & $0.00241768$ & $0.00413479$ & $0.00636003$ & $0.00943673$ & $0.01331740$ & $0.01846128$  
		\\
		CPU Time & $871.231586$ & $438.688588$ & $221.942274$ & $111.028098$ & $59.387437$ & $30.414296$  
		\\
		\hline 
		$\|X_{n_h} - X_{n_h}^{h,R}\|_{L^2(\Omega)}$ & $0.00254793$ & $0.00408174$ & $0.00646433$ & $0.00921364$ & $0.01328279$ & $0.01889861$   
		\\
		CPU Time & $1239.085105$ & $627.439369$ & $327.828209$ & $182.257985$ & $141.139272$ & $68.163751$ 
		\\
		\hline 
		$\|X_{n_h} - X_{n_h}^{h,M}\|_{L^2(\Omega)}$ & $0.00241773$ & $0.00413333$ & $0.00635909$ & $0.00943438$ & $0.01331625$ & $0.01843175$  
		\\
		CPU Time & $1292.855992$ & $658.936470$ & $334.851178$ & $186.414961$ & $137.908831$ & $62.501852$ 
		\\
		\hline
	\end{tabular}
	
	\vspace{0.2cm} 
	
	\begin{tabular}{@{}|c|ccccccc|@{}}
		\hline
		\noalign{\vspace{1.5pt}} 
		$h$ & $2^{-11}$ & $2^{-10}$ & $2^{-9}$ & $2^{-8}$ & $2^{-7}$ & $2^{-6}$ & $2^{-5}$ 
		\\ 
		\hline
		$\|X_{n_h} - X_{n_h}^{h,E}\|_{L^2(\Omega)}$ & $0.06583836$ & $0.09392476$ & $0.13115728$ & $0.18428536$ & $0.26035255$ & $0.36210007$ & $0.53869722$ 
		\\ 
		CPU Time &  $8.644915$ & $4.310352$ & $2.178243$ & $1.241116$ & $0.467730$ & $0.328214$ & $0.122277$ 
		\\
		\hline 
		$\|X_{n_h} - X_{n_h}^{h,RC}\|_{L^2(\Omega)}$ & $0.02703500$ & $0.03779432$ & $0.05309653$ & $0.08012178$ & $0.12137365$ & $0.17662309$ & $0.28952555$ 
		\\
		CPU Time & $15.479152$ & $8.300158$ & $4.174976$ & $2.137459$ & $1.111581$ & $0.474894$ & $0.225357$ 
		\\
		\hline 
		$\|X_{n_h} - X_{n_h}^{h,MC}\|_{L^2(\Omega)}$ & $0.02701819$ & $0.03726813$ & $0.05561194$ & $0.08391667$ & $0.12436160$ & $0.20276714$ & $0.31159843$ 
		\\ 
		CPU Time & $16.398693$ & $8.485918$ & $4.510199$ & $2.204941$ & $1.218397$ & $0.578837$ & $0.362392$ 
		\\
		\hline 
		$\|X_{n_h} - X_{n_h}^{h,R}\|_{L^2(\Omega)}$ & $0.02696280$ & $0.03754892$ & $0.05253226$ & $0.07807442$ & $0.11588008$ & $0.16053614$ & $0.25189485$ 
		\\ 
		CPU Time & $27.278655$ & $13.447136$ & $7.119449$ & $3.486635$ & $1.515935$ & $0.755092$ & $0.403815$
		\\
		\hline 
		$\|X_{n_h} - X_{n_h}^{h,M}\|_{L^2(\Omega)}$ & $0.02695261$ & $0.03700100$ & $0.05474560$ & $0.08171425$ & $0.11913306$ & $0.19038216$ & $0.27754029$ 
		\\
		CPU Time  & $26.865898$ & $13.511409$ & $6.545521$ & $3.310891$ & $1.452824$ & $0.712814$ & $0.336374$ 
		\\
		\hline 
	\end{tabular}
	\label{tab:comphalf}
	\endgroup
\end{table}    

\FloatBarrier

\end{example} 

In the following examples, we demonstrate the results of Theorem \ref{thm:rate:wsi} and \ref{thm:rate:ri} for their non-randomized versions as explained in the Remark \ref{rem:Half}. The both examples \ref{ex:mrp} and \ref{ex:bs} illustrate that the non-randomized half-order modified and reduced schemes are more efficient than the Euler scheme for SDEwMS.
\begin{example}
	\label{ex:mrp}
	Consider the mean reverting process with regime switching from \cite[Chapter 10]{mao2006} defined as
	\begin{align*}
	dX(t) &= \lambda(r(t))(\mu(r(t))-X(t))dt + \sigma(r(t))X(t) dB(t)
	\end{align*}
	almost surely for any $ t \in [0,1]$ with initial values $S(0)=1$ and $r(0)=1$ where 
	\begin{align*}
	\lambda(\imath) =  
	\begin{cases}
	0.5,  & \mbox{ if } \imath=0, \\
	2, & \mbox{ if }  \imath=1,
	\end{cases}, \quad
	\mu(\imath) =  
	\begin{cases}
	2,   & \mbox{ if } \imath=0,  \\
	1, & \mbox{ if }  \imath=1, 
	\end{cases}  
	 \mbox{ and } \quad \sigma(\imath) =  
	\begin{cases}
	1,   & \mbox{ if } \imath=0,  \\
	0.5, & \mbox{ if }  \imath=1, 
	\end{cases} 
	\end{align*}
	for any $\imath \in S$.
	The state space $S$ and generator $Q$ are same as in the above example.

The scheme with a step size of $h = 2^{18}$ is treated as the reference solution for the non-randomized modified and reduced schemes, as well as the Euler scheme with $120000$ sample paths.
\begin{figure}[h]
	\centering
	\begin{subfigure}[]{\textwidth}  
		\centering
		\includegraphics[width=0.85\textwidth]{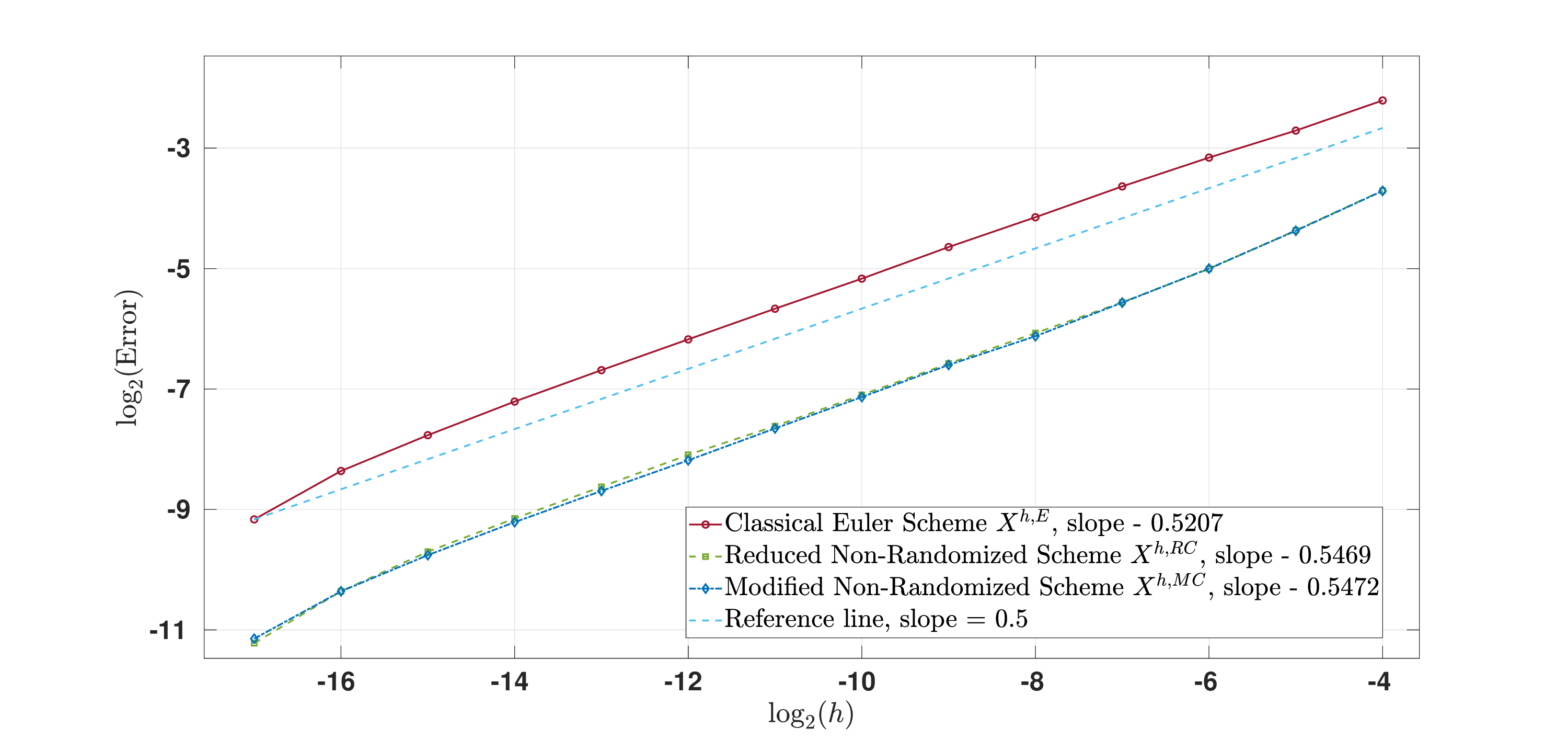}
		\caption{Error comparison}
		\label{fig:com:error:MRP}
	\end{subfigure}
\vspace{0.5cm}
	\begin{subfigure}[b]{\textwidth}  
		\centering
		\includegraphics[width=0.85\textwidth]{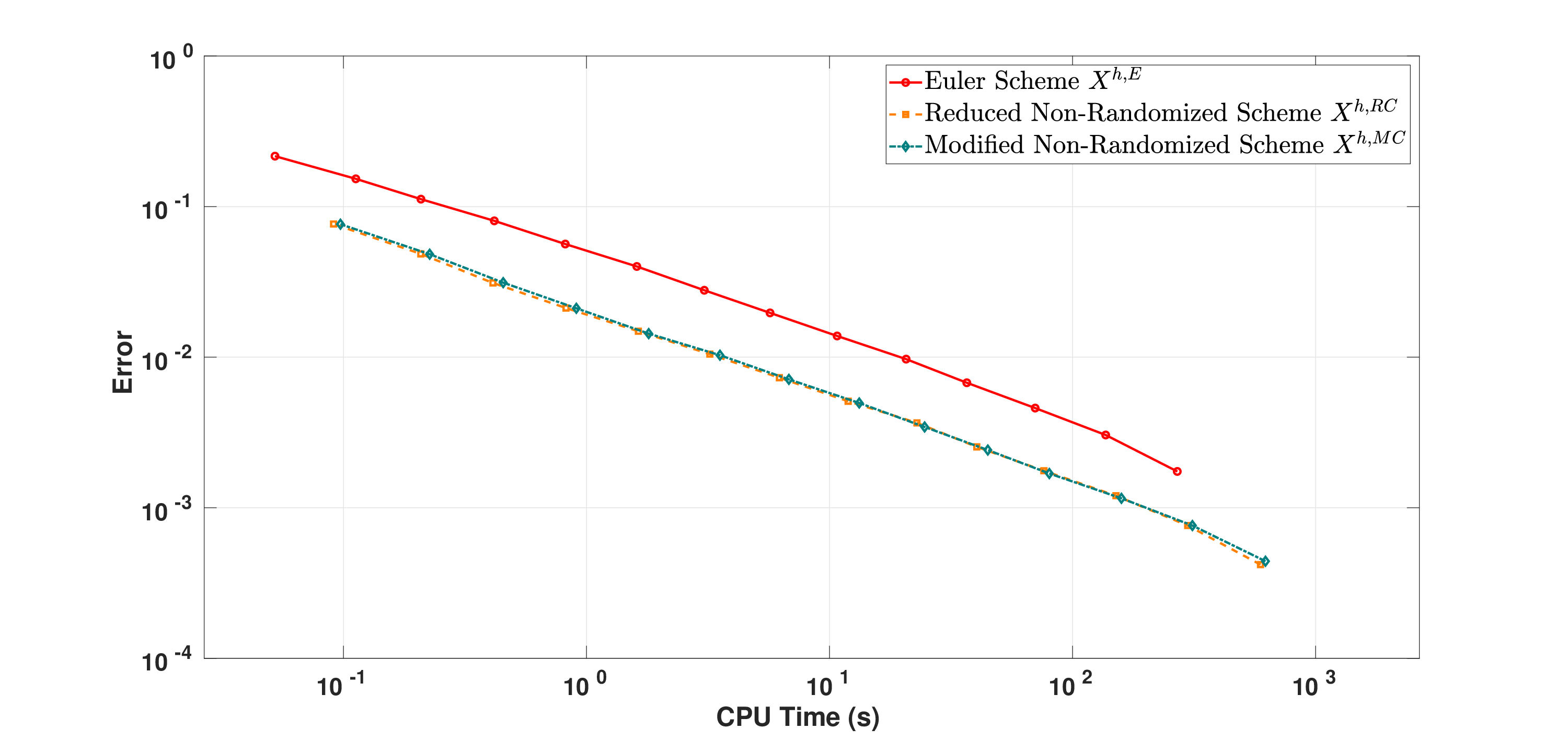}
		\caption{CPU Time comparison}
		\label{fig:com:cpu:MRP}
	\end{subfigure}
	
	\caption{Non-Randomized Scheme vs. Euler Scheme}
	\label{fig:comparison:MRP}
\end{figure}
\FloatBarrier

\begin{table}[h]
	\caption{Errors and CPU Time of Non-Randomized Scheme vs. Euler Scheme.}
	\centering
	\begingroup
	\footnotesize
	\begin{tabular}{@{}|c|ccccccc|@{}}
		\hline
		\noalign{\vspace{1.5pt}} 
		$h$ & $2^{-17}$ & $2^{-16}$ & $2^{-15}$ & $2^{-14}$ & $2^{-13}$ & $2^{-12}$  & $2^{-11}$
		\\ 
		\hline
		$\|X_{n_h} - X_{n_h}^{h,E}\|_{L^2(\Omega)}$ &  $0.00174194$ & $0.00304148$ & $0.00459130$ & $0.00676149$ & $0.00970590$ & $0.01383017$ & $0.01968373$
		\\ 
		CPU Time & $269.588664$ & $136.806453$ & $70.149249$ & $36.701544$ & $20.651390$ & $10.742751$ & $5.689832$
		\\
		\hline 
		$\|X_{n_h} - X_{n_h}^{h,RC}\|_{L^2(\Omega)}$ & $0.00041906$ & $0.00076148$ & $0.00120115$ & $0.00176037$ & $0.00253617$ & $0.00366260$ & $0.00509693$
		\\
		CPU Time & $592.890530$ & $297.713474$ & $150.758637$ & $76.182308$ & $40.325831$ & $22.881096$ & $11.925231$
		\\
		\hline 
		$\|X_{n_h} - X_{n_h}^{h,MC}\|_{L^2(\Omega)}$ & $0.00044127$ & $0.00076134$ & $0.00115479$ & $0.00168737$ & $0.00241640$ & $0.00343895$ & $0.00496271$ 
		\\
		CPU Time & $621.616301$ & $311.166626$ & $158.814760$ & $80.205689$ & $44.784252$ & $24.610021$ & $13.245413$
		\\
		\hline
	\end{tabular}
	
	\vspace{0.2cm} 
	
	\begin{tabular}{@{}|c|ccccccc|@{}}
		\hline
		\noalign{\vspace{1.5pt}} 
		$h$ & $2^{-10}$ & $2^{-9}$ & $2^{-8}$ & $2^{-7}$ & $2^{-6}$ & $2^{-5}$ & $2^{-4}$ 
		\\ 
		\hline
		$\|X_{n_h} - X_{n_h}^{h,E}\|_{L^2(\Omega)}$ & $0.02782465$ & $0.04003413$ & $0.05641012$  & $0.08044909$ & $0.11207857$ & $0.15294952$ & $0.21636128$ 
		\\ 
		CPU Time & $3.051900$ & $1.610595$ & $0.817714$ & $0.417378$ & $0.208293$ & $0.112406$ & $0.052256$
		\\
		\hline 
		$\|X_{n_h} - X_{n_h}^{h,RC}\|_{L^2(\Omega)}$ & $0.00730283$ & $0.01046996$ & $0.01486517$ & $0.02120103$ & $0.03112168$ & $0.04852725$ & $0.07664768$ 
		\\
		CPU Time & $6.215551$ & $3.216110$ & $1.634604$ & $0.823510$ & $0.411850$ & $0.207960$ & $0.091015$  
		\\
		\hline 
		$\|X_{n_h} - X_{n_h}^{h,MC}\|_{L^2(\Omega)}$ & $0.00712208$ & $0.01030526$ & $0.01433771$ & $0.02111886$ & $0.03125974$ & $0.04829711$ & $0.07641096$  
		\\ 
		CPU Time  & $6.801000$ & $3.539509$ & $1.801996$ & $0.907911$ & $0.454155$ & $0.226215$ & $0.097151$  
		\\
		\hline 
	\end{tabular}
	\label{tab:comphalfMRP}
	\endgroup
\end{table}    

\FloatBarrier 
\end{example} 
\begin{example}
	\label{ex:bs}
	Consider the geometric Brownian motion under regime-switching from \cite[Chapter 10]{mao2006} 
	\begin{align*}
	dX(t) = \mu(r(t))X(t)dt + \nu(r(t))X(t)dB(t)
	\end{align*}
		almost surely for any $ t \in [0,1]$ with initial values $S(0)=1$ and $r(0)=1$ where 
	\begin{align*}
	\mu(\imath) =  
	\begin{cases}
	0.5,   & \mbox{ if } \imath=0,  \\
	1, & \mbox{ if }  \imath=1, 
	\end{cases}  
	\mbox{ and } \quad \nu(\imath) =  
	\begin{cases}
	1.2,   & \mbox{ if } \imath=0,  \\
	0.6, & \mbox{ if }  \imath=1, 
	\end{cases} 
	\end{align*}
	for any $\imath \in S$.
	The state space $S$ and generator $Q$ are same as in the above example.

The scheme with a step size of $h = 2^{18}$ is treated as the reference solution for the non-randomized modified and reduced schemes, as well as the Euler scheme with $120000$ sample paths.

\begin{figure}[h]
	\centering
	\begin{subfigure}[]{\textwidth}  
		\centering
		\includegraphics[width=0.85\textwidth]{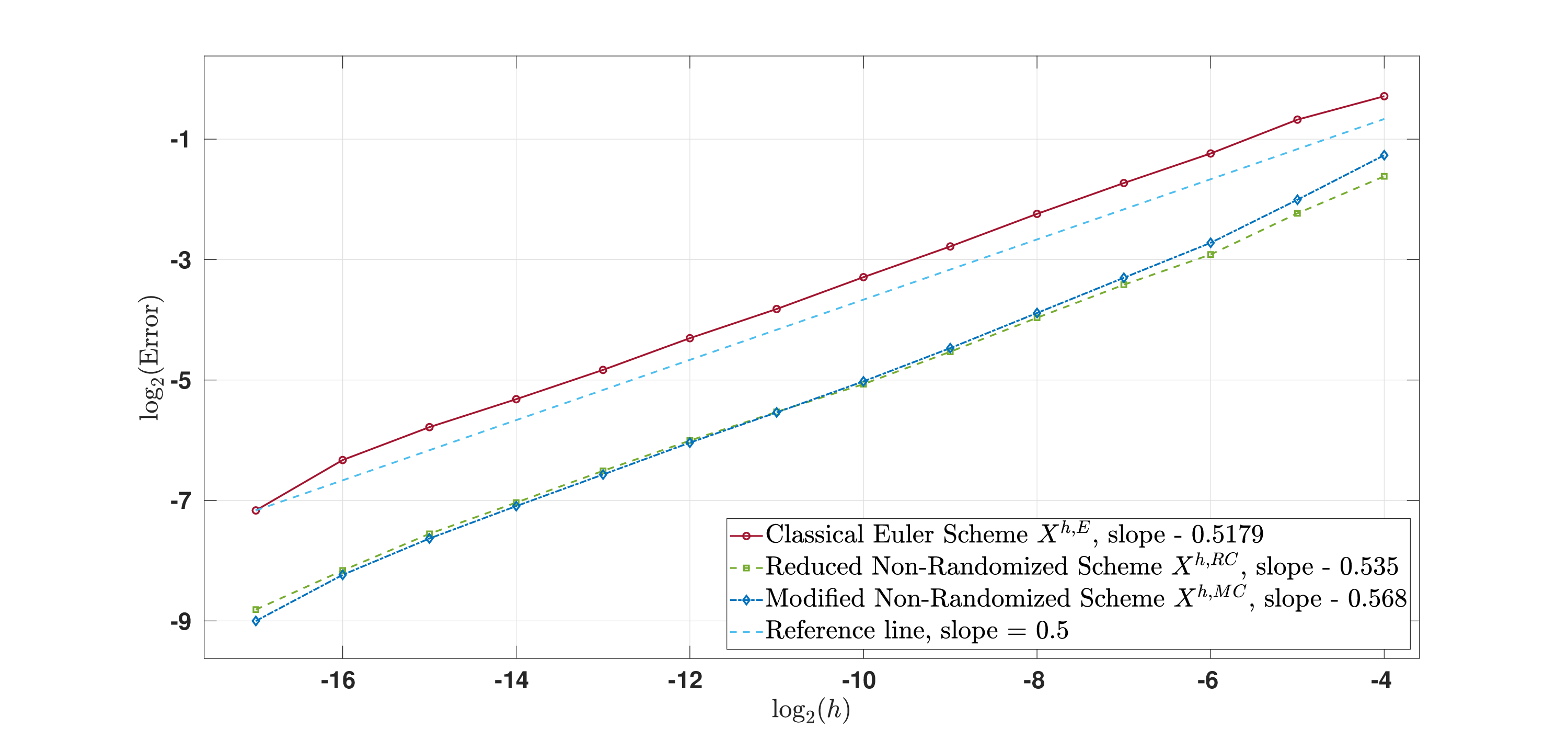}
		\caption{Error comparison}
		\label{fig:com:error:bs}
	\end{subfigure}
	\end{figure}
	\vspace{-3cm}
	\begin{figure}[h]
		\ContinuedFloat
	\begin{subfigure}[b]{\textwidth}  
		\centering
		\includegraphics[width=0.85\textwidth]{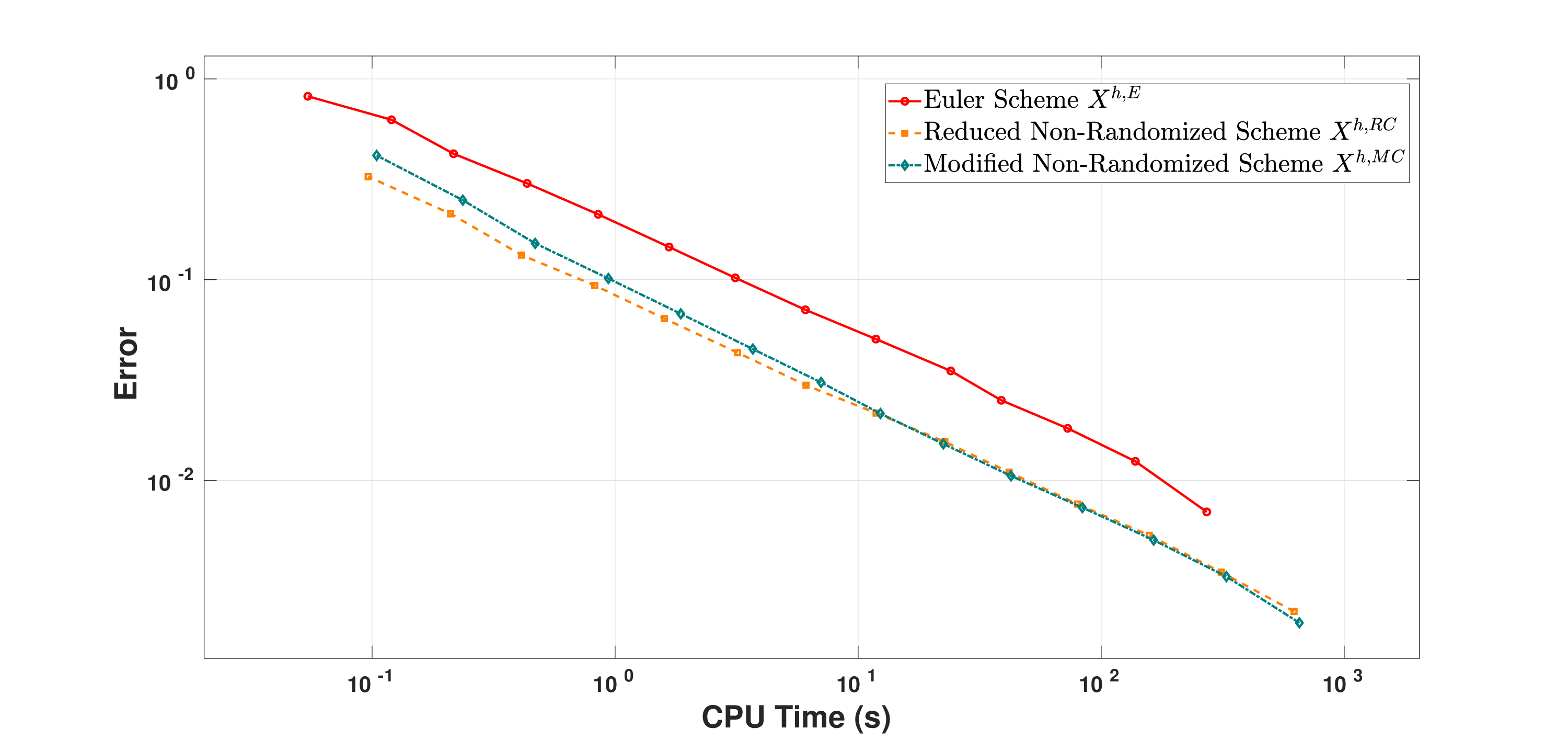}
		\caption{CPU Time comparison}
		\label{fig:com:cpu:bs}
	\end{subfigure}
	
	\caption{Non-Randomized Scheme vs. Euler Scheme}
	\label{fig:comparison:bs}
\end{figure}
\FloatBarrier
\begin{table}[h]
	\caption{Errors and CPU Time of Non-Randomized Scheme vs. Euler Scheme.}
	\centering
	\begingroup
	\footnotesize
	\begin{tabular}{@{}|c|ccccccc|@{}}
		\hline
		\noalign{\vspace{1.5pt}} 
		$h$ & $2^{-17}$ & $2^{-16}$ & $2^{-15}$ & $2^{-14}$ & $2^{-13}$ & $2^{-12}$  & $2^{-11}$
		\\ 
		\hline
		$\|X_{n_h} - X_{n_h}^{h,E}\|_{L^2(\Omega)}$ &  $0.00696620$ & $0.01244411$ & $0.01817137$ & $0.02506831$ & $0.03511297$ & $0.05058783$ & $0.07075726$ 
		\\ 
		CPU Time & $271.885095$ & $138.255123$ & $72.796728$ & $38.806524$ & $24.044686$ & $11.837339$ & $6.066299$ 
		\\
		\hline 
		$\|X_{n_h} - X_{n_h}^{h,RC}\|_{L^2(\Omega)}$ & $0.00222535$ & $0.00349048$ & $0.00532228$ & $0.00762568$ & $0.01098563$ & $0.01556446$ & $0.02164663$
		\\
		CPU Time & $621.345458$ & $311.865020$ & $157.700352$ & $79.881855$ & $41.785814$ & $22.727607$ & $11.845724$
		\\
		\hline 
		$\|X_{n_h} - X_{n_h}^{h,MC}\|_{L^2(\Omega)}$ & $0.00195249$ & $0.00331843$ & $0.00504374$ & $0.00732219$ & $0.01053223$ & $0.01519923$ & $0.02152456$  
		\\
		CPU Time & $653.652224$ & $327.488034$ & $164.415902$ & $83.680012$ & $42.555446$ & $22.411890$ & $12.350858$
		\\
		\hline
	\end{tabular}
	
	\vspace{0.2cm} 
	
	\begin{tabular}{@{}|c|ccccccc|@{}}
		\hline
		\noalign{\vspace{1.5pt}} 
		$h$ & $2^{-10}$ & $2^{-9}$ & $2^{-8}$ & $2^{-7}$ & $2^{-6}$ & $2^{-5}$ & $2^{-4}$ 
		\\ 
		\hline
		$\|X_{n_h} - X_{n_h}^{h,E}\|_{L^2(\Omega)}$ & $0.10210168$ & $0.14538977$ & $0.21155731$ & $0.30165312$ & $0.42392983$ & $0.62579877$ & $0.81950137$
		\\ 
		CPU Time & $3.122983$ & $1.667210$ & $0.851681$ & $0.434447$ & $0.216406$ & $0.120192$ & $0.054415$
		\\
		\hline 
		$\|X_{n_h} - X_{n_h}^{h,RC}\|_{L^2(\Omega)}$  & $0.02980105$ & $0.04327009$ & $0.06399915$ & $0.09358755$ & $0.13255435$ & $0.21287613$ & $0.32588021$ 
		\\
		CPU Time & $6.103634$ & $3.190577$ & $1.595359$ & $0.824581$ & $0.412425$ & $0.210510$ & $0.096383$  
		\\
		\hline 
		$\|X_{n_h} - X_{n_h}^{h,MC}\|_{L^2(\Omega)}$ & $0.03076803$ & $0.04511263$ & $0.06760361$ & $0.10139557$ & $0.15162888$ & $0.24887437$ & $0.41566770$ 
		\\ 
		CPU Time   & $7.034642$ & $3.687141$ & $1.862924$ & $0.938056$ & $0.468824$ & $0.236277$ & $0.104439$  
		\\
		\hline 
	\end{tabular}
	\label{tab:comphalfbs}
	\endgroup
\end{table} 
  
\FloatBarrier

\end{example}
Finally, as discussed in Remark \ref{rem:Half}, we demonstrate the derivative-free forms of non-randomized modified scheme $X^{DMC}$ and non-randomized reduced schemes $X^{DRC}$  alongside the Euler scheme $X^{E}$ in the following example. 
\begin{example}
	\label{ex:derivatefree}
	Consider the same model as discussed in Example \ref{ex:1} 
	and construct the derivative-free versions of the non-randomized modified and reduced schemes as discussed in Remark \ref{rem:Half}.
A step size of $h = 2^{18}$ is chosen as the true solution to compare the non-randomized derivative-free modified and reduced schemes with the Euler scheme, using $120000$ sample paths.
\begin{figure}[h]
	\centering
	\begin{subfigure}[]{\textwidth}  
		\centering
		\includegraphics[width=0.85\textwidth]{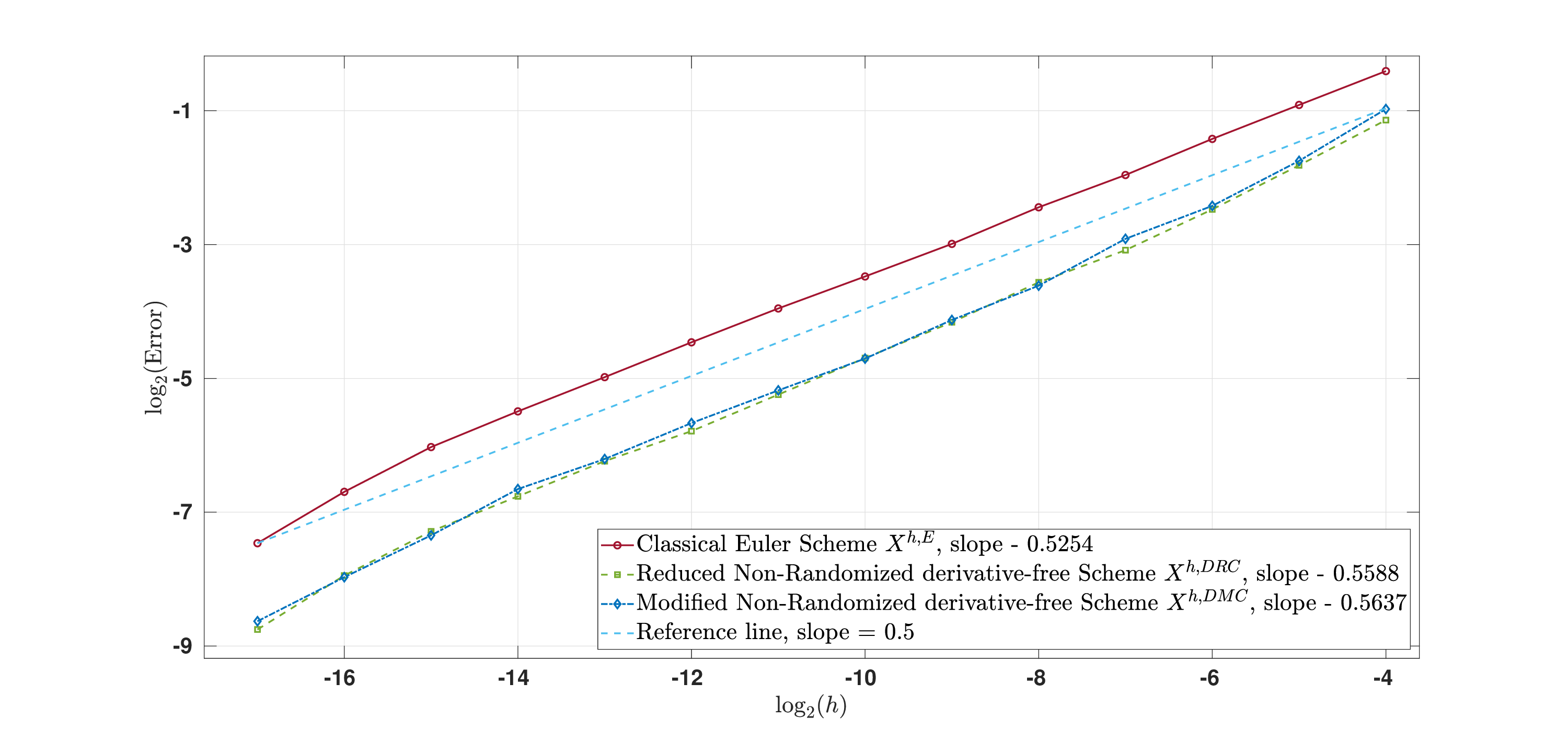}
		\caption{Error comparison}
		\label{fig:com:error:df}
	\end{subfigure}
\end{figure}
\begin{figure}[h]
	\ContinuedFloat
	\begin{subfigure}[b]{\textwidth}  
		\centering
		\includegraphics[width=0.85\textwidth]{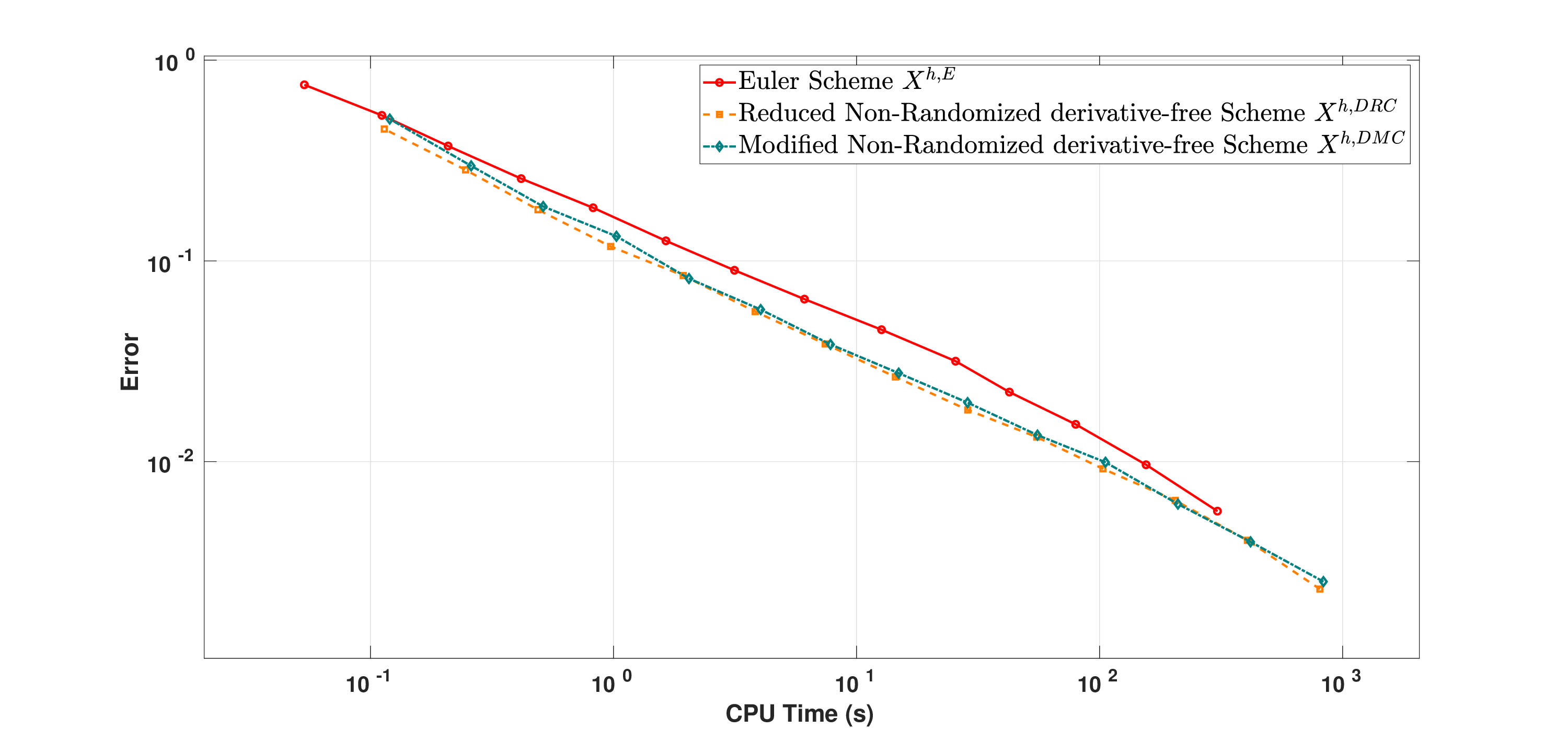}
		\caption{CPU Time comparison}
		\label{fig:com:cpu:df}
	\end{subfigure}
	
	\caption{Non-Randomized Scheme vs. Euler Scheme}
	\label{fig:comparison:df}
\end{figure}
\FloatBarrier
\begin{table}[h]
	\caption{Errors and CPU Time of Non-Randomized Scheme vs. Euler Scheme.}
	\centering
	\begingroup
	\footnotesize
	\begin{tabular}{@{}|c|ccccccc|@{}}
		\hline
		\noalign{\vspace{1.5pt}} 
		$h$ & $2^{-17}$ & $2^{-16}$ & $2^{-15}$ & $2^{-14}$ & $2^{-13}$ & $2^{-12}$  & $2^{-11}$
		\\ 
		\hline
		$\|X_{n_h} - X_{n_h}^{h,E}\|_{L^2(\Omega)}$ &  $0.00566716$ & $0.00964748$ & $0.01534008$ & $0.02220148$ & $0.03164935$ & $0.04540611$ & $0.06447661$ 
		\\ 
		CPU Time & $305.549772$	& $155.396309$	& $79.767932$	& $42.570644$ &	$25.566442$ &	$12.675962$ &	$6.106815$ 
		\\
		\hline 
		$\|X_{n_h} - X_{n_h}^{h,DRC}\|_{L^2(\Omega)}$ & $0.00231833$ & $0.00405278$ & $0.00640632$ & $0.00920944$ & $0.01325757$ & $0.01809466$ & $0.02641738$ 
		\\
		CPU Time & $806.553308$ &	$406.110562$ &	$204.820692$ &	$103.234321$ &	$55.153552$ &	$28.690261$ &	$14.448302$
		\\
		\hline 
		$\|X_{n_h} - X_{n_h}^{h,DMC}\|_{L^2(\Omega)}$ & $0.00252518$ & $0.00398567$ & $0.00614552$ & $0.00992798$ & $0.01355151$ & $0.01967248$ & $0.02759418$   
		\\
		CPU Time & $833.856323$ &	$418.168368$ &	$210.136879$ &	$105.777307$ &	$55.534373$ &	$28.644622$ &	$14.909527$ 
		\\
		\hline
	\end{tabular}
	
	\vspace{0.2cm} 
	
	\begin{tabular}{@{}|c|ccccccc|@{}}
		\hline
		\noalign{\vspace{1.5pt}} 
		$h$ & $2^{-10}$ & $2^{-9}$ & $2^{-8}$ & $2^{-7}$ & $2^{-6}$ & $2^{-5}$ & $2^{-4}$ 
		\\ 
		\hline
		$\|X_{n_h} - X_{n_h}^{h,E}\|_{L^2(\Omega)}$  & $0.08980111$ & $0.12584179$ & $0.18388249$ & $0.25691523$ & $0.37341561$ & $0.53082773$ & $0.75359352$
		\\ 
		CPU Time 	& $3.147398$ &	$1.643780$ &	$0.824708$ &	$0.417155$ &	$0.208715$ &	$0.111426$ &	$0.053475$
		\\
		\hline 
		$\|X_{n_h} - X_{n_h}^{h,DRC}\|_{L^2(\Omega)}$  & $0.02980105$ & $0.04327009$ & $0.06399915$ & $0.09358755$ & $0.13255435$ & $0.21287613$ & $0.32588021$ 
		\\
		CPU Time & $7.436777$ &	$3.831167$ &	$1.936506$ &	$0.973538$ &	$0.489165$ &	$0.246314$ &	$0.114055$  
		\\
		\hline 
		$\|X_{n_h} - X_{n_h}^{h,DMC}\|_{L^2(\Omega)}$ & $0.03835858$ & $0.05721629$ & $0.08165519$ & $0.13263672$ & $0.18662305$ & $0.29739362$ & $0.50869141$ 
		\\ 
		CPU Time   &	$7.813160$ &	$4.033923$ &	$2.045251$ &	$1.028350$ &	$0.513834$ &	$0.259656$ &	$0.120086$  
		\\
		\hline 
	\end{tabular}
	\label{tab:comphalfdf}
	\endgroup
\end{table}   
\FloatBarrier
\end{example}

\section*{Acknowledgement}
The second author gratefully acknowledges the financial support provided by National Board for Higher Mathematics (NBHM), Department of Atomic Energy (DAE), India under the project no. DAE-2188-MTD.


\section*{References}

\end{document}